\documentclass[11pt]{amsart}
\usepackage{amsmath, amssymb, amscd, mathrsfs, slashed, enumerate, url, stmaryrd}
\usepackage{color}
\usepackage{xcolor}
\usepackage[all,cmtip]{xy}
\usepackage{multicol}
\usepackage{indentfirst}
\usepackage{latexsym}
\usepackage{bm}
\usepackage{graphicx}
\usepackage{subfigure,esint}
\usepackage{float}
\usepackage{verbatim}
\usepackage{comment}
\usepackage[top=1in, bottom=1in, left=1.25in, right=1.25in]{geometry}
\usepackage{epsfig,dsfont,amsthm,amsfonts,amsbsy}
\usepackage[colorlinks]{hyperref}
\usepackage[toc,page]{appendix}
\usepackage{tikz-cd}
\usepackage{stmaryrd}

\newcommand{\oti}{\otimes}

\newcommand{\da}{\dagger}

\newcommand{\MH}{\mathcal{H}}
\newcommand{\MK}{\mathcal{K}}

\newcommand{\MM}{\mathcal{M}}
\newcommand{\MO}{\mathcal{O}}

\newcommand{\CS}{\mathbb{C}^{\st}}

\newcommand{\MW}{\mathcal{W}}
\newtheorem{theorem}{Theorem}[section]

\newtheorem{corollary}[theorem]{Corollary}

\newtheorem{definition}[theorem]{Definition}

\newtheorem{example}[theorem]{Example}

\newtheorem{lemma}[theorem]{Lemma}

\newtheorem{proposition}[theorem]{Proposition}

\renewcommand{\Im}{\mathrm{Im}}

\newcommand{\MC}{\mathcal{C}}
\newcommand{\Tr}{\mathrm{Tr}}

\newcommand{\st}{\ast}

\newcommand{\pa}{\partial}

\newcommand{\ti}{\times}

\newcommand{\vp}{\varphi}

\newcommand{\ME}{\mathcal{E}}
\newcommand{\al}{\alpha}

\newcommand{\na}{\nabla}
\newcommand{\ep}{\epsilon}

\newcommand{\bpa}{\bar{\pa}}

\newcommand{\be}{\beta}

\newcommand{\lam}{\lambda}

\newcommand{\mft}{\mathfrak{t}}
\newcommand{\mfe}{\mathfrak{e}}

\newcommand{\MP}{\mathcal{P}}

\newcommand{\diag}{\mathrm{diag}}

\newcommand{\MG}{\mathcal{G}}
\newcommand{\mfm}{\mathfrak{m}}

\newcommand{\su}{\mathfrak{su}}

\newcommand{\MF}{\mathcal{F}}

\newcommand{\Lam}{\Lambda}

\newcommand{\MV}{\mathcal{V}}

\newcommand{\MMH}{\MM_{\mathrm{Hit}}}

\newcommand{\loc}{\mathrm{loc}}
\newcommand{\LC}{\mathrm{Lim}}
\newcommand{\tL}{\widetilde{L}}

\newcommand{\MB}{\mathcal{B}}
\newcommand{\mbM}{\mathbb{M}}
\newcommand{\MI}{\mathcal{I}}

\newcommand{\Dol}{\mathrm{Dol}}
\newcommand{\MD}{\mathcal{D}}
\newcommand{\BMP}{\overline{\MP}}
\newcommand{\SL}{\mathrm{SL}}
\newcommand{\MMHQLC}{\MM_{\mathrm{Hit},q}^{\LC}}

\newcommand{\Fix}{\mathrm{Fix}}

\newcommand{\MSD}{\mathscr{D}}

\newcommand{\MSF}{\mathscr{F}}

\newcommand{\MSN}{\mathscr{N}}

\newcommand{\MSP}{\mathscr{P}}

\newcommand{\MSV}{\mathscr{V}}
\newcommand{\MSW}{\mathscr{W}}

\newcommand{\rank}{\mathrm{rank}}

\newcommand{\tJ}{\widetilde{J}}
\newcommand{\mss}{\mathrm{ss}}
\newcommand{\Tot}{\mathrm{Tot}}

\newcommand{\Aut}{\mathrm{Aut}}

\newcommand{\hvp}{\hat{\vp}}

\newcommand{\id}{\mathrm{id}}
\newcommand{\MMD}{\MM_{\mathrm{Dol}}}
\newcommand{\BMMD}{\overline{\MM}_{\mathrm{Dol}}}
\newcommand{\BMMDq}{\overline{\MM}_{\mathrm{Dol},q}}
\newcommand{\BMMH}{\overline{\MM}_{\mathrm{Hit}}}
\newcommand{\BMMHq}{\overline{\MM}_{\mathrm{Hit},q}}
\newcommand{\BXi}{\overline{\Xi}}
\newcommand{\BMH}{\overline{\MH}}
\newcommand{\CP}{\mathbb{C}P}
\newcommand{\PMB}{\overline{\MB}}
\newcommand{\Lim}{\mathrm{Lim}}
\newcommand{\SLC}{\mathrm{SL}(2,\mathbb{C})}

\newcommand{\Tor}{\mathrm{Tor}}

\newcommand{\sing}{\mathrm{sing}}

\newcommand{\MGC}{\mathcal{G}_{\mathbb{C}}}
\newcommand{\tME}{\widetilde{\ME}}
\newcommand{\tvp}{\tilde{\vp}}

\newcommand{\BJac}{\overline{\mathrm{Jac}}}
\newcommand{\BPic}{\overline{\mathrm{Pic}}}
\newcommand{\Jac}{\mathrm{Jac}}
\newcommand{\tS}{\widetilde{S}}
\newcommand{\tMO}{\widetilde{\MO}}
\newcommand{\hME}{\widehat{\ME}}
\newcommand{\hq}{\hat{q}}

\newcommand{\mbC}{\mathbb{C}}

\newcommand{\om}{\omega}

\newcommand{\Div}{\mathrm{Div}}
\newcommand{\ML}{\mathcal{L}}
\newcommand{\MT}{\mathcal{T}}
\newcommand{\TheMoc}{\Theta^{\mathrm{Moc}}}
\newcommand{\MNM}{\mathrm{Nm}}
\newcommand{\hPMod}{\widehat{\mathrm{PMod}}(\tS)}

\newcommand{\tK}{\tilde{K}}
\newcommand{\tpi}{\tilde{\pi}}
\newcommand{\Pic}{\mathrm{Pic}}
\newcommand{\Prym}{\mathrm{Prym}}
\newcommand{\Nm}{\mathrm{Nm}}
\newcommand{\PMod}{\mathrm{PMod}(\tS)}
\newcommand{\supp}{\mathrm{supp\;}}

\newcommand{\ord}{\mathrm{ord}}
\newcommand{\BMT}{\overline{\MT}}

\newcommand{\BNR}{\mathrm{BNR}}
\newcommand{\UpMoc}{\Upsilon^{\mathrm{Moc}}}
\newcommand{\MMHLC}{\MMH^{\LC}}
\newcommand{\reg}{\mathrm{reg}}
\newcommand{\BMMDr}{\BMMD^{\reg}}
\newcommand{\BMMDs}{\BMMD^{\sing}}
\newcommand{\BMMHr}{\BMMH^{\reg}}
\newcommand{\BMMHs}{\BMMH^{\sing}}

\newcommand{\MBV}{\mathbb{V}}
\newcommand{\MBW}{\mathbb{W}}

\newcommand{\app}{\mathrm{app}}
\newcommand{\ptf}{p^{\star}_{\mathrm{tf}}}

\newcommand{\bcup}{\bigcup}

\newcommand{\pr}{\mathrm{pr}}

\newcommand{\sta}{\mathrm{st}}
\newcommand{\Gr}{\mathrm{Gr}}
\newcommand{\Ze}{Z_{\mathrm{even}}}
\newcommand{\Zo}{Z_{\mathrm{odd}}}
\newcommand{\tMT}{\widetilde{\MT}}
\newcommand{\Hit}{\mathrm{Hit}}

\newcommand{\tlh}{\tilde{h}}
\newcommand{\tp}{\tilde{p}}
\newcommand{\tq}{\tilde{q}}
\newcommand{\tx}{\tilde{x}}

\newcommand{\tR}{\widetilde{R}}

\newcommand{\tlam}{\tilde{\lam}}

\newcommand{\tH}{\widetilde{H}}
\newcommand{\tZ}{\widetilde{Z}}

\newcommand{\tMP}{\widetilde{\MP}}

\newcommand{\mbfH}{\mathbf{H}}

\newcommand{\mbfc}{\mathbf{c}}

\newcommand{\lra}{\longrightarrow}

\newcommand{\isorightarrow}{\xrightarrow{
       \,\smash{\raisebox{-0.5ex}{\ensuremath{\sim}}}\,}}

\usepackage[utf8]{inputenc}
\usepackage[english]{babel}
\usepackage{fancyhdr}

\author[He]{Siqi He}

\address{Morningside Center of Mathematics,
  Chinese Academy of Sciences, 
   Beijing, 100190 China}
\email{sqhe@amss.ac.cn}

\author[Mazzeo]{Rafe Mazzeo}

\address{
Department of Mathematics, 
Stanford University, 
Stanford, CA, 94305
}
\email{rmazzeo@stanford.edu}

\author[Na]{Xuesen Na}

\address{Department of Mathematics,
   University of Illinois,
   Champaign, IL 61820, USA}
\email{xna@illinois.edu}

\author[Wentworth]{Richard A. Wentworth}

\address{Department of Mathematics,
   University of Maryland,
   College Park, MD 20742, USA}
\email{raw@umd.edu}
   
\begin{document}
	\title[Compactifications of the Hitchin Moduli Space]{The Algebraic and Analytic 
    Compactifications\\  of the Hitchin Moduli Space}
	
\begin{abstract}
Following the work of Mazzeo-Swoboda-Weiss-Witt \cite{mazzeo2016ends} and
    Mochizuki \cite{Mochizukiasymptotic}, there is a map $\BXi$ between the
    algebraic compactification of the Dolbeault moduli space of $\SLC$
    Higgs bundles on a smooth projective curve coming from the
    $\mathbb{C}^\ast$ action,  and the
    analytic compactification of Hitchin's moduli space of solutions to the  $\mathsf{SU}(2)$
    self-duality equations on a Riemann surface obtained by adding solutions to the decoupled equations, known as ``limiting
     configurations''. This map extends the classical Kobayashi-Hitchin correspondence. The main result of this paper is that $\BXi$ fails to be  continuous at the boundary over a certain subset of 
    the discriminant locus of the Hitchin fibration.
    This suggests  the possibility of a third, refined compactification
    which dominates both. 
\end{abstract}
	\maketitle
 
\section{Introduction}	
Let $\Sigma$ be a closed Riemann surface of genus $g\geq 2$. 
The coarse Dolbeault moduli space of $\SLC$
    semistable Higgs bundles on $\Sigma$, denoted by $\MMD$, 
    and Hitchin's moduli space of solutions to the  $\mathsf{SU}(2)$
    self-duality equations on $\Sigma$, denoted by $\MMH$, 
    have been extensively studied since their introduction over 35 years ago. 
    The Kobayashi-Hitchin correspondence, proved in \cite{hitchin1987self},
    gives a homeomorphism between these two moduli spaces:
$$\Xi:\MMD\isorightarrow \MMH\ .$$
The space $\MMD$ is naturally a quasiprojective variety
\cite{Nitsure:91, simpson1994moduli}, and similarly to  monopole moduli
spaces, $\MMH$ fails to be compact. 
Recently, there has been interest from several directions on
natural  compactifications of these two spaces. 
A key feature on the Dolbeault side
is the existence of a $\mathbb{C}^\ast$ action with the Bia{\l}ynicki-Birula
property, and  this may be used to define a  completion
of $\MMD$ as a projective variety
\cite{hausel1998compactification,de2018compactification,fan2022analytic}. 
The ideal points  are identified with   
the $\mathbb{C}^*$ orbits in the complement of the nilpotent cone of $\MMD$. 
The Hitchin moduli space also  admits  a more recently introduced compactification, $\BMMH$, 
based on the work of several authors (see \cite{mazzeo2016ends, Mochizukiasymptotic, 
Taubes20133manifoldcompactness}). 
The boundary of $\BMMH$ is given by gauge equivalence classes of \emph{limiting
configurations}. This compactification is relevant to many aspects of
Hitchin's moduli space.  For more details, we refer the reader to 
\cite{DumasNeitzke:19,mazzeo2012limiting,fredrickson2020exponential,fredrickson2022asymptotic,ott2020higgs,katzarkov2015harmonic,chen2022asymptotic},
and the references therein.

By the work of \cite{mazzeo2016ends,Mochizukiasymptotic}, there is a
natural extension 
$$\BXi:\BMMD\lra \BMMH$$
of the Kobayashi-Hitchin correspondence to the two compactifications
described above, and
it is of interest to study the geometry of this map. 
This involves another key feature of Hitchin's moduli space; namely,
spectral curves.
Spectral curves and spectral data \cite{hitchin1992lie} play a central
role in the realization of the Dolbeault moduli space  as an algebraically complete integrable
system $\mathcal{H}: \MMD\to \mathcal{B}$. 
In the case of $\SLC$, the base $\mathcal B$  is the space of
holomorphic quadratic differentials  on $\Sigma$. Given
$q\in H^0(K^2)$, one obtains a (scheme theoretic) spectral curve $S_q$. This curve is 
reduced if $q\neq 0$, irreducible if $q$ is not  the square of an abelian
differential, and
smooth if  $q$ has simple zeros. We let $\mathcal{B}^{\rm
reg}\subset \mathcal{B}$ denote  the open cone of quadratic differentials
with simple zeros. 

The ideal points of both compactifications $\BMMD$ and $\BMMH$ have
associated nonzero quadratic differentials, and therefore spectral curves. 
We write $\BMMDr$ for the elements in $\BMMD$ with smooth spectral curves,
and $\BMMDs=\BMMD\setminus \BMMDr$ for those with singular spectral curves; 
similarly for $\BMMHr$ and $\BMMHs$. We then  have the following result.
\begin{theorem}
	The restriction of the compactified Kobayashi-Hitchin map
    $\BXi:\BMMD\to \BMMH$ to the locus with  smooth associated spectral
    curves
    defines a homeomorphism $\BMMD^{\reg}\simeq \BMMH^{\reg}$. 
    On the singular spectral curve locus, however,
    $\BXi^{\sing}:\BMMD^{\sing}\to \BMMH^{\sing}$ is neither surjective nor injective.
\end{theorem}

It is convenient to analyze the behavior along rays in $\mathcal{B}$, where
the spectral curve is simply rescaled.
Let $q\neq 0$ be a quadratic differential and $\BMMDq$ (resp.\ $\BMMHq$) be the
points in  $\BMMD$ (resp.\  $\BMMH$) with spectral curves $S_{tq}$, $t\in
\mathbb{C}^\ast$.
The restriction of $\BXi$ gives us the map $\BXi_q:\BMMDq\to \BMMHq$. 
We  shall study the continuous behavior of $\BXi_{q}$ for points in the
fiber of $tq$ as $t\to \infty$.
For convenience, we set 
$\MM_{q^\ast}:=\BMMDq\cap \MMD$.
 When $q$ is irreducible, i.e.\ not a square,  all elements in $\MM_q$ are stable. 
Via the Hitchin \cite{hitchin1987stable} and Beauville-Narasimhan-Ramanan (BNR) correspondence \cite{bnr1989spectral}, this reduces the
description of the fiber  $\mathcal{H}^{-1}(q)$ 
to the characterization of rank $1$ torsion free sheaves on the integral
curve $S_q$. 

In \cite{rego1980compactified}, parameter spaces for  rank $1$ torsion free
sheaves on algebraic curves with Gorenstein singularities were studied in
the context of compactified Jacobians, and the crucial notion of a
parabolic module was introduced. This was extensively investigated by Cook in 
\cite{cook1993local,cook1998compactified}, partially following ideas of
Bhosle \cite{Bhosle:92}. For simple plane curve singularities of the type appearing
in spectral curves, one makes use of the local classification of torsion free
modules of Greuel-Kn\"orrer \cite{Greuel1985}. These methods were applied
to study the Hitchin fibration by Gothen-Oliveira in
\cite{gothen2013singular} (see also \cite{kydonakis2022monodromy} for
recent study).  In parallel, Horn \cite{horn2022semi} defines a stratification of
$\MM_q=\bigcup_D\MM_{q,D}$  by certain effective 
divisors contained in the divisor of $q$ (see Section
\ref{sec:divisor-stratification},
and  also \cite{hornna2022geometry} for
the more general situation). 
Using the results from these references, we have a  reinterpretation of Mochizuki's construction
\cite{Mochizukiasymptotic}.  This leads to the following result.

\begin{theorem}
	Let $q\neq0$ be an irreducible quadratic differential. 
	\begin{enumerate}
		\item 	If $q$ has only zeros of odd order, then $\BXi_q$ is
            continuous.
		\item  If $q$ has at least one zero of even order,  then for each $D\neq 0$
            there exists an integer $n_D\geq 1$ so that for any Higgs bundle
            $(\MF,\psi)\in \MM_{q,D}$, there exist $n_D$ sequences of
            Higgs bundles $(\ME_i^k,\vp_i^k)$ with $k=1,\ldots, n_D$
            satisfying
		\begin{itemize}
			\item $\lim_{i\to \infty}(\ME_i^k,\vp_i^k)=(\MF,\psi)$ for $k=1,\ldots,n_D$,
			\item  
                and if we write $$\eta^k:=\lim_{i\to
                \infty}\BXi_q(\ME_i^k,\vp_i^k)\quad ,\quad \xi:=\lim_{i\to
                \infty}\BXi_q(\MF,t_i\psi)\ ,$$  for some  sequence
                $t_i\in\mathbb{R}^+$,  $t_i\to+\infty$, then $\xi,\eta^1,\ldots,\eta^{n_D}$ are $n_D+1$
                different limiting configurations.
	\end{itemize}
		\end{enumerate}
\end{theorem}

When $q$ is reducible, the description of Higgs bundles in   $\MM_q$
becomes more complicated because, among other things,  of the existence of
strictly semistable objects.
To understand this, we use the local descriptions of Gothen-Oliveira
and Mochizuki (see \cite{gothen2013singular,Mochizukiasymptotic}).
Our result, which focuses on the  stable locus, is the following.
\begin{theorem}
	Suppose $q\neq 0$ is reducible. Let $\BMMDq^{\sta}$ denote the stable locus
    of $\BMMDq$. If $g\geq 3$, then the restriction map
    $\BXi_q|_{\BMMDq^{\sta}}$ is discontinuous. However, if $g=2$, the map
    $\BXi_q|_{\BMMDq^{\sta}}$ is continuous.
\end{theorem}

We also note that there is recent work of Mochizuki and Szab\'{o}
\cite{mochizuki2023asymptotic} on the asymptotic behavior for families of
Higgs bundles in  the higher rank case.

This paper is organized as follows: 
in Section \ref{sec_background_Higgs_bundles}, we provide an overview of Higgs bundles and BNR correspondence. 
In Section \ref{sec_filtered_bundles_and_compactness}, we delve into the concepts of 
filtered bundles and their compactness properties. 
Section \ref{sec_the_algebraic_and_analytic_compactifications} defines 
the algebraic and analytic compactifications. 
Section \ref{sec_parabolic_modules} 
introduces parabolic modules and examines their connection to spectral curves. 
The main results for Hitchin fibers with irreducible singular spectral curves are established 
in Section \ref{sec_irreducible_singular_fiber}. In Section \ref{sec_reducible_singular_fiber}, 
the results for the reducible case are proven. 
Finally, in Section \ref{sec_compactified_Kobayashi_Hitchin_map}, 
we construct the compactified Kobayashi-Hitchin map and prove the main results.
The Appendix, based on the work of Greuel-Kn\"orrer, calculates some invariants of rank $1$ torsion free 
sheaves on the spectral curves we consider.

\textbf{Acknowledgements.} We extend our sincere gratitude to Takur\={o} Mochizuki for his valuable insights and stimulating
discussions during the BIRS conference in 2021. 
The authors also wish to express their gratitude to a great many people for their interest and helpful comments. 
Among them are Mark de Cataldo, Ron Donagi, Simon Donaldson, Laura
Fredrickson, Johannes Horn, Laura Schaposnik, Shizhang Li, Jie Liu, Tony
Pantev, Thomas Walpuski, Daxin Xu. S.H is supported by NSFC grant
No.12288201. R.W.'s research is supported by NSF grants DMS-1906403 and
DMS-2204346, and he  thanks the Max Planck Institute for Mathematics in
Bonn for its hospitality and financial support. 
The authors also gratefully acknowledge the support  of the MSRI
under NSF grant DMS-1928930 during the program ``Analytic  and Geometric Aspects of Gauge Theory'', Fall 2022.

\section{Background on Higgs bundles}
\label{sec_background_Higgs_bundles}
This section gives a very brief overview of 
the Dolbeault and Hitchin moduli spaces, spectral curve descriptions,
and the nonabelian Hodge correspondence. For  more details on  these
topics, see  \cite{hitchin1987self,
hitchin1987stable, Simpson1992, Wentworth2016}.
\subsection{Higgs bundles}
As in the Introduction, throughout  this paper
 $\Sigma$ will denote a closed Riemann surface of genus $g\geq 2$ with
 structure sheaf $\MO=\MO_\Sigma$ and 
 canonical bundle $K=K_\Sigma$. Let $E\to \Sigma$ be a complex vector bundle. 
  A Higgs bundle consists of a pair $(\mathcal{E},\varphi)$, where
  $\mathcal{E}$ is a holomorphic bundle and $\varphi\in
  H^0(\mathrm{End}(\mathcal{E})\otimes K)$. 
 If $\rank(\mathcal{E})=1$, then a Higgs field is just an abelian
 differential $\omega$.  
  The pair $(\mathcal{E},\varphi)$ is called  
   an $\mathrm{SL}(2,\mathbb{C})$ Higgs bundle if $\mathrm{rank}(E)=2$, 
   $\det(\mathcal{E})$ has a fixed isomorphism with the  trivial bundle, and $\mathrm{Tr}(\varphi)=0$. In this
   paper we will focus mainly on $\mathrm{SL}(2,\mathbb{C})$ Higgs bundles,
   but the rank $1$ case will also be important.  

Let $(\mathcal{E},\varphi)$ be an $\mathrm{SL}(2,\mathbb{C})$ Higgs bundle.
A (proper) Higgs subbundle of $(\mathcal{E},\varphi)$ is a holomorphic line
bundle $\mathcal{L}\subset \mathcal{E}$ that  is
$\varphi$-invariant, i.e.\ $\varphi:\mathcal{L}\to \mathcal{L}\otimes K$.
In this case, the restriction
$\varphi_{\mathcal{L}}:=\varphi\bigr|_{\mathcal L}$,  makes 
$(\mathcal{L},\varphi_{\mathcal{L}})$  a  rank $1$ Higgs bundle. 
Moreover, $\varphi$ induces a Higgs bundle structure on the quotient
$\mathcal{E}/\mathcal{L}$.
We say $(\ME,\vp)$ is stable (resp.\ semistable) if for all Higgs
subbundles $\ML$, $\deg \ML<0$ (resp.\ $\deg \ML\leq 0$). We say  $(\ME,\vp)$ is polystable
if  $(\ME,\vp)\simeq(\ML,\om)\oplus (\ML^{-1},-\om)$, where $\ML$ is a
degree zero holomorphic line bundle and $\om \in H^0(K)$. 

If $(\ME,\vp)$ is strictly semistable, i.e.\ semistable but not stable,  
the Seshadri filtration \cite{seshadri1967space} gives a unique Higgs
subbundle $0\subset (\ML,\om)\subset (\ME,\vp)$ with
$\deg(\ML)=\frac{1}{2}\deg(\ME)=0$. Write $(\ML',\om'):=(\ME,\vp)/(\ML,\om)$, then we have $\om'=-\om$ and $\ML'=\ML^{-1}$.  
The associated graded bundle $\Gr(\ME,\vp)=(\ML,\om)\oplus (\ML^{-1},-\om)$ of this filtration 
 is a polystable $\SLC$ Higgs bundle. We say that  $(\ME,\vp)$ is
 S-equivalent to  $\Gr(\ME,\vp)$.

Holomorphic bundles $\mathcal{E}$ with underlying $C^\infty$ bundle $E$
are in 1-1 correspondence with $\bar\partial$-operators $\bar\partial_E:
\Omega^0(E)\to \Omega^{0,1}(E)$. We use the
notation  $\mathcal{E}:=(E,\bar{\partial}_E)$. 
Let $\MC$ denote the space of pairs $ (\bar\partial_E,\varphi)$,
$\bar\partial_E\varphi=0$.
Let $\MC^s$ and $\MC^{ss}$ denote the subspaces of $\MC$ where the Higgs
bundles are stable (resp.\ semistable). 
The complex gauge transformation group $\MGC:=\Aut(E)$ has a right action
on $\MC$ by defining for $g\in\MGC$,
$(\bar{\pa}_E,\vp)g:=(g^{-1}\circ\bar{\pa}\circ g,g^{-1}\circ\vp\circ g)$.

There is a quasiprojective scheme $\MMD$ whose closed points are in 1-1
correspondence with polystable Higgs bundles constructed via (finite dimensional) Geometric
Invariant Theory (see \cite{Nitsure:91,
simpson1994moduli}). In \cite{Fan:22} it was shown that the infinite
dimensional quotient   $\MC^{\mss}\sslash\MGC$, 
where the double slash indicates that S-equivalent orbits are identified,
admits the structure of a complex analytic space that is 
biholomorphic to the analytification $\MMD^{\rm an}$ of $\MMD$. 
Henceforth, we shall work in the complex analytic category, identify the
algebro-geometric and gauge theoretic moduli spaces,  and  simply denote
them both by  $\MMD$.
We note that 
 the set of stable Higgs bundles modulo gauge transformations,
 $\MMD^s:=\MC^{\mathrm{s}}\sslash\MGC$,  is  a geometric quotient and an open subset of $\MMD$.

 Finally, notice that the a pair $(\mathcal{E},\varphi)$ is stable (resp.\
 semistable) if and only if the same is true for
 $(\mathcal{E},\lambda\varphi)$, $\lambda\in \mathbb{C}^\ast$. Hence,
 $\MMD$ admits an action of $\mathbb{C}^\ast$ that preserves $\MMD^{s}$.
 Though $\MMD$ is only
 quasiprojective, the $\mathbb{C}^\ast$ action satisfies the
 Bia{\l}ynicki-Birula property:  

 \begin{theorem} [{\cite{hitchin1987self,Simpson1992}}]
     \label{thm:proper-action}
     For any $[(\mathcal{E},\varphi)]\in \MMD$, 
     $$
     \lim_{\lambda\to 0}
     \lambda\cdot[(\mathcal{E},\varphi)]:=\lim_{\lambda\to 0}[(\mathcal{E},\lambda\varphi)]
     $$
     exists in $\MMD$.
 \end{theorem}

\subsection{Spectral curves and the Hitchin fibration}
The Hitchin map 
is defined as
\begin{equation*}
	\MH: \MMD \longrightarrow H^0(K^2)\ , \  [(\ME,\vp)] \mapsto \det(\vp)\ ,
\end{equation*}
where $H^0(K^2)=:\MB$ is known as the Hitchin base. 
Hitchin \cite{hitchin1987self,hitchin1987stable} showed that $\MH$ is a
proper map and a fibration by abelian varieties
over the open cone  $\MB^{\rm reg}\subset\MB$
consisting of nonzero  quadratic differentials with only simple zeros.
The discriminant
locus   $\MB^{\rm sing}:=\MB\setminus \MB^{\rm reg}$ consists of quadratic differentials that are
either identically zero or have at least one zero with multiplicity. 
For $q\in \mathcal{B}$, let  
$\mathcal{M}_q:=\mathcal{H}^{-1}(q)$. The ``most singular fiber''
$\mathcal{M}_0$ is called  the \emph{nilpotent cone}.

Consider the total space $\Tot(K)$ of $K$,  
along with its
projection  $\pi: \Tot(K) \to \Sigma$. The pullback bundle $\pi^* K$ has 
a tautological section, which we denote by $\lambda \in
H^0(\Tot(K),\pi^\ast K)$. Given any $q\neq 0 \in H^0(K^2)$, 
the \emph{spectral curve} $S_q$ associated with $q$ is the 
zero scheme of the section $\lambda^2 - \pi^\ast q \in H^0(\Tot(K),\pi^* K)$. 
This is a reduced, but possibly reducible, projective algebraic curve.
The restriction of $\pi$ to $S_q$, also denoted by $\pi: S_q \to \Sigma$, 
is a double covering branched along the zeros of $q$.

The spectral curve $S_q$ is smooth if and only if $q$ has only simple
zeros. It is reducible if and only if $q = -\omega \otimes \omega$ for some $\omega \in H^0(K)$.
We shall refer to such quadratic differentials as \emph{reducible}, and
\emph{irreducible} otherwise.
There is a noteworthy observation regarding irreducible spectral curves. 
\begin{proposition}[cf.\ {\cite{hitchin1987stable}}]
	Let $(\ME,\vp)$ be a Higgs bundle with $q = \det(\vp)$, and suppose $q$ is irreducible. Then
    $(\ME,\vp)$ has no proper invariant subbundles. In particular, $(\ME,\vp)$ is stable.
\end{proposition}

\begin{proof}
    Suppose $\mathcal{L}\subset\mathcal{E}$ is $\varphi$-invariant, and
    let $\varphi_{\mathcal L}$ be the restriction. 
    Then 
    $$
    \det\varphi=-\frac{1}{2}\Tr(\varphi^2)=-(\varphi_{\mathcal L})^2\ ,
    $$
    contradicting the assumption.
\end{proof}

Let us emphasize that being reducible is not the same as having only even zeros. 
To see this, 
 suppose that $\Div(q)=2\Lam'$. Then  $K\simeq\MO(\Lam')\otimes \MI$, 
where $\MI$ is a 2-torsion point in the Jacobian. 
The spectral curve $S_q$ is reducible if and only if $\MI$ is trivial.

\subsection{Rank 1 torsion free  sheaves and the BNR correspondence}
In this subsection, we provide some 
background on rank 1 torsion free sheaf theory over  spectral curves in the
context of the Hitchin and BNR correspondence, as developed in \cite{hitchin1987stable,bnr1989spectral}.

Let $S$ be a reduced and irreducible complex projective curve and $\MO_S$  its structure sheaf. 
The moduli space of invertible sheaves on $S$ is denoted by $\Pic(S)$, and
$\Pic^d(S)\subset \Pic(S)$ is the degree $d$ component. If $\mathcal F$ is a
coherent analytic sheaf on $S$, we can define its cohomology groups
$H^i(X,\mathcal{F})$. Since $\dim S=1$,  $H^i(X,\mathcal{F})=0$ for $i\geq 2$. The Euler characteristic 
is defined as $\chi(\mathcal{F})=\dim H^0(X,\mathcal{F})-\dim
H^1(X,\mathcal{F})$. The degree of a torsion free sheaf $\MF$ is given by
$\deg(\mathcal{F})=\chi(\mathcal{F})-\rank(\mathcal{F})\chi(\MO_S)$.
If $\mathcal{F}$ is locally free, then
$\deg(\mathcal{F})$ coincides with the degree of the invertible sheaf $\det(\mathcal{F})$.

Let $\BPic^d(S)$ be the moduli space of degree $d$ rank 1 torsion free
sheaves on $S$, and $\BPic(S)=\prod_{d\in \mathbb{Z}}\BPic^d(S)$ \cite{cyril1979compactification}. 
Then $\BPic^d(S)$ is an irreducible projective scheme containing
$\Pic^d(S)$ as an open subscheme. When $S$
is smooth, we have $\BPic^d(S)=\Pic^d(S)$. The relationship to Higgs
bundles is given by the following.

\begin{theorem}[{\cite{hitchin1987stable,bnr1989spectral}}]
	\label{thm_BNRcorrespondence}
	Let $q\in H^0(K^2)$ be an irreducible quadratic differential with
    spectral curve $S_q$. There is a bijective correspondence between
    points in $\BPic(S_q)$ and
    isomorphism classes of rank 2 Higgs pairs $(\ME,\vp)$ with
    $\Tr(\vp)=0$ and 
    $\det(\vp)=q$. Explicitly:
    if $\ML\in \BPic(S_q)$, then $\ME:=\pi_{\st}(\ML)$ is a rank 2 vector
    bundle, and the homomorphism $\pi_{\st}\ML\to \pi_{\st}\ML\otimes K\cong
    \pi_{\st}(\ML\otimes \pi^{\st}K)$ given by multiplication by the
    canonical section $\lam$ defines the Higgs field $\vp$. 
\end{theorem}

This correspondence gives the very useful exact sequence
\begin{equation} \label{eq_BNR}
	\begin{split}
		0\to \ML(-\Delta)\to
        \pi^{\st}\ME\xrightarrow{\pi^{\st}\vp-\lam}\pi^{\st}\ME\otimes
        \pi^{\st}K\to \ML\otimes \pi^{\st}K\to 0\ ,
	\end{split}
\end{equation}
where $\MO_S(\Delta):=K_S\otimes \pi^{\st}K_{\Sigma}^{-1}$ is the ramification divisor.
This will be used below in Section
\ref{sec_irreducible_singular_fiber}.

Let $q$ be a quadratic differential with only simple zeros, and define a
divisor on $S$ by $\Lam=\Div(\lambda)$.
Then $\Lam$ is the ramification divisor of the map $\pi:S\to \Sigma$. 
By the Riemann-Hurwitz formula, the genus of $S$ is $g(S)=4g-3$. 
Furthermore, for any $\ML\in \Pic(S)$, Riemann-Roch gives $\deg(\pi_{\st}\ML)=\deg(\ML)-(2g-2).$
The $\SLC$ Higgs bundles are characterized by
\begin{equation} \label{eq_definition_T}
	\MT:=\{\ML\in\Pic^{2g-2}(S)\mid\det(\pi_{\st}\ML)=\MO_\Sigma\}.
\end{equation}
By the Hitchin-BNR correspondence (Theorem \ref{thm_BNRcorrespondence}), the map $\chi_{BNR}:\MT\to \MM_q$ is a bijection.

The branched  double cover $\pi:S\to \Sigma$
is given by an involution $\sigma:S\to S$.
We have the norm map
$\Nm_{S/\Sigma}:\Jac(S)\to \Jac(\Sigma)$, where $\Nm_{S/\Sigma}(\MO_{S}(D)):=\MO_{\Sigma}(\pi(D))$. The Prym variety is defined as 
$$\Prym(S/\Sigma):=\ker(\MNM_{S/\Sigma})=\{\ML\in\Pic(S)\mid\ML\otimes
\sigma^{\st}\ML=\MO_S\}\ .$$ 
Also, we have $\det(\pi_{\st}\ML)\cong \MNM_{S/\Sigma}(\ML)\otimes K^{-1}$.
Thus, $\MT$ can be expressed as
$$\MT=\{\ML\in\Pic^{2g-2}(S)\mid\MNM_{S/\Sigma}(\ML)\cong K\}\ .$$
Hence, $\MT$ is a torsor over $\Prym(S,\Sigma)$.
Explicitly, by choosing
$\ML_0\in \MT$,  we obtain an isomorphism $\MT\isorightarrow
\Prym(S,\Sigma)$ given by $\ML\to \ML\otimes\ML_0^{-1}$. 

To summarize, we have the following:
\begin{proposition}
	Let $q$ be a quadratic differential with simple zeros. Then $\MM_q\cong
    \MT\cong \Prym(S,\Sigma)$.
\end{proposition}

If $q\neq 0$ is irreducible but nongeneric, the spectral curve $S$ is
singular and irreducible. We may still define the set $\BMT\subset \BPic^{2g-2}(S)$ as follows:
\begin{equation*}
	\BMT:=\{\ML\in\BPic^{2g-2}(S)\mid\det(\pi_{\st}\ML)\cong \MO_\Sigma\}\
    ,
\end{equation*}
We also set $\MT:=\BMT\cap \Pic^{2g-2}$.
Then $\BMT$ is the natural compactification of $\MT$ induced by the inclusion $\Pic^{2g-2}(S)\subset \BPic(S)$. The BNR correspondence, as stated in Theorem \ref{thm_BNRcorrespondence}, implies that 
$\chi_{\BNR}:\BMT\to \MM_q$ is an isomorphism.


\subsection{The Hitchin moduli space and the nonabelian Hodge
correspondence} \label{sec:nah}
We now recall the well-known nonabelian Hodge correspondence (NAH), 
which relates the space of flat $\mathrm{SL}(2,\mathbb{C})$ connections, Higgs bundles, and solutions to the Hitchin equations. This result was developed in the work of Hitchin \cite{hitchin1987self}, Simpson \cite{Simpson1988Construction}, Corlette \cite{corlette1988flat}, and Donaldson \cite{donaldson1987twisted}.

As above, let $E$ be a trivial, smooth,  rank 2 vector bundle over a Riemann surface
$\Sigma$, and let
$H_0$ be a fixed Hermitian metric on $E$. 
We denote by $\mathfrak{sl}(E)$ (resp.\ $\mathfrak{su}(E)$) the bundle of
traceless (resp.\ traceless skew-hermitian) endomorphisms of $E$.
Let $A$ be a unitary  (with respect to $H_0$)
connection on $E$ that induces the trivial connection on $\det E$,
and let $\phi \in \Omega^1(i\mathfrak{su}(E))$. We will sometimes also
refer to $\phi$ as a Higgs field. 
The Hitchin equations  for the pair $(A,\phi)$ are given by:
\begin{equation} \label{eq_Hitchin_real}
	\begin{split}
        F_{A} + \phi\wedge\phi= 0\ ,\ d_A \phi = d_A^\ast\phi= 0\ .
	\end{split}
\end{equation}
If we split the Higgs field into type:
$\phi = \vp + \vp^{\dagger}$, with $\vp \in \Omega^{1,0}(\mathfrak{sl}(E))$, then
\eqref{eq_Hitchin_real} is equivalent to:
\begin{equation} \label{eq_Hitchin}
	\begin{split}
		F_{A} + [\vp,\vp^{\dagger}]= 0\ ,\ \bar{\partial}_A \vp = 0\ .
	\end{split}
\end{equation}
Notice that $(\bar\partial_E,\varphi)$ then defines an $\SL(2,\mathbb{C})$ Higgs bundle.
The Hitchin moduli space, denoted by $\MMH$, is the moduli space of solutions to the Hitchin equation, given by
$$\MMH := \{(A,\phi) \mid (A,\phi)\;\mathrm{satisfies}\;\eqref{eq_Hitchin_real}\}/\MG,$$
where $\MG$ is the gauge group of unitary automorphisms of $E$. 
Recall that a flat connection $\MD$ is called completely reducible if and
only if it is a direct sum 
of irreducible flat connections. The NAH can be summarized as follows:

\begin{theorem}[{\cite{hitchin1987self,simpson1990harmonic,corlette1988flat,donaldson1987twisted}}]
A Higgs bundle $(\mathcal{E},\varphi)$ is polystable if and only if there
    exists a Hermitian metric $H$ such that the corresponding Chern
    connection $A$ and Higgs field $\phi=\varphi+\varphi^{\dagger}$ solve the Hitchin
    equations \eqref{eq_Hitchin_real}. Moreover, the connection
    $\mathcal{D}$ defined by $\mathcal{D}=\nabla_A+\phi$ is a completely
    reducible flat connection, and it is irreducible if and only if $(\mathcal{E},\varphi)$ is stable.

Conversely, a flat connection $\mathcal{D}$ is completely reducible if and
    only if there exists a Hermitian metric $H$ on $E$ such that when we
    express $\mathcal{D}=\nabla_A+\varphi+\varphi^{\dagger}$, we have
    $\bar{\partial}_{\mathcal{E}}\varphi=0$. Moreover, the corresponding
    Higgs bundle
    $(\mathcal{E},\varphi)$ is polystable, and it is stable if and only if $\mathcal{D}$ is irreducible.
 \end{theorem}

The nonabelian Hodge correspondence defines the following
Kobayashi-Hitchin homeomorphism
\begin{equation*}
	\begin{split}
		\Xi:\MMD\to \MMH\ ,
	\end{split}
\end{equation*}
which is a diffeomorphism to irreducible solutions of
\eqref{eq_Hitchin_real}  when restricted to the stable locus. 

Finally, we note that there is an
action of $S^1$ on $\MMH$ defined by $(A,\phi) \to (A,e^{i\theta}\cdot
\phi)$, where
$e^{i\theta}\cdot\phi=e^{i\theta}\varphi+e^{-i\theta}\varphi^\dagger$.
With respect to this and the $S^1\subset\mathbb{C}^\ast$ action on $\MMD$, 
the map $\Xi$ is $S^1$-equivariant.

\section{Filtered bundles and compactness} \label{sec_filtered_bundles_and_compactness}
Filtered (or parabolic) bundles 
are described, for example,  in \cite{simpson1990harmonic}. 
They play a key role in the analytic compactification.
This section provides a brief overview of filtered line bundles and
demonstrates a compactness result.

\subsection{Filtered line bundles and  nonabelian Hodge}
Let $Z$ be a finite collection of distinct points on a closed Riemann surface $\Sigma$, and
let $\Sigma' = \Sigma\setminus Z$. Viewing $\Sigma$ as a projective
algebraic curve,  an algebraic line bundle $L$ over $\Sigma'$ is a
line bundle defined by regular transition functions on Zariski
open sets over $\Sigma'$. The sheaf of sections of $L$ can be extended
in infinitely many different ways over $Z$ to obtain 
coherent (invertible) sheaves on $\Sigma$. 
The sections of $L$ are then realized as 
  meromorphic sections of an extension, regular on $\Sigma'$.

A \emph{filtered line bundle} $\MF_{\st}(L)$ is an algebraic line bundle
$L$ with a collection of coherent extensions $L_{\al}$ across the punctures $
Z$ such that $L_{\al} \subset L_{\be}$ for $\al \geq \be$, for a fixed
sufficiently small $\ep$,
$L_{\al-\ep}=L_{\al}$, and $L_{\al}  = L_{\al+1} \otimes
\MO_\Sigma(Z)$.
Let $\Gr_{\al}= L_{\al+\ep}/L_{\al}$ denote the quotient (torsion)
sheaf. A value $\alpha$ where  $\Gr_{\al} \neq 0$ is called a jump.
Since we are considering  line bundles,  for each $p$  in
the support of $\Gr_{\al_p}$, 
there is exactly 
one jump $\al_p$ in the interval $[0,1)$. The collection of  jumps $\{\alpha_p\}$ fully determines the filtered
bundle structure. If we denote by
$\ML:=L_{0}$, the degree of a filtered line bundle is defined as 
$$\deg(\MF_{\st}(L)) = \deg(\ML) + \sum_{p\in Z}\alpha_p
\ .$$

Alternatively,
a \emph{weighted line bundle} is a pair $(\ML,\chi)$ where $\ML\to \Sigma$
is a holomorphic line bundle  and $\chi:Z\to \mathbb{R}$ is a weight function.
The degree of a weighted bundle is defined as 
$$\deg(\ML,\chi) = \deg(\ML) + \sum_{p\in Z}\chi_p\ .$$

The concepts of filtered line bundles and weighted line bundles are
nearly equivalent.  Namely, given
a filtered line bundle $\MF_{\st}(L)$, we define $\ML:=\ML_{0}$ and
$\chi_p=\alpha_p$.
Conversely, given a weighted line bundle $(\ML,\chi)$, 
let $\alpha_p=\chi_p+n_p$, where $n_p\in \mathbb{Z}$ is the unique integer 
so that $0\leq \chi_p+n_p<1$. A filtered bundle $\MF_{\ast}(L)$ is then determined by
setting $L_{0} =
\ML(-\sum_{p\in Z} n_p p)$ with jumps $\alpha_p$. 
Clearly,
$\deg(\mathcal{F}(L))=\deg(\mathcal{L},\chi)$. 
We shall use the notation $\MF_{\ast}(\mathcal{L},\chi)$ for the filtered
bundle associated to a weighted bundle $(\mathcal{L},\chi)$ in this way. 

Different weighted bundles can give rise to the same filtered bundle. 
The following is a fact that will be frequently used in this paper. 
If $D=\sum_{x\in Z}d_xx$ is a divisor  supported  on $Z$, 
let 
$$
\chi_D(x):=\begin{cases}d_x\ ,&x\in Z\ ;\\ 0\ ,&x\in\Sigma\setminus Z\ .
\end{cases}
$$
 Then for any weighted bundle $(\mathcal{L},\chi)$ we have 
$\MF_\ast(\mathcal{L}(D),\chi-\chi_{D})=\MF_\ast(\mathcal{L},\chi)$.

Let $(\ML_1,\chi_1)$ and $(\ML_2,\chi_2)$ be two weighted lines bundles. We define the tensor product 
$$(\ML_1,\chi_1)\otimes (\ML_2,\chi_2):=(\ML_1\otimes \ML_2,\chi_1+\chi_2)\
.$$
Then the degree is additive on tensor products. For filtered bundles, we
\emph{define}
\begin{equation}\label{eqn:tensor}
    \mathcal{F}_\ast(\ML_1,\chi_1)\otimes\mathcal{F}_\ast(\ML_2,\chi_2):=\mathcal{F}_\ast(\ML_1\otimes
    \ML_2,\chi_1+\chi_2)\ .
\end{equation}
The degree is again additive for the tensor product of filtered bundles.
This agrees with the usual definition of tensor product for parabolic
bundles. 

\subsection{Harmonic metrics for weighted line bundles}
\begin{proposition}
	\label{prop_NAHforfilteredlinebundle}
	Let $(\ML,\chi)$ be a degree 0 weighted bundle. Then there exists a Hermitian metric $h$ 
    on $\mathcal{L}_{\Sigma'}$ such that:
	\begin{itemize}
        \item [(i)] the Chern connection $A_h$ of $(\ML,h)$ is flat: $F_{A_h}=0$;
		\item [(ii)] for $p\in Z$, and $(U_p,z)$  a holomorphic coordinate
            centered at $p$,  $|z|^{-2\chi_p}h$ extends to  a $\MC^{\infty}$ Hermitian metric on $\ML|_{U_p}$;
		\item [(iii)]  $h$ is uniquely determined up to a multiplication by a nonzero constant.
	\end{itemize}
\end{proposition}
\begin{proof}
	We first choose a background Hermitian metric $h_0$ such that
    $|z|^{-2\chi_p}h_0$ defines a $\mathcal{C}^{\infty}$ Hermitian metric
    defined on $U_p$. Let $A_{h_0}$ be the Chern connection,  and $F_{A_0}$  the curvature.
    Note that $F_{A_0}$ is smooth on $\Sigma$.  By the Poincar\'e-Lelong
    formula,  we have
    $\frac{\sqrt{-1}}{2\pi}\int_{\Sigma}F_{A_0}=\deg(\mathcal{L},\chi)=0$.
    Therefore, there exists a $\mathcal{C}^{\infty}$ function $\rho$ such
    that $\Delta\rho+\frac{\sqrt{-1}}{2\pi}\Lambda F_{A_0}=0$. We define
    $h=h_0e^{\rho}$. For the corresponding Chern connection $A_h$, we have
    $F_{A_h}=0$, which implies (i). (ii) follows from the property for
    $h_0$, since   $\rho$ is a  smooth function on $\Sigma$. As $\rho$ is well-defined up to a constant, $h$ is well-defined up to a constant, which implies (iii).
\end{proof}

The metric obtained above is called the \emph{harmonic metric}. For a weighted bundle $(\mathcal{L},\chi)$, the holomorphic bundle $\mathcal{L}$ and the harmonic metric $h$ define a filtration as follows. 
For $\epsilon>0$ sufficiently small, let 
$$L_{\alpha}:=\{s\in\ML(\ast Z)  \mid |s|_h\leq Cr^{\alpha-\epsilon}\text{
    for some $C$}\}\ . $$
Here, $r$ denotes the distance to $Z$ in any smooth conformal metric  on
$\Sigma$. It is straightforward to check that  this defines a filtered
bundle that matches $\MF_\ast(\ML,\chi)$ under the correspondence in the previous
section.

Even though the harmonic metric is only well-defined up to a constant, the
Chern connection $A=(\ML,h)$ is independent of this choice.
The $(1,0)$ part of $A$, denoted $\nabla_h$, then defines
logarithmic connections $\nabla_h:L_{\alpha}\to
L_{\alpha}\otimes K(Z).$

\subsection{Convergence of weighted line bundles}
In this subsection, we will consider the convergence of weighted line
bundles. The main result we prove here is a consequence  of \cite[Theorem 1.8]{mochizuki2023asymptotic}. For the reader's convenience, we present a short proof here for our situation. 

Let $(\Sigma_0,g_0)$ be a metrized Riemann surface (i.e.\ a Riemann surface
$\Sigma_0$
with conformal metric $g_0$).  We view $\Sigma_0$ as given by an underlying
surface $C$ with almost complex structure $J_0$. 
Consider a neighborhood $U_1$ of $J_0$ in the moduli space of holomorphic structures and a neighborhood $U_2$ of $g_0$ in the space of smooth metrics. We denote the product of these neighborhoods by $U=U_1\times U_2$. We can define the fiber bundle $\Pic_U \to U$, where each fiber is the Picard group defined by the holomorphic structure.  
Let $(\Sigma_t=(C,J_t),g_t)$ be a family of metrized Riemann surfaces that converge smoothly to $(\Sigma_0,g_0)$ as $t\to 0$. Let $Z_t\subset \Sigma_t$ be a collection of a finite number of points that converge to $Z_0$ in suitable symmetric products of $C$. For each $p\in Z_0$, we can write $Z_t=\cup_{p\in Z_0}Z_{t,p}$ such that all points in $Z_{t,p}$ converge to $p$. We define the convergence of weighted line bundles as follows:

\begin{definition}
	\label{def_convergence_parabolic_bundles}
	A family of weighted line bundles $(\ML_t,\chi_t)$ over $\Sigma_t$ with weights $\chi_t:Z_t\to \mathbb{R}$ converges to $(\ML_0,\chi_0)$ if
	\begin{itemize}
		\item [(i)] $\ML_t$ converges to $\ML_0$ in $\Pic_U$,
		\item [(ii)] for all $p\in Z_0$ and $t$ sufficiently small, $\sum_{q\in Z_{t,p}}\chi_t(q)=\chi_0(p)$.
	\end{itemize}
\end{definition}
A sequence of filtered bundles $\MF_{\st}(\ML_t)$ converges to $\MF_{\st}(\ML_0)$ if the corresponding weighted bundles converge. The following theorem provides insight into the compactness of a sequence of weighted line bundles:
\begin{theorem}
\label{thm_convergence_family_harmonic_bundles}
Let $(\ML_t,\chi_t)$ be a sequence of weighted line bundles over
    $\Sigma_t\setminus Z_t$, where $\deg(\ML_t,\chi_t)=0$, and let $h_t$ be
    the corresponding harmonic metrics. If $Z_t$ converges to $Z_0$, we
    write $Z_t=\cup_{p\in Z_0}Z_{t,p}$.
    Then there exists a weighted line bundle $(\ML_0,\chi_0)$ over $Z_0$ with a harmonic metric $h_0$ such that:
\begin{itemize}
	\item [(i)] After being rescaled by $c_t>0$, $c_th_t$ converges to $h_0$ over $\Sigma_0\setminus Z_0$ in the $\MC^{\infty}_{\loc}$ sense.
	\item [(ii)] Let $\na_t$ be the unitary connection of $h_t$. Then on 
        $\Sigma_0\setminus Z_0$, $\lim_{t\to 0}\na_t=\na_0$  in $\MC^{\infty}_{\loc}$.
\end{itemize}
\end{theorem}
\begin{proof}
By the assumptions on weights, 
    $\deg(\ML_t)$ is a fixed, $t$-independent constant. Let $\gamma_t=(J_t,g_t)$ be a path in $U$. 
    Then $\Pic_U|_{\gamma_t}$ is compact, and there exists an $\ML_0\in \Pic(\Sigma_0)$ such that $\ML_t$ converges to $\ML_0$.
For $p\in Z_0$,  define $\chi_0(p)=\sum_{q\in Z_{t,p}}\chi_t(q)$, and  thus 
    obtain a weighted line bundle $(\ML_0,\chi_0)$. We can choose a family
    of approximate harmonic metrics  $h_t^{\app}$, such that $|z|^{-2\chi_p}h_t^{\app}$ extends to a smooth metric in a neighborhood of $p$ and $h_t^{\app}$ converges to $h_0^{\app}$ in $\MC^{\infty}_{\loc}(\Sigma_0\setminus Z_0)$. 
    Moreover, we write $h_t=h_t^{\app}e^{s_t}$. After a suitable rescale of $h_t$, we can assume $\|s_t\|_{L^2}=1$.
Let $\rho_t:=\Delta_th_t^{\app}$ be the curvature defined by the metric
    $h^{\app}_t$. Then $s_t$ satisfies the equation $\Delta_ts_t=\rho_t$
    over $\Sigma$. As $\rho_t$ converges to $\rho_0\in\MC^{\infty}_{\loc}(\Sigma\setminus Z_0)$, and $g_t$ is a family with bounded geometry, we obtain the estimate
$$
\|s_t\|_{\MC^{k+2,\al}(\Sigma)}\leq
    C_{k,\al}(\|\rho_t\|_{\MC^{k,\al}(\Sigma)}+1)\ ,
$$
where $C_{k,\al}$ is a $t$-independent constant. Therefore, passing to a
    subsequence, $s_t$ converges to $s_0$ in $\MC^{\infty}(\Sigma)$, which
    implies (i). The assertion (ii) follows from (i).
\end{proof}

\section{The algebraic and analytic compactifications}
\label{sec_the_algebraic_and_analytic_compactifications}
In this section, we introduce compactifications of  the Dolbeault and
Hitchin moduli spaces. 

\subsection{The algebraic compactification of the Dolbeault moduli space}
In this subsection, we present the algebraic method for compactifying the Dolbeault moduli space.
This technique is based on the $\CS$ action on $\MMD$, and was introduced in 
\cite{simpson1996hodge, Schmitt:98, hausel1998compactification,de2018compactification,katzarkov2015harmonic}. 
The gauge theoretic approach can be found in \cite{fan2022analytic}.

\begin{theorem}[{\cite[Thm.\ 11.2]{simpson1996hodge},\cite{de2018compactification}}]
	\label{thm_algebraiclemmasimpson}
Let $V$ be an algebraic variety with $\CS$ action. Suppose 
\begin{itemize}
	\item [(i)]the fixed point set of the $\CS$ action is proper,
    \item [(ii)]for every $t\in \CS,\;v\in V$, the limit $\displaystyle\lim_{t\to 0}t\cdot v$ exists.
\end{itemize}
Then the space $\displaystyle U:=\{v\in V\mid \lim_{t\to\infty}t\cdot
    v\;\mathrm{does\;not\;exist}\}$ is open in $V$, and the quotient $U/\CS$
    is separated and proper.
\end{theorem}

We apply this to the Dolbeault moduli space. The first step is to note that
the possible isotropy subgroups are limited.
\begin{lemma}{\cite[Thm.\ 6.2]{hausel1998compactification}}
Let $\xi=[(\ME,\vp)]$ be a Higgs bundle equivalence class with $\MH(\xi)\neq 0$. Then the stabilizer $\Gamma_{\xi}$ 
of $\xi$ for the $\CS$ action is either trivial or $\mathbb{Z}/2$. The latter case holds if and only if 
$(\ME,\vp)$ and $(\ME,-\vp)$ are complex gauge equivalent.
\end{lemma}
\begin{proof}
For $t\in\Gamma_\xi$,  $\mathcal{H}(t\cdot \xi)=t^2\mathcal{H}(\xi)$, hence $t^2=1$ if $\MH(\xi)\neq 0$.
\end{proof}

By this Lemma, the space 
$(\MMD\setminus \MH^{-1}(0))/\CS$ has an orbifold structure. 
In passing, we note that the 
fixed points of the $\mathbb{Z}/2$ action correspond to real representations under the nonabelian Hodge 
correspondence \cite[Sec.\ 10]{hitchin1987self}. 


By the properness of the Hitchin map $\mathcal{H}$ (see Theorem
\ref{thm:proper-action}), 
it follows that $\displaystyle \lim_{t\to \infty}t\cdot \xi$  exists if and only if $\mathcal{H}(\xi)=0$. Now define
\begin{equation} \label{eqn:compactification}
    \BMMD = \left\{(\MMD \times \CS) \coprod (\MMD \setminus \MH^{-1}(0))\right\}/\CS\ .
\end{equation}
The analytic topology on the disjoint union is generated by  open sets
$U\times W_1$ and $V\times (W_2\cap\CS)\amalg V\cap(\MMD \setminus \MH^{-1}(0))$,
where
$U,V\subset\MMD$,
 $W_1,W_2\subset \mathbb{C}$  are
open, and $0\not\in  W_1$, $0\in W_2$.
The topology on $\BMMD$ is then the quotient topology, and it is straightforward to see that with this topology,
it is compact. 


Since $(\MMD \times \CS)/\CS=\MMD$, there is a natural inclusion
\begin{equation*}
    \iota : \MMD \to \BMMD,\;\iota(\xi) = [(\xi,1)]\ ,
\end{equation*}
where brackets denote the equivalence  class under the $\CS$ action. The boundary of $\BMMD$ is 
$$\partial \BMMD = \BMMD \setminus \iota(\MMD)=(\MMD \setminus
\MH^{-1}(0))/\CS\ .$$
There is a \emph{boundary map}
\begin{equation*}
	\iota_{\partial} : \MMD \setminus \MH^{-1}(0) \longrightarrow \partial
    \BMMD,\; \xi \mapsto [( \xi, 0)]\ ,
\end{equation*}
which is invariant under the $\CS$ action, i.e., $\iota_{\partial}(\lambda \xi) = \iota_{\partial}(\xi)$
for $\lambda \in \CS$. 

The $\CS$ action on $\MMD$ covers the square of the action on $\MB$.
Hence, it is natural to compactify $\MB$ by projectivizing: 
$$
\overline\MB := \mathbb{P}(H^0(K^2)\oplus \mathbb{C}) \ . 
$$
The inclusion is given, as usual, by
\begin{equation*}
\iota_0:\MB\to\PMB,\;\iota(q)=[q\ti \{1\}]\ ,
\end{equation*}
where $q\ti \{1\}\in H^0(K^2)\oplus \mbC$.  
We also define $\pa\PMB=\PMB\setminus \iota_0(\MB)\simeq
\mathbb{P}(H^0(K^2))$, with boundary projection map
\begin{equation*}
	\iota_{0,\pa}:\MB\setminus \{0\}\to \pa\PMB\ ,\ \iota_{0,\pa}(q)=[q\ti
    \{0\}]\ .
\end{equation*}
The Hitchin map $\MH:\MMD\to \MB$ extends to $\BMH:\BMMD\to \PMB$, where 
$\BMH|_{\MMD}:=\iota_0\circ \MH$,  and for every $[(\ME,\vp)]/\CS\in \pa\BMMD$, 
$$
\BMH([(\ME,\vp)]/\CS):=[(\MH(\vp), 0)]\subset\PMB\ .
$$ 
This is well defined, since $\det(\vp) \neq 0$ if $[(\ME,\vp)]/\CS\in \pa
\BMMD$. Moreover,
$$
\begin{tikzcd} \MMD \arrow[r, "\iota"] \arrow[d, "\MH"] & \BMMD \arrow[d,"\BMH"] \\ \MB \arrow[r, "\iota_0"] & \PMB \end{tikzcd}
$$
commutes.

There is a good algebraic structure on this algebraic compactification:

\begin{theorem}[{\cite{simpson1996hodge, Schmitt:98,
    hausel1998compactification,de2018compactification,fan2022analytic}}]
	The compactified space $\BMMD$ is a normal projective variety, and
    $\pa\BMMD$ is a Cartier divisor of $\BMMD$. 
\end{theorem}

The following characterization of sequential convergence is useful.
\begin{proposition}
\label{prop_convergence_Dol_space}
Let $[(\ME_i,\vp_i)]\in\MMD$ be a sequence of Higgs bundles, and write $q_i=\det(\vp_i)$ and $r_i=\|q_i\|_{L^2}^{\frac12}$. 
Suppose $\limsup r_i\to \infty$. Then up to subsequence:
\begin{itemize}
\item [(i)] there exists a Higgs bundle $[(\hME_{\infty},\hvp_{\infty})]$ with $\tq_{\infty}=\det(\tvp_{\infty})$ and 
$\|\hq_{\infty}\|_{L^2}=1$ such that $\lim_{i\to \infty}[(\ME_i,r_i^{-1}\vp_i)]=[(\hME_{\infty},\hvp_{\infty})]$ in $\MMD$ 
and $\lim_{i\to \infty}r_i^{-1}q_i=\hq_{\infty}$ in $H^0(K^2)$;
\item [(ii)] 
\begin{align*}
    \lim_{i \to \infty}
    \iota[(\ME_i,\vp_i)]&=\iota_{\pa}[(\hME_{\infty},\hvp_{\infty})]\ ,\
    \mbox{on}\ \BMMD\ , \\ 
\lim_{i\to \infty}\iota_0(q_i)&=\iota_{0,\pa}(\hq_{\infty})\ ,\  \mbox{on}\
    \PMB\ .
\end{align*}
\end{itemize}
\end{proposition}
\begin{proof}
The first point follows since the Hitchin map $\MH$ is proper and $\MH(r_i^{-1}\vp_i)$ is
bounded. The second follows directly from the definition.
	\end{proof}

\subsection{The analytic compactification of the Hitchin moduli space}
We next describe the compactification of the Hitchin moduli space, as 
developed in \cite{mazzeo2012limiting,Mochizukiasymptotic,taubes2013compactness}. 
\subsubsection{Decoupled Hitchin equations}
We begin by defining the decoupled Hitchin equations. Recall the notation from Section \ref{sec:nah}, let $E$ be a trivial, smooth,  rank 2 vector bundle over a Riemann surface $\Sigma$, and let $H_0$ be a background Hermitian metric on $E$. Let $Z$ be a finite set of distinct points in $\Sigma$. For a smooth  unitary connection $A$ on $E|_{\Sigma\setminus Z}$ and
smooth $\phi=\vp+\vp^{\da}\in \Omega^1(i\su(E))|_{\Sigma\setminus Z}$, the
\emph{decoupled Hitchin equations} on $\Sigma\setminus Z$  are:
\begin{equation} \label{eq_decoupled_Hitchin_equation}
	\begin{split}
		F_A=0\ ,\ [\vp,\vp^{\da}]=0\ ,\ \bar{\pa}_A\vp=0\ .
	\end{split}
\end{equation}
Solutions to \eqref{eq_decoupled_Hitchin_equation} may be quite singular near $Z$, so we make the following restriction:
\begin{definition}
A solution $(A,\phi)$ to the decoupled Hitchin equations over
    $\Sigma\setminus Z$ is called \emph{admissible} 
    if $\phi\neq 0$, and  $|\phi|$ extends to a continuous function on $\Sigma$ with $|\phi|^{-1}(0)=Z$.
\end{definition}
By a \emph{limiting configuration} we always mean an admissible solution to the decoupled Hitchin equations. 
Clearly, $Z$ is determined by $(A,\phi)$. Admissibility guarantees  that $\det(\vp)$ extends to a holomorphic 
quadratic differential $q=\det(\vp)$ on $\Sigma$,  with $Z=q^{-1}(0)$ the
zero locus. Hence, the spectral curve $S_q$ 
is well-defined. We emphasize that $Z$ may vary for different admissible solutions, but one always has 
that $\# Z\leq 4g-4$. 

The equivalence relation on limiting configurations is that $(A_1,\phi_1)\sim (A_2,\phi_2)$ if $Z_1=Z_2$ and
$(A_1,\phi_1)g=(A_2,\phi_2)$ for  a smooth unitary gauge transformation $g$ on $\Sigma\setminus Z_1$. 
The moduli space of decoupled Hitchin equations is then
\begin{equation*}
        \MMH^{\LC}=\{\text{admissible solutions to }\eqref{eq_decoupled_Hitchin_equation}\}/\sim\ .
\end{equation*} 
We denote by $\MMHQLC$ the elements in $\MMH^{\LC}$ with a fixed quadratic
differential $q$. In this case,  the equivalence 
relation is  induced by the action of the unitary gauge group over $\Sigma\setminus Z$,   $Z=q^{-1}(0)$.

There is a natural $\CS$ action on the moduli space $\MMH^{\LC}$: given $(A,\phi=\vp+\vp^{\da})\in \MMH^{\LC}$ 
and $t\in \CS$, we set $t\cdot[(A,\phi)]=[(A,t\vp+\bar{t}\vp^{\da})]$, which is also a solution 
to \eqref{eq_decoupled_Hitchin_equation}.

\subsubsection{Compactification of the Hitchin moduli space}
The following compactness result is due to Taubes \cite{Taubes20133manifoldcompactness}
and Mochizuki \cite{Mochizukiasymptotic} (see also \cite{he2020behavior}).

\begin{proposition} 
	\label{prop_general_convergence_solutions}
	Let  $(A_{i},\vp_i)$ be a sequence of 
    solutions to \eqref{eq_Hitchin_real}, with  $q_i=\det(\vp_i)\in
    H^0(K^2)$. Then 
	\begin{itemize}
		\item [(i)] if $\limsup\|q_i\|_{L^2(\Sigma)}<\infty$, then there
            is a subsequence (also denoted $\{i\}$),
            a smooth solution  $(A_{\infty},\phi_{\infty})$ to
            \eqref{eq_Hitchin_real},
           and a sequence $g_i$ of smooth unitary gauge transformations on
            $\Sigma$,  
            such that $(A_i,\phi_i)g_i$ converges smoothly to
            $(A_{\infty},\phi_{\infty})$ on $\Sigma$;
		\item [(ii)]	if $\lim\|q_i\|_{L^2(\Sigma)}=\infty$, then there
            is a subsequence (also denoted $\{i\}$), and $q_{\infty}\in H^0(K^2)$ 
            so that
            $$\frac{q_i}{\Vert q_i\Vert_{L^2}}\longrightarrow
            q_{\infty}$$ over $\Sigma$, and an admissible solution 
            $(A_\infty,\phi_{\infty}=\vp_{\infty}+\vp^{\da}_{\infty})$ to
            \eqref{eq_decoupled_Hitchin_equation}, with 
			$Z_{\infty}:=q_{\infty}^{-1}(0)$,  and smooth unitary gauge
            transformations $g_i$ on $\Sigma\setminus Z_\infty$, 
            such that over any open set $ \Omega\Subset \Sigma\setminus Z_{\infty}$, 
            $(A_i)g_i\to A_\infty$, and  
            $$\frac{g_i^{-1}\phi_ig_i}{\Vert\phi\Vert_{L^2}}\longrightarrow
            \phi_{\infty}$$
           smoothly on $\Omega$.  
	\end{itemize}
\end{proposition}
The norm on $H^0(\Sigma,K^2)$ can be chosen arbitrarily, since it is a finite dimensional space.

There is also a compactness result for sequences of solutions in $\MMH^{\LC}$.
\begin{proposition}
	\label{prop_seq_compactification_limitingconfiguration}
    Let $[(A_i,\phi_i=\vp_i+\vp_i^{\dagger})]\in \MMH^{\LC}$ be a sequence
    of admissible solutions to \eqref{eq_decoupled_Hitchin_equation},
    and let $q_i=\det(\vp_i)$ be the corresponding quadratic differentials.
    Then after passing to a subsequence, there are $t_i\in\CS$, a limiting
    configuration $(A_{\infty},\phi_{\infty}=\vp_{\infty}+\vp_{\infty}^{\dagger})$
    with quadratic differential $q_{\infty}=\det(\vp_{\infty})\neq 0$,
    and a sequence $g_i$ of smooth gauge transformations on
    $\Sigma\setminus Z_\infty$,  such that:
	\begin{itemize}
		\item [(i)] $t_i^2q_i$ converges smoothly to $q_{\infty}$,
		\item [(ii)] over any open set $\Omega\Subset X\setminus Z_{\infty}$, 
            $(A_i,t_i\cdot \phi_i)g_i$ converges smoothly to $(A_{\infty},\phi_{\infty})$.
	\end{itemize}
\end{proposition}

\begin{proof}
Write $q_i=\det(\vp_i)\in H^0(K^2)$. Adjusting by $t_i$ if necessary, we
    may assume   $q_i$ converges to $q_{\infty}$ over $\Sigma$. Also, since $F_{A_i}=0$ over $\Sigma\setminus Z_i$
    and $Z_i$ converges to $Z_{\infty}$, we can apply both Uhlenbeck
    compactness and the classical bootstrapping method to obtain
    $A_{\infty}$ such that up to gauge $A_i$ converges smoothly to $A_{\infty}$ over $\Sigma\setminus Z_{\infty}$.
    Finally, the convergence of $\varphi_i$ follows by the bound on
    $q_i$'s.
\end{proof}

\subsubsection{The topology on the compactified space}
We now carefully define the topology on the space $\MMH \coprod
\MMH^{\LC}/\CS$. Choose a metric in the conformal class on $\Sigma$.  Let $W^{k,2}$ denote the Sobolev spaces on $\Sigma$ of
distributional sections with at least  $k$ derivatives in $L^2$. For a finite set of points $Z \subset \Sigma$ (or indeed any
closed subset), 
\begin{equation*}
	\begin{split}
		W^{k,2}_{\loc}(\Sigma\setminus Z):=\{f\mid f\in W^{k,2}(K),\;K\subset \Sigma\setminus Z,\;K\;\mathrm{compact}\}.
	\end{split}
\end{equation*}
These definitions extend easily to the space of connections and $\Omega^1(i \su(E))$ for a Hermitian vector bundle $(E,H_0)$ 
over $\Sigma$ with a fixed smooth background connection. 

Let ${\omega_n}$ be a nested collection of open sets with $\omega_n\subset \overline{\omega_n} \subset \omega_{n+1}$, 
with $\bigcup_{n}\omega_n= \Sigma\setminus Z$. We then define the seminorms $\|f\|_n:=\|f\|_{W^{k,2}(\omega_n)}$; in
terms of these, $W^{k,2}_{\loc}(\Sigma\setminus Z)$ a Fr\'echet space. 
 
For any $q\in H^0(K^2)\setminus\{0\}$, set $Z_q:=q^{-1}(0)$, and consider the moduli space
\begin{equation*}
	\mbM_q=\{(A,\phi)\in
    \MM_{q^\ast}\cap W^{k,2}(\Sigma)\}\cup \{(A,\phi)\in\MMHQLC\cap
    W^{k,2}_{\loc}(\Sigma\setminus Z_q)\}/\CS\ .
\end{equation*}
By classical bootstrapping of the gauge-theoretic elliptic equations, $\mbM_q$ is independent of $k\geq 2$. 

Next define $\mbM:=\mathcal{M}_{0}\cup \bigcup_{q\in H^0(K^2)\setminus\{0\}}\mbM_q$.  Its topology is generated by two types of open sets. For interior points 
$\xi=[(A,\phi)]\in \MMH\subset \mbM$ we use the open sets
\begin{equation*}
	\begin{split}
		V_{\xi,\ep}:=\{[(A',\phi')]\in \MMH\mid\|A'-A\|_{W^{k,2}(\Sigma)}+\|\phi'-\phi\|_{W^{k,2}(\Sigma)}<\ep\}
	\end{split}
\end{equation*}
from the topology of $\MMH$.  For any boundary point $\xi_0\in\MMH^{\LC}/\CS$, choose a representative $(A_0,e^{i\theta}\phi_0)$ for some 
constant $\theta$. Let $q=\det(\phi_0)$, and fix any open set $\om\Subset \Sigma\setminus Z_q$. Then, setting $\MMH^{\st}=\MMH\setminus \MH^{-1}(0)$, 
\begin{equation*}
	\begin{split}
		U_{\xi_0,\om,\ep}:=&\{(A,\phi)\in\MMH^{\st}\mid
        \|A-A_0\|_{W^{k,2}(\om)}+\inf_{\theta\in S^1}\|\|\phi\|_{L^2}^{-\frac12}\phi-e^{i\theta}\phi_0\|_{W^{k,2}(\om)}<\ep,\;\|\phi\|_{L^2}>\ep\}\\
			&\bigcup \, \{(A,\phi)\in\MMH^{\LC}|
			\|A-A_0\|_{W^{k,2}(\om)}+\|\phi-\phi_0\|_{W^{k,2}(\om)}<\ep\}
	\end{split}
\end{equation*}
defines an open set around $\xi_0$.  The sets $U_{\xi_0,\om,\ep}$ and $V_{\xi,\ep}$ generate the topology on $\mbM$.

\begin{theorem}
	The space $\mbM$ is a Hausdorff and compact.
\end{theorem}
\begin{proof}
The Hausdorff property follows from the definition of the topology. By Propositions \ref{prop_general_convergence_solutions} and
\ref{prop_seq_compactification_limitingconfiguration}, $\mbM$ is sequentially compact. Moreover, using this explicit base for
the topology $\mbM$ is first countable, and hence compact. 
\end{proof}

We may now define the compactification of the Hitchin moduli space as the
closure $\BMMH\subset\mathbb{M}$;
we write $\pa\BMMH$ for the boundary of the closure, and $\BMMHq:=\BMMH\cap \mbM_q$ for the subset of elements 
with a fixed quadratic differential.  

The following result is described in \cite{mazzeo2016ends,ott2020higgs,mazzeo2019asymptotic}.
\begin{theorem}[{\cite[Prop.\ 3.3]{mazzeo2019asymptotic}}]
\label{thm_simple_zero_bijective}
If $q$ has only simple zeros, then $\BMMHq=\mbM_q$.
\end{theorem}
In other words, the compactification of any slice where $q$ does not lie in
the discriminant locus is ``the obvious one''. 

\section{Parabolic modules and stratification of BNR data}
\label{sec_parabolic_modules}
In this section, we review the notion of a parabolic module, as
described in \cite{rego1980compactified,cook1993local,cook1998compactified,
gothen2013singular}. This concept leads to a partial normalization of the
generalized Jacobian and Prym varieties of the spectral curve.
\subsection{Normalization of the spectral curve}
Let $q\neq 0$ be a quadratic differential with an irreducible,  singular spectral curve $S=S_q$. The zeros of $q$ define a divisor 
$\Div(q) = \sum_{i=1}^{r_1} m_i p_i + \sum_{j=1}^{r_2} n_j p_j'$, where the $m_i$ and $n_j$ are even and odd integers, respectively, 
and hence $r_1$ and $r_2$ are the numbers of even and odd zeros, respectively, counted without
multiplicity. Write $\Ze = \{p_1, \dots, p_{r_1}\}$, $\Zo = \{p_1', \dots, p_{r_2}'\}$, and $Z = \Ze \cup \Zo$, so $\#Z=r = r_1 + r_2$.

The map $\pi: S \to \Sigma$ is a double covering branched along $Z$; hence, we may view $p_i$ and $p_i'$ as points in $S$. 
For $x \in S$, let $\MO_x$ be the local ring, $\MO_x^\ast$ its group of units, and  $R_x$ the completion. 
We say that $S$ has an $A_n$ singularity at $x$ if $R_x \cong \mbC[[r,s]]/(r^2-s^{n+1})$, where $n \geq 1$. If $S$ has an $A_1$ 
singularity at $x$, we call it a \emph{nodal} singularity, and if $S$ has an $A_2$ singularity at $x$, we
call it a \emph{cusp} singularity.

Let $p:\tS \to S$ be the normalization of $S$, and let $\tpi := \pi \circ p$: 
\[
\begin{tikzcd}
	\tS \arrow{r}{p} \arrow[swap]{dr}{\tpi} & S \arrow{d}{\pi} \\
	& \Sigma
\end{tikzcd}
\]
For even zeros $p_i$  we write $p^{-1}(p_i) = \{\tp_i^{+}, \tp_i^{-}\}$, and for odd zeros $p_i'$ we write $p^{-1}(p_i') = \tp_i'$. Since 
$\pi:S\to\Sigma$ is a branched double cover, the involution $\sigma$ on $S$ extends to an involution of $\tS$ which we also denote 
by $\sigma$. Note that $\sigma(\tp_i') = \tp_i'$ while $\sigma(\tp_i^\pm) = \tp_i^\mp$.

The ramification divisor
$\Lam'=\frac{1}{2}\sum_{i=1}^{r_1}m_ip_i+\sum_{j=1}^{r_2}(n_j-1)p_j''$,
is a divisor on $S$, and there is an exact sequence: 
\begin{equation} \label{eq_normalization_exact_sequence}
    0\lra \MO_{S}\lra p_{\st}\MO_{\tS}\lra \MO_{\Lam'}\lra 0\ .	
\end{equation}
The genus of $\tS$ is $g(\tS)=4g-3-\deg(\Lam')=2g-1+r_2/2$. 

\subsection{Jacobian under the pull-back of the normalization}
We now recall some facts about the Jacobian under the pull-back of the normalization, cf.\ \cite{gothen2013singular}. Let $x \in Z
\subset S$ be a singular point, i.e.\ either $x\in\Ze$ or $x=p'_j$ with
$n_j\geq 3$.   Let  $\tMO_{x}$ be the
integral closure of $\MO_x$. We take $V:=\prod_{x\in
Z}\tMO_x^{\st}/\MO_x^{\st}$. Then we have the following well-known short
exact sequence induced by the pull-back of the normalization.
\begin{equation} \label{eq_normalization_Jacobian_fibration}
	\begin{split}
		0\lra V\lra \Jac(S)\xrightarrow{p^{\st}} \Jac(\tS)\lra 0\ .
	\end{split}
\end{equation}
This will play an important role later on. 

\subsubsection{Hitchin fiber}
Now we examine the locally free part $\MT$ of the Hitchin fiber under
the pull-back. Here, $\MT$ is defined to be the set of $L\in\Pic^{2g-2}(S)$ such that 
$\det(\pi_{\st}L)=\MO_{\Sigma}$.  For any $L\in\Pic(S)$, from \eqref{eq_normalization_exact_sequence} we 
see that $\det(\tpi_{\st}p^{\st}L)\cong \det(\pi_{\st}L)\otimes \MO_\Sigma(\Lam')$. We define a new set, $\tMT$, as follows:
\begin{equation*}
	\tMT:=\{\tL\in \Pic^{2g-2}(\tS)\mid\det(\tpi_{\st}L)\cong\MO(\Lam')\}.
\end{equation*}
It follows that $p^{\st}$ maps $\MT$ to $\tMT$. Furthermore, if $L_1,L_2\in \Pic(S)$ satisfy $p^{\st}L_1\cong p^{\st}L_2$, then we have $\pi_{\st}L_1\cong \pi_{\st}L_2$. This means that the fiber of $p^{\st}:\Jac(S)\to \Jac(\tS)$ is the same as $p^{\st}:\MT\to \tMT$, resulting in the following exact sequence:
\begin{equation} \label{eq_exactsequenceofT}
	0\longrightarrow
    V\longrightarrow\MT\xrightarrow{\hspace{.4cm}p^{\st}\ }\tMT\longrightarrow  0\ .
\end{equation}
\subsection{Torsion free sheaves}
\label{subsection_torsion_free_sheaf_integers}
Now we present Cook's parametrization of rank 1 torsion free sheaves on
curves with
Gorenstein singularities (see  \cite[p.\ 40]{cook1998compactified} and also
\cite{cook1993local, rego1980compactified}). An explicit computation of the
invariants used in this paper is provided in Appendix \ref{appendixA}.
Let $x \in Z$ be a singular point of $S$, and let $L\to S$ be a rank 1 torsion free sheaf.
We again let $\MO_x$ denote the local ring at $x$, $\tMO_x$ its integral closure,
$\MK_x$ its fraction field,
and $\delta_x = \dim_{\mbC}(\tMO_x/\MO_x)$. 
According to \cite[Lemma 1.1]{greuel1993moduli}, there exists a fractional
ideal $I_x$ that is isomorphic to $L_x$, uniquely defined up to multiplication
by a unit of $\tMO_x$, such that $\MO_x \subset I_x \subset \tMO_x$. 
We define $\ell_x := \dim_{\mbC} (I_x/\MO_x)$ and $b_x:=\dim_{\mbC}(\Tor(I_x\otimes_{\MO_x}\tMO_x))$. 
Then, $\ell_x$ and $b_x$ are invariants of $L$.

Let $\MK_x$ be the field of fractions of $\MO_x$. The \emph{conductor} of
$I_x\subset \tMO_x$ is defined to be 
$$C(I_x)=\{u\in\MK_x\mid u\cdot\tMO_x\subset I_x\}\ .$$
The singularity is characterized by the following dimensions:
\begin{equation*}
 	\underbrace{C(\MO_x)\subset\overbrace{C(I_x)\subset
    \overbrace{\MO_x\subset \underbrace{I_x\subset
    \tMO_x}_{\delta_x-\ell_x}}^{\delta_x}}^{2\delta_x-2\ell_x}}_{2\delta_x}\ .
\end{equation*}
For $x = p_i \in \Ze$, we have $\delta_{p_i}=m_i/2$, and there are
two maximal ideals $\mfm_{\pm}$ in $\tMO_x$ corresponding to the points
 $\tp_i^{\pm}$. We let $(\tMO_{p_i}/C(I_{p_i}))_{\mfm_{\pm}}$
be the localization by the ideals $\mfm_{\pm}$, and define
$a_{\tp_i^{\pm}}:=\dim_{\mbC}(\tMO_{p_i}/C(I_{p_i}))_{\mfm_{\pm}}$.
Moreover, we have
$\dim_{\mbC}(\tMO_{p_i}/C(\MO_{p_i}))_{\mfm_{\pm}}=m_i/2=\delta_{p_i}$.
By Appendix \ref{appendixA}, 
$a_{\tp_i^{\pm}}=(m_i/2)-\ell_{p_i}$, and therefore $a_{\tp_i^+}+a_{\tp_i^-}=2\delta_{p_i}-2\ell_{p_i}$, and also 
$b_{p_i}=\ell_{p_i}$. Define 
$$
V(L_{p_i})=\{(c_i^+,c_i^-)\mid c_{i}^{\pm}\in\mathbb{Z}_{\ge 0}\ ,\ c_i^++c_i^-=\ell_{p_i}\}\ .
$$
For $x=p_i' \in \Zo$, we have $\delta_{p_i'}=(n_i-1)/2$, and the maximal ideal $\mfm$ of $\tMO_x$ is unique.
Define $a_{\tp_i'}:=\dim_{\mbC}(\tMO_{p_i'}/C(I_{p_i'}))_{\mfm}$. By Appendix \ref{appendixA}, 
we have $a_{\tp_i'}=2\delta_{p_i'}-2\ell_{p_i'}$ and $b_{p_i'}=\ell_{p_i'}$. Moreover, 
$\dim_{\mbC}(\tMO_{p_i'}/C(\MO_{p_i'}))_{\mfm}=n_i-1=2\delta_{p_i'}$.  In this case we set $V(L_{p_i'})=\{\ell_{p_i'}\}$.

Now consider modules compatible with  $L_x$. Let $\eta:\tMO_x\to \tMO_x/C(\MO_x)$ be the quotient map.
Define 
$$
S(L_x):=\{\MO_x\text{-submodules}\ F_x\subset \tMO_x/C(\MO_x)\mid \dim_{\mathbb
C}(F_x)=\delta\ ,\ \eta^{-1}(F_x)\cong L_x\}\ .
$$
Hence, if $J_x=\eta^{-1}(F_x)$ with $F_x\in S(L_x)$, there exists an ideal $\mft_x$ such that  
$J_x=\mft_x\cdot L_x$. For $x=p_i\in \Ze$, we obtain two integers $c_{i}^{\pm}=\dim_{\mbC}(\tMO_x/(\mft_x\cdot \tMO_x))_{\mfm_{\pm}}$.
By \cite[Lemma 6]{cook1998compactified}, $(c_i^+,c_i^-)\in V(L_{p_i})$, for $x=p_i'\in \Zo$, 
$\dim_{\mbC}(\tMO_x/(\mft_x\cdot \tMO_x))=\ell_{p_i'}\in V(L_{p_i'})$, 
and these only depend on $F_x$. Hence, there is a well-defined map:
\begin{equation*}
    \kappa_x:S(L_x)\longrightarrow V(L_x)\ :\ \begin{cases} F_x\to
        (c_i^+,c_i^-)&\mathrm{when}\;x=p_i,\;\\
    F_x\to \ell_{p_i'}&\mathrm{when}\;x=p_i' \ . \end{cases}
\end{equation*}
\begin{lemma}[{\cite[Lemma 6]{cook1998compactified}}]
For $x\in Z$, the components of $S(L_x)$ are parameterized by elements in $V(L_x)$.
\end{lemma}

Set $V(L):=\prod_{x\in Z}V(L_x)$ and $S(L):=\prod_{x\in Z}S(L_x)$.  Write $N(L) := |V(L)|$ for the number of points in $V(L)$. 
 there is a map
\begin{equation*}
\kappa:=\prod_{x\in Z}\kappa_x: S(L)\lra V(L)\ .
\end{equation*}
For any $\mbfc\in V(L)$, write $\mbfc=(c_1^{\pm},\ldots,c_{r_1}^{\pm},\ell_{p_1'},\ldots,\ell_{p_{r_2}'})$. Associate to this the 
divisor 
\begin{equation*}
 D_{\mbfc}=\sum_{i=1}^{r_1}(c_i^+\tp_i^++c_i^-\tp_i^-)+\sum_{i=1}^{r_2}\ell_{p_i'}\tp_i'
\end{equation*}
on $\tS$. 
Composing $\kappa$ with the map above, we define 
\begin{equation} \label{eq_divisor_map}
\varkappa: S(L)\lra \Div(\tS)\ :\	\prod_{x\in Z}F_x \mapsto \mbfc\mapsto D_{\mbfc}\ .
\end{equation}

The following result is straightforward but important:
\begin{proposition}
$L$ is locally free if and only if $\varkappa=0$ on $S(L)$.
\end{proposition}
\begin{proof}
$L$ is locally free if and only if $\ell_x=0$ for $x\in Z$. The claim then follows directly from the definition of $D_{\mbfc}$.
\end{proof}

\subsection{Parabolic modules}
In this subsection, we define the notion of a parabolic module, following \cite{rego1980compactified,cook1993local,cook1998compactified}.

First note that $\dim_{\mbC}(\tMO_x/C(\MO_x))=2\delta_x$. 
Let $\Gr(\delta_x,\tMO_x/C(\MO_x))$ be the Grassmannian of $\delta_x$ dimensional subspaces 
of the vector space $\tMO_x/C(\MO_x)$. Then $\tMO^{\st}_x$ acts on $\Gr(\delta_x,\tMO_x/C(\MO_x))$ 
by multiplication, with fixed points corresponding to $\delta_x$-dimensional $\MO_x$ submodules of $\tMO_x/C(\MO_x)$. 
We write $\MP(x)$ for the (reduced) variety of fixed points. This is a closed subvariety of $\Gr(\delta_x,\tMO_x/C(\MO_x))$.

Suppose $x$ is an $A_n$ singularity. For notational convenience, we write $\MP(A_n):=\MP(x)$. We have the following:
\begin{proposition}[{\cite[Prop.\ 2]{cook1998compactified}}]
	\label{prop_fiber_of_parabolic_module}
    The following holds:
	\begin{itemize}
		\item [(i)] $\MP(A_n)$ is connected and depends only on $\delta_x$.
            Also, $\dim_{\mbC}\MP(A_n)=n$, and we have isomorphisms $\MP(A_{2n-1})\cong \MP(A_{2n})$.
		\item [(ii)] If $P(A_0)$ is defined to be a point, then the
            inclusions $\MP(A_0)\subset \MP(A_2)\subset \cdots \subset \MP(A_{2n})$ give a cell decomposition of $\MP(A_{2n})$.
		\item [(iii)] The singular locus $\mathrm{Sing}(\MP(A_{2n}))\cong
            \MP(A_{2n-4})$. In particular, it
            has codimension $\geq 2$. Moreover, $\MP(A_1)=\MP(A_2)\cong
            \CP^1$, and $\MP(A_4)$ is a quadric cone.
	\end{itemize}
\end{proposition}
Define $\MSP(S)=\prod_{x\in Z}\MP(x)$. This only depends on the curve singularity of $S$.
Let $J\in\Pic(\tS)$. We have an
isomorphism $J_x\otimes \MO_{\Lam',x}\cong \tMO_x/C(\MO_x)$ as
$\MO_x$-modules. More explicitly,  as vector spaces,
$$
J_{\tp_i^+}^{\oplus \frac{m_i}{2}}\oplus J_{\tp_i^-}^{\oplus \frac{m_i}{2}}\cong
\tMO_{p_i}/C(\MO_{p_i'})\ , \ J_{\tp_i'}^{\oplus (n_i-1)}\cong \tMO_{p_i'}/C(\MO_{p_i'})\ .
$$

\begin{definition}
A parabolic module $\PMod$ consists of pairs $(J,v)$, where $J\in
\Jac(\tS)$ and $v=\prod_{x\in Z} v_x$, with $v_x\in J_x\otimes \MO_{\Lam',x}$.
\end{definition}
By \cite[p.\ 41]{cook1998compactified}, $\PMod$ has a natural algebraic structure. Let $\pr:\PMod\to \Jac(\tS)$ be the projection 
to the first component. Then $\pr$ defines a fibration of $\PMod$ with fiber $\MSP(S)$. 
Moreover, there is a finite morphism $\tau: \PMod\to \BJac(S)$ defined by sending $(J,v) \to L$,
where $L$ is  the kernel of the restriction map $p_{\st}J\to (J\otimes \MO_{\Lam})/v$:
$$
0\lra L\lra p_{\st}J\lra (J\otimes\MO_{\Lam})/v\lra 0\ .
$$
There is a diagram:
\begin{equation*}
	\begin{tikzcd}
		&\MSP(S) \arrow[r, ] & \PMod \arrow[r, "\pr"] \arrow[d, "\tau"]
		& \Jac(\tS)  \\ 
		& &\BJac(S) &  &
	\end{tikzcd}.
\end{equation*}
The map $\tau$ may be regarded as the compactification of the pull-back normalization map $p^{\st}$ in \eqref{eq_normalization_Jacobian_fibration}.

\begin{theorem}[{\cite[Thm.\ 1]{cook1998compactified}}]
\label{thm_parabolic_module_main_theorem}
For the map $\tau:\PMod\to \BJac(S)$ defined above, 
\begin{itemize}
\item[(i)] $\tau$ is a finite morphism, where the fiber over $L$ consists of $N(L)$ points,
\item[(ii)] The restriction $\tau: \tau^{-1}\Jac(S)\to \Jac(S)$ is an isomorphism. Moreover, for $L\in\Jac(S)$, we have $\pr\circ \tau^{-1}(L)=p^{\st}(L)$.
\item [(iii)] Suppose $\tau(J,v)=L$. For $x\in Z$, we have $v\in S(L)$. 
Let $D_{v}=\varkappa(v)$ be the divisor defined in \eqref{eq_divisor_map}. Then 
\begin{equation*}
0\lra p^{\st}L/\Tor(p^{\st}L)\lra J\lra \MO_{D_v}\lra 0\ .
\end{equation*}
In particular, $p^{\st}L/\Tor(p^{\st}L)=J(-D_v)$.
\end{itemize}
\end{theorem}

Suppose all of the zeros of the quadratic differential $q$ are odd. 
Then for $L\in \BJac(S)$, $N(L)=1$, and we can deduce the following.
\begin{corollary}
If $q^{-1}(0)=\{p_1',\ldots, p_r'\}$ and all zeroes have odd multiplicity, then $\tau:\PMod\to \BJac(S)$ is a bijection. 
Moreover, for $L\in \BJac(S)$ with $\tau(J,v)=L$, we have 
$$
p^{\st}L/\Tor(p^{\st}L)=J(-\sum \ell_{p_i'}\tp_i')\ .
$$
\end{corollary}

We will now present an example of a parabolic module.

\begin{example}[{\cite[Ex.\ 2]{cook1998compactified}}]
Suppose $q$ contains $4g-2$ simple zeros and one zero $x$ of order $2$.
    Then the spectral curve $S$ has one nodal singularity at $x$. 
    Denote $p:\tS\to S$ the normalization, with $p^{-1}(x)=\{\tx_+,\tx_-\}$. Then 
    $\MSP(S)=\CP^1$, and we obtain a fibration $\CP^1\to \PMod\to \Jac(\tS)$. 
    Let $L\in\BJac(S)$. If we write $\tL:=p^{\st}L/\Tor(p^{\st}L)$, then 
    $\tau^{-1}(L)={(\tL\otimes \MO(\tx_+),v_+),(\tL\otimes \MO(\tx_-),v_-)}$. 
We can define two sections:
$$s_{\pm}:\BJac(S)\lra \PMod,\;s_{\pm}:L\lra (\tL\otimes
    \MO(\tx_{\pm}),v_{\pm})\ .$$
Then $\BJac(S)$ is the quotient of $\PMod$ given by the identification 
    $$\BJac(S)\cong \PMod/(s_+\sim \MO(\tx_--\tx_+)s_-)\ .$$ In 
    particular, $\PMod$ is not a fibration over $\BJac(S)$.
\end{example}

\begin{proposition}
The singular set of $\PMod$ has codimension at least $2$. Moreover, if the spectral curve $S$ contains only cusp or nodal singularities, then $\PMod$ is smooth.
\end{proposition}

\begin{proof}
	As the singularities of $\PMod$ come from the space $\MSP(S)$, the claim follows from Proposition \ref{prop_fiber_of_parabolic_module}.
\end{proof}

Since we focus on $\SLC$ Higgs bundles, we must consider the parabolic module  compactification of the fibration
 $$
0\longrightarrow V\longrightarrow \MP\xrightarrow{\hspace{.3cm}p^{\st}\ } \Prym(\tS/\Sigma)\longrightarrow 0\ .
$$
Setting, $\hPMod:=\tau^{-1}(\BMP)$, then there is a diagram
from \cite[p.\ 17]{gothen2013singular}
\begin{equation} \label{eq_parabolic_module_for_Prym}
	\begin{tikzcd}
		&\MSP(S) \arrow[r, ] & \hPMod \arrow[r, "\pr"] \arrow[d, "\tau"]
		&  \Prym(\tS/\Sigma)  & \\ 
		& &\BMP &  &
	\end{tikzcd}
\end{equation}
Theorem \ref{thm_parabolic_module_main_theorem} proves that $\pr\circ \tau^{-1}|_{\MP}=p^{\st}$. 

\subsection{Stratifications of the BNR data}
\label{sec:divisor-stratification}
Parabolic modules define a stratification of $\BMP$ and $\BMT$.  In the following, $\pi:S\to \Sigma$ 
is a branched double cover, 
$\sigma$ the associated involution on $S$,  and by  $\sigma$ we also denote
its extension to an involution on the normalization  $\tS$ of $S$.

For a rank 1 torsion free sheaf $L\in\BPic(S)$, consider the map
\begin{equation*}
	\ptf:\BPic(S)\lra \Pic(\tS)\ ,\ \ptf(L):=p^{\st}L/\Tor(p^{\st}L)\ ,
\end{equation*} 
i.e.\ the torsion free part of the  pull-back to the normalization. 
By \cite{rabinowitz1979monoidal}, $\ptf(L)=p^{\st}L$ at $x\in \tS$ if and only if 
$L$ is locally free at $p(x)\in S$.

\begin{definition}[{\cite{horn2022semi}}]
	An effective divisor $D\in \Div(\tS)$ is called a $\sigma$-divisor if 
	\begin{itemize}
		\item [(i)] $D\leq \Lam$ and $\sigma^{\st}D=D$;
		\item [(ii)] and for any  $x\in \Fix(\sigma)$, $D|_x=d_xx$, where $d_x\equiv 0\mod 2$.
	\end{itemize}
\end{definition}

The $\sigma$-divisors play an important role in describing the singular
Hitchin fibers. 
\begin{proposition}[{\cite{horn2022semi,Mochizukiasymptotic}}]
	\label{prop_stratification_BNR_data}
	Let $L\in \BMP$ and  write $\tL:=\ptf L$. Then we have $\tL\otimes \sigma^{\st}\tL=\MO(\Lam-D)$ for $D$ a $\sigma$-divisor.
\end{proposition}

For  a $\sigma$-divisor $D$, define 
\begin{align}
	\begin{split}
        \tMT_D&=\{J\in \Pic(\tS)\mid J\otimes
        \sigma^{\st}J=\MO(\Lam-D)\}\ ; \\
        \tMP_D&=\{J\in\Pic(\tS)\mid J\otimes\sigma^{\st}J=\MO(-D)\}\ .
	\end{split}	
\end{align}
Then by \cite[Prop.\ 5.6]{horn2022semi}, $\tMT_D$ and $\tMP_D$
are abelian torsors over $\Prym(\tS/\Sigma)$ with dimension $g(\tS)-g(S)=g-1+\frac12r_2$. 
In addition, we define
\begin{align}
    \begin{split}
        \BMT_D&=\{L\in \BMT\mid \ptf L\in \tMT_D\}\ ;\\
        \BMP_D&=\{L\in\BMP\mid\ptf L\in \tMP_D\}\ .
    \end{split}
\end{align}
Then the partial order on divisors defines a stratification of $\BMT$ (resp.\ $\BMP$)
by: $\cup_{D'\leq
D}\BMT_{D'}$ (resp.\ $\cup_{D'\leq D}\BMP_{D'}$). The  top strata  are
$\BMT_{D=0}$ (resp.\ $\BMP_{D=0}$), and these consist of  the locally free
sheaves. From the definition, $\MT=\BMT_{D=0}$ and $\MP=\BMP_{D=0}$. 

\begin{theorem}[{\cite[Thm.\ 6.2]{horn2022semi}}]
\label{thm_stratification_fibration}
 (i)	Suppose $q$ contains at least one zero of odd order. Then for each
    stratum indexed by $\sigma$-divisor $D$, if we let $n_{ss}$ be the 
number of $p$ such that $D|_p=\Lam|_p$, then there are holomorphic fiber bundles
	\begin{align}
        \begin{split}
        (\mbC^{\st})^{k_1}\ti \mbC^{k_2}\longrightarrow
        \BMT_D\xrightarrow{\hspace{.2cm}\ptf\ }
        \tMT_D\ ;\\
        (\mbC^{\st})^{k_1}\ti \mbC^{k_2}\longrightarrow
        \BMP_D\xrightarrow{\hspace{.2cm}\ptf\ } \tMP_D\ ,
	\end{split}
    \end{align}
	where $k_1=r_1-n_{ss}$, $k_2=2g-2-\frac12\deg(D)-r_1+n_{ss}-\frac {r_2}{2}$, and $r_1,r_2$ are the number of even 
    and odd zeros. 
    (ii)
	Suppose $q$ is irreducible but all zeros are of even order. Then there exists an
    analytic space $\BMT_D'$ and a double branched covering $p:\BMT_D\to
    \BMT_D' $, with $\BMT_D'$ a holomorphic fibration
	\begin{equation*}
			(\mbC^{\st})^{k_1}\ti \mbC^{k_2}\longrightarrow
            \BMT_D'\xrightarrow{\hspace{.25cm}\ptf\ } \tMT_D\ .
	\end{equation*}
	In particular, $\dim(\BMP_D)=\dim(\BMT_D)=3g-3-\frac12\deg(D).$
\end{theorem}

As explained in \cite{horn2022semi},  via the BNR correspondence the stratification above translates
into a stratification of the Hitchin  fiber.
Let $\chi_{\BNR}:\BMT\isorightarrow\MM_q$ be the bijection in Theorem
\ref{thm_BNRcorrespondence}. Let $D$ be a $\sigma$-divisor. Define
$\MM_{q,D}:=\chi_{\BNR}(\BMT_D)$. Then the stratification of $\BMT$ 
induces a stratification on $\MM_q=\bigcup_D\MM_{q,D}$. 

For each $\sigma$-divisor $D$, since $\sigma^{\st}D=D$, we can write $D':=\frac{1}{2}\tpi(D)$. 
Then $D'$ is an effective divisor with $\supp D'\subset Z$. Moreover, for
$x\in q^{-1}(0)$, $D_x'\leq \frac12\lfloor \ord_x(q) \rfloor$.
Therefore, $\MM_q$ may be regarded as also being stratified by divisors $D'$ defined over $\Sigma$.

\subsection{The structure of the parabolic module projection}
We now explain the relationship between the divisor $D_v$ in Theorem \ref{thm_parabolic_module_main_theorem} and the $\sigma$-divisor. 
Given $L \in \BMP$, define
\begin{align}
    \begin{split}
        \MSN_L&:=\{(J,v)\in\hPMod\mid\tau(J,v)=L\}\ ;\\
        \MSD_L&:=\{D_v\mid(J,v)\in\MSN_L\}\ ,
\end{split}
\end{align}
where $\MSN_L$ is $\tau^{-1}(L)$, and $\MSD_L$ is the collection of divisors $D_v$ 
such that $J(-D_v) = \ptf(L)$. If $L$ is locally free, then $J = p^{\ast}
L$, and $\MSD_L$ is empty. Moreover, if $\tau(J, v) = \tau(J', v)$, then $J' = J(D_{v'} - D_v)$.

The divisor $D_v$ satisfies the following symmetry property:
\begin{proposition}
	\label{prop_relationship_D_v_and_D}
Let $D$ be a $\sigma$-invariant divisor and $L \in \BMP_D$. For any $D_v \in \MSD_L$, we have $D_v + \sigma^{\ast} D_v = D$.
\end{proposition}
\begin{proof}
Let $\tau(J, v) = L$.
    Then by Theorem \ref{thm_parabolic_module_main_theorem}, we have $\tL =
    J(-D_v)$, where $\tL = \ptf(L)$. As $L \in \BMP_D$ and $J \in
    \Prym(\tS/\Sigma)$, we have $\tL \otimes \sigma^{\ast} \tL = \MO(-D)$ and $J \otimes \sigma^{\ast} J = \MO_{\tS}$, which implies $D_v + \sigma^{\ast} D_v = D$.
\end{proof}

As a consequence, we have the following:
\begin{corollary}
Suppose $q$ has only zeroes of odd order. Then for $L \in \BMP_D$ and $D_v \in \MSD_L$, we have $\sigma^{\ast} D_v = D_v$ 
and $D_v = \frac{1}{2}D$. In addition, $\tau: \hPMod \to \BMP$ is a bijection.
\end{corollary}
\begin{proof}
Since each zero has odd order  $\supp(D_v) \subset \Fix(\sigma)$, which implies $D_v = \sigma^{\ast} D_v$. 
By Proposition \ref{prop_relationship_D_v_and_D}, we must have $D_v = \frac{1}{2}D$.
\end{proof}

There are relationships between the integers appearing in the construction of the parabolic module: 
\begin{lemma}[\cite{Greuel1985}]
\label{lem_integers_in_parabolic_module}
Let $D=\sum_{i=1}^{r_1} d_i(\tp_i^+ + \tp_i^-) + \sum_{i=1}^{r_2} d_{i}'\tp_i'$ be a $\sigma$-divisor, and let $L\in \BMP_D$. Then we have
\begin{itemize}
\item[(i)] $\ell_{p_i}=d_i$ and $\ell_{p_i'}=d_i'/2$;
\item[(ii)] $a_{\tp_i^+}=a_{\tp_i^-}=(m_i/2)-d_i$ and $a_{\tp_i'}=n_i-1-d_i'$.
\end{itemize}
\end{lemma}
\begin{proof}
Since $L\in \BMP_D$, we have $\dim\Tor(p^{\st}L_{p_i})=d_i$ and $\dim\Tor(p^{\st}L_{p_i'})=d_i'/2$. The claim then follows
from Proposition \ref{prop_appendix_computation}.
\end{proof}

The elements in $\MSD_L$ can be explicitly computed.
\begin{proposition}
\label{prop_computation_NL}
Let $D=\sum_{i=1}^{r_1} d_i(\tp_i^+ + \tp_i^-) + \sum_{i=1}^{r_2} d_{i}'\tp_i'$ be a $\sigma$-divisor, and let $L\in \BMP_D$. Then
$N_L=\prod_{i=1}^{r_1}(d_i+1)$. The number $N_L$ depends only on the $\sigma$-divisor $D$.
\end{proposition}
\begin{proof}
By Lemma \ref{lem_integers_in_parabolic_module}, $V(L)$ can be rewritten as 
\begin{equation*}
\begin{split}
V(L)=\{(c_1^{\pm},\ldots,c_{r_1}^{\pm},c'_{1}=l_{p_1'},\ldots,c'_{r_2}=l_{p'_{r_2}})\mid c_i^++c_i^-=d_i,c_i^{\pm}\in\mathbb{Z}_{\geq  0}\}\ ,
\end{split}
\end{equation*}
which implies the claim. The condition $D_v+\sigma^{\st}D_v=D$ is automatically satisfied.
\end{proof}

If we define $n_L$ to be the number of $D_v\in\MSD_L$ such that $\sigma^{\st}D_v\neq D_v$, then we have the following:
\begin{proposition}
	\label{prop_computation_nL}
\begin{itemize}
\item [(i)] $n_L$ is even;
\item [(ii)] if $L\in \BMP_D$ with 
$$
D=\sum_{i=1}^{r_1}d_i(\tp_i^++\tp_i^-)+\sum_{i=1}^{r_2}d_{i}'\tp_i'\ ,$$
and if there exists $i_0\in\{1,\dots,r_1\}$ such that $d_{i_0}$ is not even, then $n_L=N_L$;  otherwise, $n_L=N_L-1$.
\end{itemize}
\end{proposition}
\begin{proof}
	To prove (i), note that if $\sigma^{\st}D_v\neq D_v$, then
    $\sigma^{\st}(\sigma^{\st}D_v)\neq \sigma^{\st}D_v$, which means that
    $n_L$ is even. For (ii), by Proposition \ref{prop_computation_NL},
    $D_v=\sigma^{\st}D_v$ for $D_v\in\MSD_L$ if and only if
    $c_{i}^+=c_i^-=d_i/2$. Therefore, $n_L\neq N_L$ if and only if all $d_i$ are even, which implies (ii).
\end{proof}

\section{Irreducible singular fibers  and the Mochizuki map}
\label{sec_irreducible_singular_fiber}
In this section, we provide a reinterpretation of the limiting configuration construction of a Higgs bundle 
on an irreducible fiber, as  introduced by Mochizuki in \cite{Mochizukiasymptotic} (see also \cite{horn2022semi}). 
We also investigate the relationship between limiting configurations and the stratification.

\subsection{Abelianization of a Higgs bundle}
\label{subsection_abelianzation}
Let $q$ be a fixed irreducible quadratic differential with spectral curve $S$, with normalization $p:\tS\to S$.
We define $\widetilde{K}:=\widetilde{\pi}^{\ast}K$ (but note that $\widetilde{K}\neq K_{\widetilde S}$) and 
$\widetilde{q}:=\widetilde{\pi}^{\ast}q\in H^0(\widetilde{K}^2)$, where $\widetilde{\pi}$ is the double 
branched covering of $\widetilde S\to\Sigma$ associated to the branching set $Z$ of $q$. 
Choose a square root $\omega\in H^0(\widetilde{K})$ such that $\widetilde{q}=-\omega\otimes\omega$. 
Let $\Lambda:=\mathrm{Div}(\omega)$ and $\tZ:=\mathrm{supp}(\Lambda)$. We can then write
$$\Lam=\sum_{i=1}^{r_1}\frac{m_i}{2}(\tp_i^++\tp_i^-)+\sum_{j=1}^{r_2}n_j\tp_j'\ .
$$
If $\sigma:\tS\to \tS$ is the involution, then $\sigma^{\st}\om=-\om$.

Let $(\mathcal{E},\varphi)$ be a Higgs bundle  with $\det\varphi=q$. Consider the pullback
$(\tME,\tvp):=(\widetilde{\pi}^{\ast}\mathcal{E},\widetilde{\pi}^{\ast}\varphi)$. We have $\tvp\in H^0(\mathrm{End}(\tME)\otimes\widetilde{K})$ 
and $\widetilde{q}=\det(\tvp)$. Since $\widetilde{q}=-\omega\otimes\omega$, $\pm\omega$ are well-defined eigenvalues 
of $\tvp$ over $\tS$. Let $\tilde{\lambda}$ be the canonical section of $\mathrm{Tot}(\widetilde{K})$. 
The spectral curve for $(\tME,\tvp)$ is defined by the equation
$$
\tS':=\{\tlam^2-\tq=0\}.
$$
The set $\tS'=\mathrm{Im}(\omega)\cup\mathrm{Im}(-\omega)\subset \mathrm{Tot}(\widetilde{K})$ decomposes into two irreducible pieces.

Having fixed a choice of $\omega$, the eigenvalues of $\tvp$ are globally well-defined, and we can define the line bundle 
$\tL_+\subset \mathcal{E}$ as $\tL_+:=\ker(\tvp-\omega)$.  Since $\sigma^{\st}\om=-\om$, $\tL_-=\sigma^{\ast}\tL_+=\ker(\tvp+\omega)$, 
and there is an isomorphism $\tME|_{\tS\setminus \tZ}\cong \tL_+\oplus \tL_-|_{\tS\setminus \tZ}$. 

There is a local description of $(\tME,\tvp)$:
\begin{lemma}[{\cite[Lemma 5.1, Thm.\ 5.3]{horn2022semi},\cite[Lemma
    4.2]{Mochizukiasymptotic}}]
	\label{lem_local_description}
	Let $x\in \tZ$ and write $\Lambda|_x=m_xx$. Let $U$ be a holomorphic
    coordinate neighborhood of $x$. Then there exists a frame $\mathfrak{e}\in H^0(U,\widetilde{K})$ such that, under a suitable trivialization of $\mathcal{E}|_U\cong U\times \mathbb{C}^2$, we can write
	\begin{equation} \label{eq_local_description_Higgs_bundle}
		\textrm{ }\tvp=z^{d_x}\begin{pmatrix}
			0 & 1\\
			z^{2m_x-2d_x} & 0
		\end{pmatrix} \otimes \mathfrak{e}.
	\end{equation}
	Moreover, if we define $D:=\sum_{x\in\tZ}d_xx$, then $D$ is a $\sigma$-divisor.
\end{lemma}

\begin{lemma}[{\cite[Sec.\ 4.1]{Mochizukiasymptotic}}]
	\label{lem_exact_sequence_abelianization}
	For the $\tL_{\pm}$ defined above, we have $\tL_+\otimes \tL_-=\MO(D-\Lam)$. Moreover, if we denote $\tL_0:=\tL_+(\Lam-D)$ and $\tL_1:=\sigma^{\st}\tL_0$, then 
	$\tL_+=\tME\cap \tL_0$, $\tL_-=\tME\cap \tL_1$, and we have the exact sequences 
    \begin{align*}	
	0\longrightarrow \tL_+\longrightarrow \tME \longrightarrow
        \tL_1\longrightarrow 0\ ;\\
        0\longrightarrow\tL_-\longrightarrow \tME \longrightarrow
        \tL_0\longrightarrow 0\ .
    \end{align*}	
\end{lemma}
\begin{proof}
The inclusion of $\tL_{\pm}\to \ME$ defines an exact sequence of sheaves 
$$
0\longrightarrow \MO(\tL_+)\oplus \MO(\tL_-)\longrightarrow
    \MO(\ME)\longrightarrow \MT\longrightarrow 0\ ,
$$	where $\MT$ is a torsion sheaf with $\supp \MT\subset \tZ$. From the local description in 
    \eqref{eq_local_description_Higgs_bundle}, in the same trivialization, $\tL_{\pm}$ are spanned by the bases $s_{\pm}=\begin{pmatrix}
	1 \\
	\pm z^{m_x-d_x}
\end{pmatrix}.$ Therefore, as $\det(\ME)=\MO_\Sigma$, we obtain $\tL_+\otimes \tL_-=\MO(D-\Lam)$. 
Since $s_+,s_-$ are linear independent away from $z$, $\tME/\tL_+$ is
    locally generated  by the section $z^{d_x-m_x}s_-$. Therefore, $\tME/\tL_+\cong \tL_-(\Lam-D)=\tL_1$. Using the involution, we obtain the other exact sequence. 
\end{proof}

Therefore, if $\tL\otimes \sigma^{\st}\tL=\MO(D-\Lam)$, we have
$\tL_0=\tL(\Lam-D)\in \tMT_D$. In summary, the construction above leads us to consider the composition of the following maps:
\begin{equation*}
	\begin{split}
        \delta &: \MM_q \longrightarrow \Pic(\tS)\xrightarrow{\otimes \MO(\Lam-D)}\tMT_D\
        ;\\
		& (\ME,\vp)\mapsto \tL_+\mapsto \tL_+(\Lam-D)\ ,
	\end{split}
\end{equation*}
where the first map is given by taking the kernel of $(\tpi^{\st}\vp-\om)|_{\tpi^{\st}\ME}$. 

This construction is directly related to the torsion free pull-back. Recall that $\chi_{\BNR}:\BMT\to \MM_q$ is the BNR correspondence map,
and $\ptf:\BPic(S)\to\Pic(\tS)$ is the torsion free pull-back. Then we have
\begin{proposition}
	\label{prop_equivalent_torsionfree_pull_back}
	$\delta\circ\chi_{\BNR}=\ptf$. In particular, if $J\in \BMT_D$, then $\delta\circ\chi_{\BNR}(J)\in \tMT_D$. 
\end{proposition}
\begin{proof}
	Let $J\in \BMT$, and write $(\ME,\vp)=\chi_{\BNR}(J)$, $(\tME,\tvp):=\tpi^{\st}(\ME,\vp)$.	
    Recall the BNR exact sequence on $S$ (see \eqref{eq_BNR}):
\begin{equation*}
	\begin{split}
				0\longrightarrow J(-\Delta)\longrightarrow
                \pi^{\st}\ME\xrightarrow{\pi^{\st}\vp-\lam}\pi^{\st}\ME\otimes
                \pi^{\st}K\longrightarrow J\otimes
                \pi^{\st}K\longrightarrow 0\ .
	\end{split}
\end{equation*}
As $p^{\st}$ is right exact, we obtain 
    $$\tME\xrightarrow{\hspace{.3cm}\tvp-\tlam\ }\tME\otimes
    \tK\longrightarrow p^{\st}J\otimes \tK\longrightarrow
    0\ . $$
Since the spectral curve is $\tS'=\Im(\om)\cup\Im(-\om)$, we can consider the restriction to the component $\Im(\om)$ and 
    write $\tlam=\om,\;\tL_{\pm}:=\ker(\tvp\mp\om)$. We obtain an exact sequence 
\begin{equation*}
    0\longrightarrow\tL_+\longrightarrow\tME\xrightarrow{\hspace{.25cm}\tvp-\om\
    }\tME\otimes \tK\longrightarrow p^{\st}J\otimes \tK\longrightarrow 0\ ,
\end{equation*}
which breaks into short exact sequences
\begin{align*}
	0\longrightarrow \tL_+\longrightarrow \tME\longrightarrow
    \Im(\tvp-\om)\longrightarrow 0\ ;\\
    0\longrightarrow \Im(\tvp-\om)\longrightarrow \tME\otimes
    \tK\longrightarrow p^{\st}J\otimes \tK\longrightarrow 0\ .
\end{align*}
Using the local trivialization in Lemma \ref{lem_local_description}, $\Im(\tvp-\om)$ is locally spanned by $\begin{pmatrix}
z^{d_x}\\ -z^{m_x} \end{pmatrix}\mfe$. From Lemma \ref{lem_exact_sequence_abelianization}, if we write
$\tL_0:=\tL_+(\Lam-D)$ and $\tL_1:=\sigma^{\st}\tL_0$, then 
$$
\delta\circ\chi_{BNR}(J)=\tL_+(\Lam-D)\ .
$$ 
Moreover, there is an isomorphism $\Im(\tvp-\om)\cong \tL_1$. Letting $\tL_1'$ be the saturation of $\tL_1$,
then we obtain the commutative diagram:
\begin{equation*}
	\begin{tikzcd}
		0\arrow[r, ] & \tL_1 \arrow[r, ]  \arrow[d, "i"] & \tME\otimes\tK \arrow[r, ] \arrow[d, "\cong"] & p^{\st}J\otimes \tK \arrow[r, ] \arrow[d, ] & 0 \\ 
		0\arrow[r, ]& \tL_1' \arrow[r, ]&\tME\otimes \tK \arrow[r, ]&\ptf J\otimes \tK \arrow[r, ] &0
	\end{tikzcd}
\end{equation*}
where $i:\tL_1\to \tL_1'$ is the natural inclusion. Moreover, in the same
    trivialization, $\tL_1'$ is spanned by the section $\begin{pmatrix}
	1\\
	-z^{m_x-d_x}
\end{pmatrix}\mfe$. Therefore, $\tL_1'\cong \tL_-\otimes \tK$ and from
    Lemma \ref{lem_exact_sequence_abelianization}, $\ptf J=\delta\circ\chi_{\BNR}(J)$.
\end{proof}

 If $(\ME,\vp)$ is a Higgs
bundle with $(\ME,\vp)=\chi_{\BNR}(L)$, and $\tL_0=\delta\circ\chi_{BNR}(L)$, then by Proposition \ref{prop_equivalent_torsionfree_pull_back}, $\tL_0=\ptf(L)$. We define a Higgs bundle $(\tME_0,\tvp_0)$
as follows
\begin{equation*}
	\begin{split}
		\tME_0=\tL_0\oplus \sigma^{\st}\tL_0,\;\tvp_0=\begin{pmatrix}
			\om & 0 \\
			0 & -\om
		\end{pmatrix}.
	\end{split}
\end{equation*}
Moreover, $\tME$ is an $\MO_{\tS}$ submodule of $\tME_0$ with a natural inclusion $\iota:\tME\to \tME_0$ satisfying the following: 
\begin{itemize}
	\item [(i)] the induced morphism $\tME\to \tL_0,\;\tME\to \sigma^{\st}\tL_0$ is surjective,
	\item [(ii)] the restriction of $\iota|_{\tS\setminus \tZ}$ is an isomorphism,
	\item [(iii)] $\tvp_0\circ \iota=\iota\circ \tvp$.
\end{itemize}
Following \cite[Sec.\ 4.1]{Mochizukiasymptotic},
we call $(\tME_0,\tvp_0)$ the \emph{abelianization of the Higgs bundle} $(\ME,\vp)$.

\subsection{The construction of the algebraic Mochizuki map}
In this subsection, we define the algebraic Mochizuki map, as introduced in \cite{Mochizukiasymptotic}.
Recall that for any divisor $D=\sum_{x\in Z}d_x x$, there is a canonical weight function 
$$\chi_D(x):=\begin{cases}d_x  & 
x\in\supp D\ ;\\
    0 &x\notin \supp D\ .
\end{cases}
$$
We also have the stratification 
$\BMT=\cup_D\BMT_D$, for $\sigma$ divisors $D$. Let $\MSF(\tS)$ be the space of all degree zero filtered
line bundles over $\tS$. The \emph{algebraic Mochizuki map} $\TheMoc$ is defined as
\begin{equation*}
        \TheMoc : \BMT\lra \MSF(\tS)\ ,\
        L\mapsto \MF_{\st}(\ptf(L),\tfrac{1}{2}\chi_{D-\Lam})\ .
\end{equation*}

\begin{example}
When $q$ has only simple zeroes, this construction generalizes that of \cite{mazzeo2016ends} (see also \cite{fredrickson2018generic}). 
In the case of a quadratic differential with simple zeros, the spectral curve $S$ is smooth,  and every torsion free sheaf is locally free,
so that $\MT = \BMT$. If $Z = \{p_1, \dots, p_{4g-4}\}$ are the branch points of $S$, and 
$\Lambda = \sum_{i=1}^{4g-4} p_i$, then the weight function $\frac{1}{2}\chi_{-\Lambda}$ assigns a weight of $-\frac{1}{2}$ 
at each $p_i$. For $L\in \MT$, $\TheMoc(L)=\MF_{\st}(L,\frac12\chi_{-\Lam}).$
\end{example} 
Below are some additional properties of $\TheMoc$.
\begin{proposition}
	$\TheMoc|_{\BMT_D}$ is a continuous map. 
\end{proposition}
\begin{proof}
	This follows directly from the definition of $\TheMoc$ and Theorem \ref{thm_convergence_family_harmonic_bundles}. 
\end{proof}

From Theorem \ref{thm_stratification_fibration}, we know that for a
$\sigma$-divisor $D$, the preimage of the map $\ptf:\BMT_D\to \tMT_D$ has
dimension $2g-2-\frac{1}{2}\deg(D)-r_2/2$, where $r_2$ is the number of odd
zeros of $q$. Even for the top stratum $D=0$, $\ptf$ is not injective if
the spectral curve is not smooth. Indeed, if  $L_1,L_2\in \BMT_D$ with $\ptf(L_1)=\ptf(L_2)$, then based on the construction
we have $\TheMoc(L_1)=\TheMoc(L_2)$. In summary, we have the following result:

\begin{proposition}
	\label{prop_simple_zero_injective}
	If $q\in H^0(K^2)$ is irreducible, then $\TheMoc$ is injective if and only if $q$ has simple zeros.
\end{proposition}

\subsection{Convergence of subsequences}
Fix a  locally free $L_0\in \MT$. 
Using the isomorphism $\psi_{L_0}:\BMT\to \BMP$ defined by $\psi_{L_0}(L)=LL_0^{-1}$, we can extend the Mochizuki map $\TheMoc$ to $\BMP$. 
For $J\in \BMP_D$, we write $\tJ:=\ptf(J)$ and choose the weight function $\frac{1}{2}\chi_D$. We then define:
\begin{equation*}
    \TheMoc_0:\BMP\lra
    \MSF(\tS)\ ,\ J\mapsto\MF_{\st}(\tJ,\tfrac{1}{2}\chi_D)\ .
\end{equation*}

\begin{proposition}
	\label{prop_algebraic_Mochizuki_map_on_PMod}
	The map $\TheMoc_0$ satisfies the following properties:
	\begin{itemize}
		\item [(i)] Let $J\in \BMP$ and $L:=L_0J$, then
            $$\TheMoc_0(J)=\TheMoc(L)\otimes \TheMoc(L_0)^{-1}\ ,$$ where
            $\otimes$ is the tensor product for filtered line bundles
            \eqref{eqn:tensor}.
		\item [(ii)] Suppose $L=\tau(I,v)$ with $(\MI,v)\in\hPMod$ and $L\in \BMP_D$, then
		\begin{equation*}
			\TheMoc_0\circ\tau(I,v)=\MF_{\st}(I(-D_v),
            \tfrac{1}{2}\chi_{D_v+\sigma^{\st}D_v})\ ,
		\end{equation*}
		where $D_v$ is the corresponding divisor defined in Theorem \ref{thm_parabolic_module_main_theorem}.
		\item [(iii)] If $\sigma^{\st}D_v=D_v$, then
            $\TheMoc_0\circ\tau(\MI,v)=\MF_{\st}(\MI,0)$, where $0$ means
            all parabolic weights are zero.
	\end{itemize}
\end{proposition}
\begin{proof}
As $L_0$ is locally free, we have $\ptf J=(p^{\st}L_0)^{-1}\otimes \ptf L$. By definition, 
$$
\TheMoc_0(J)=\MF_{\st}(\ptf J,\tfrac{1}{2}\chi_D)\ ,\
    \TheMoc(L)=\MF_{\st}(\ptf L,\tfrac{1}{2}\chi_{D-\Lam})\ ,\ \TheMoc(L_0)=
\MF_{\st}(\ptf L_0,\tfrac{1}{2} \chi_{-\Lam})\ ,
$$ 
which implies (i). For (ii), by Theorem \ref{thm_parabolic_module_main_theorem}, $\ptf L=I(-D_v)$, and from Proposition 
\ref{prop_relationship_D_v_and_D}, we have $D=D_v+\sigma^{\st}D_v$, which implies (ii). When $\sigma^{\st}D_v=D_v$, we compute
$$
\MF_{\st}(I(-D_v)\ ,\
    \tfrac{1}{2}\chi_{D_v+\sigma^{\st}D_v})=\MF_{\st}(I(-D_v)\ ,\
    \chi_{D_v})=\MF_{\st}(I,0)\ ,
$$
which implies (iii).
\end{proof}

We now give a criterion for the continuity of the map $\TheMoc$. By 
Proposition \ref{prop_algebraic_Mochizuki_map_on_PMod}, it is sufficient to study the map $\TheMoc_0$.
Recall that for $L\in \BMP$, we have 
$$
\MSN_L:=\{(J,v)\in\hPMod\mid \tau(J,v)=L\},\quad\MSD_L:=\{D_v\mid (J,v)\in\MSN_L\},
$$ 
and the number $n_L$ is defined to be the number of divisors $D_v\in \MSD_L$ such that $\sigma^{\st}D_v\neq D_v$.

\begin{proposition}
Let $D$ be a $\sigma$-divisor, $L\in\BMP_D$, and assume that $\TheMoc_0$ is continuous at $L$. Then, for $(J,v)\in \MSN_L$ and 
$D_v\in \MSD_{L}$, we have $\sigma^{\st}D_v=D_v$, i.e., $n_L=0$.
\end{proposition}
\begin{proof}
As the top stratum $\MP$ is dense in $\BMP$, there exists a family of $L_i\in \MP$ such that $\lim_{i\to \infty}L_i=L$. Let $(J_i,v_i)\in \hPMod$ be such that $\tau(J_i,v_i)=L_i$. Then, $\lim_{i\to\infty}(J_i,v_i)=(J_{\infty},v_{\infty})$, and $\tau(J_{\infty},v_{\infty})=L$. As $L_i$ is locally free, we have $D_{v_i}=0$. Moreover, by Theorem \ref{thm_parabolic_module_main_theorem}, we have $\ptf L=J_{\infty}(-D_{v_\infty})$, and from Proposition \ref{prop_relationship_D_v_and_D}, we have $D=D_{v_{\infty}}+\sigma^{\st}D_{v_{\infty}}$.
	By Proposition \ref{prop_algebraic_Mochizuki_map_on_PMod}, we have 
    $$\TheMoc_0(L_i)=\TheMoc_0\circ\tau(J_i,v_i)=\MF_{\st}(J_i,0)\ ,$$
    and we compute 
    $$\lim_{i\to\infty}\TheMoc_0(L_i)=\MF_{\st}(J_{\infty},0)=\MF_{\st}(J_{\infty}(-D_{v_{\infty}}),\chi_{D_{v_{\infty}}})
    \ .$$
Moreover, by Proposition \ref{prop_algebraic_Mochizuki_map_on_PMod}, we have 
    $$
    \TheMoc_0(L)=\MF_{\st}(J_{\infty}(-D_{v_{\infty}}),\frac{1}{2}(\chi_{D_{v_{\infty}}}+\chi_{{\sigma^{\st}D_{v_{\infty}}}})).$$
    Since $\TheMoc_0$
    is continuous on $L$, we have $\lim_{i\to \infty}\TheMoc_0(L_i)=\TheMoc(L)$, which implies that $\chi_{D_{v_{\infty}}}=\chi_{\sigma^{\st}D_{v_{\infty}}}$.
\end{proof}

By Proposition \ref{prop_computation_nL}, $n_L>0$ if and only if $q$ has at least one zero of even order. Hence, the following is immediate. 
\begin{corollary}
\label{cor_even_zero_not_continuous}
Suppose $q$ is irreducible and has a zero of even order, then $\TheMoc_0$ is not continuous.
\end{corollary}

By contrast, we have the following.
\begin{proposition}
If $q$ is irreducible with all zeroes of odd order, then $\TheMoc_0$ is continuous.
\label{prop_odd_zero_continuous}
\end{proposition}
\begin{proof}
Since all zeroes of $q$ are odd, for any $L\in \BMP$, we have $n_L=0$. Let $L_{\infty}\in \BMP$ be fixed and let $L_i\in \BMP$ be 
any sequence such that $\lim_{i\to\infty}L_i=L_{\infty}$. Since $\tau:\hPMod\to \BMP$ is bijective, we take $(J_i,v_i)\in\hPMod$ 
with $\tau(J_i,v_i)=L_i$. Moreover, we assume $\lim_{i\to \infty}(J_i,v_i)=(J_{\infty},v_{\infty})$ with $\tau(J_{\infty},v_{\infty})=L_{\infty}$. 
Since $q$ only contains odd zeros, it follows that $\supp D_v\subset \Fix(\sigma)$. By Proposition \ref{prop_algebraic_Mochizuki_map_on_PMod}, 
we have $\TheMoc_0(L_i)=\MF_{\st}(J_i,0)$. Therefore, we have:
$$
\lim_{i\to \infty}\TheMoc_0(L_i)=\lim_{i\to \infty}\MF_{\st}(J_i,0)=\MF_{\st}(J_{\infty},0)=\TheMoc_0(L_{\infty}).
$$
This concludes the proof.
\end{proof}

\begin{theorem}
	\label{thm_algebraic_Mochizuki_map_continuous}
Suppose $q$ is irreducible. For the map $\TheMoc:\MM_q\rightarrow \MSF(\tS)$,  we have:
\begin{itemize}
\item [(i)] $\TheMoc$ is injective if and only if $q$ only has only simple zeros;
\item[(ii)] if $q$ has only zeroes of odd order, $\TheMoc$ is continuous;
\item[(iii)] if $q$ contains a zero of even order, $\TheMoc$ is not continuous.
\end{itemize}
\end{theorem}
\begin{proof}
(i) follows from Proposition \ref{prop_simple_zero_injective}. (ii) follows from Proposition \ref{prop_odd_zero_continuous}. (iii) follows from
 Corollary \ref{cor_even_zero_not_continuous}.
\end{proof}

\begin{proposition}
\label{prop_number_of_limits_of_filtered_bundle}
Suppose $n_L> 0$. Then for $k=1,\ldots,n_L$, there exist sequences ${L_i^k}$ with $\lim_{i\to \infty}L_i^k=L$ such that if we denote 
$\MF_{\st}^k:=\lim_{i\to \infty}\TheMoc_0(L_i^k)$, $\MF_{\st}^0:=\TheMoc_0(L)$, then $\MF_{\st}^{k_1}\neq \MF_{\st}^{k_2}$ for $k_1\neq k_2$. 
Moreover, there exist $\{D_{1},\ldots,D_{n_L}\}\subset \MSD_L$ such that $\MF_{\st}^{k}=\MF_{\st}(\ptf L,\chi_{D_k})$.
\end{proposition}
\begin{proof}
By the definition of $n_L$, we can find $(J^k,v^k)$ with $\tau(J^k,v^k)=L$.
    If we define $D_k:=D_{v^k}$, then $\sigma^{\st}D_{k}\neq D_{k}$.
    Moreover, by Theorem \ref{thm_parabolic_module_main_theorem}, we have
    $\ptf L=J^k(-D_k)$. As $\tau^{-1}(\MP)$ is dense in $\hPMod$, for each
    $(J^k,v^k)$, we can find a sequence $(J_{i}^k,v_{i}^k)\in \tau^{-1}(\MP)$ such that $\lim_{i\to \infty}(J_i^k,v_i^k)=(J^k,v^k)$ and we define $L_i^k:=\tau(J_{i}^k,v_{i}^k)$.
Since $L_i^k$ is locally free, $D_{v_i^k}=0$, and thus $\TheMoc_0(L_i^k)=\MF_{\st}(J_i^k,0)$. We compute
\begin{equation*}
	\lim_{i\to\infty}\TheMoc_0(L_i^k)=\MF_{\st}(J^k,0)=\MF_{\st}(\ptf L,\chi_{D_k})
\end{equation*}	
and $\TheMoc_0(L)=\MF_{\st}(\ptf L,\frac12\chi_{D})$. Based on our assumptions, we have $D_{k_1}\neq D_{k_2}$ for $k_1\neq k_2$ and $\sigma^{\st}D_k\neq D_k$, which implies that $\chi_{D_{k_1}}\neq \chi_{D_{k_2}}$ for $k_1\neq k_2$ and $\chi_{D_k}\neq \frac12\chi_D$. 
\end{proof}
We now present a computation for the case of a simple nodal curve.
\begin{example}
Let $q$ be a quadratic differential with $2g-4$ simple zeros, and let $x$ be an even zero of $q$ of order two. Then $S$ has a singular point, which we also denote by $x$. Let $p:\tS\to S$ be the normalization map and 
    $p^{-1}(x)=\{x_1,x_2\}$. Consider the $\sigma$-divisor $D=x_1+x_2$, and let $L\in \BMP_D$. Then $n_L=2$, and we can write $\MSN_L={(J_1,v_1),(J_2,v_2)}$, where $D_{v_1}=x_1$ and $D_{v_2}=x_2$. Moreover, we have $\ptf L=J_1\otimes \MO(-x_1)=J_2\otimes \MO(-x_2)$.
Let $(\alpha,\beta)$ denote the parabolic weight that is equal to $\alpha$ at $x_1$, $\beta$ at $x_2$, and $\frac12$ at all other zeros. Then the filtered bundles obtained in Proposition \ref{prop_number_of_limits_of_filtered_bundle} are 
\begin{equation*}
	\MF_{\st}(\ptf L, (1,0))\ ,\ \MF_{\st}(\ptf L, (0,1))\ ,\
    \MF_{\st}(\ptf L, (\tfrac{1}{2},\tfrac{1}{2}))\ . 
\end{equation*}
\end{example}

\subsection{Mochizuki's convergence theorem for irreducible fibers}
In this subsection, we recall Mochizuki's construction of the limiting configuration metric \cite[Section 4.2.1, 4.3.2]{Mochizukiasymptotic} and the convergence theorem.

\subsubsection{Limiting configuration metric}
Let $q$ be an irreducible quadratic differential and $(\ME,\vp)\in
\MM_q$  a Higgs bundle with $(\ME,\vp)=\chi_{\BNR}(L)$. We write
$\tL_0=\ptf L$ and $(\tME,\tvp):=p^{\st}(\ME,\vp)$. Then the abelianization
of $(\ME,\vp)$, which is a Higgs bundle over $\tS$, can be written as
$\tME_0=\tL_0\oplus \sigma^{\st}\tL_0,\;\tvp_0=\diag(\om,\;-\om)$. 
The natural inclusion $\iota:(\tME,\tvp)\to (\tME_0,\tvp_0)$ is an isomorphism over $\tS\setminus \tZ$. Moreover, we write $D$ be the $\sigma$-divisor of $(\ME,\vp)$. 

From the construction of $\TheMoc(L)$ and Proposition
\ref{prop_number_of_limits_of_filtered_bundle}, we have $n_L$ different
divisors $D_k$ for $k=1,\cdots,n_L$ with $\sigma^{*}D_k\neq D_k$ and
$D_k+\sigma^{*}D_k=D$. Moreover, we can find $n_L+1$ different filtered
bundles with deg $0$. Define $$\MF_{*,0}:=\TheMoc(L)=\MF_{*}(\tL_0,\chi_{\frac12(D-\Lam)}),\; \MF_{*,k}:=\MF_{*}(\tL_0,\chi_{D_k-\frac12\Lam)}),$$
which are all degree zero filtered bundles with different level of filtrations.

Now, we will introduce the construction in \cite[Section 4.2.1, 4.3.2]{Mochizukiasymptotic}. For $k=0,\cdots,n_L$, we define $\tlh_k$ to be the harmonic metric for the filtered bundle $\MF_{*,k}$; this is well-defined up to a multiplicative constant. To fix this constant, assume that $\sigma^{\st}\tlh_k\otimes \tlh_k=1$. This gives  a unique choice of $\tlh_0$. We then define the metric $\tH_k=\diag(\tlh_k,\;\sigma^{\st}\tlh_k)$ on $\tME_0$, with $\det(\tH_k)=1$. For the resulting harmonic bundle 
$(\tME_0,\vp_0,\tH_k)$, we define $\tilde{\na}_k$ to be the unitary connection determined by $\tH_k$. Since $\tH_k$ is diagonal, over $\tS\setminus \tZ$, it follows that $F_{\tilde{\na}_k}=0$, and we have $[\vp_0,\vp_0^{\da_{\tH_k}}]=0$. Furthermore, as $\iota$ is an isomorphism on $\tS\setminus \tZ$, the metric $\tH_k$ also defines a metric on $(\tME,\tvp)$ over $\tS\setminus \tZ$. 

For any $\tx\in \tS\setminus \tZ$ with $x:=p(\tx)$, we have isomorphisms 
$$
(\tME_0,\tvp_0)|_{\sigma(\tx)}\cong (\tME_0,\tvp_0)|_{\tx}\cong(\tME,\tvp)|_{\tx}\cong (\ME,\vp)|_x.
$$
Therefore, $\tH_k$ induces a metric $H^{\LC}_k$ on $\Sigma\setminus Z$, and we may consider $H^{\LC}_k$ as the 
push-forward of $\tlh_k$. In \cite[Theorem 5.2]{horn2022sltypesingularfibers}, the push-forward metric of $\TheMoc(L)$ is explicitly written in local coordinates. 

Recall the notation from Section \ref{sec:nah}, let $E$ be a trivial, smooth,  rank 2 vector bundle over a Riemann surface $\Sigma$, and let $K$ be a background Hermitian metric on $E$. Over $\Sigma\setminus Z$, we write $\na_k^{\LC}$ for the Chern connection defined by $H_k^{\LC}$, which is unitary w.r.t. $H_0$ and $\phi^{\LC}_k=\vp_k^{\LC}+\vp_k^{\da_\LC}$ be the corresponding Higgs field in the unitary gauge. They satisfy the decoupled Hitchin equations over $\Sigma\setminus Z$. 
Thus from any Higgs bundle $(\ME,\vp)$, we obtain $n_L+1$ limiting configurations $$(\na_k^{\LC},\phi_k^{\LC}=\vp+\vp_k^{\da_{\LC}})\in \MMH^{\LC}.$$

The flat connection, which is defined over $\Sigma\setminus Z$,  may be
understood by using the nonabelian Hodge correspondence for filtered vector
bundles \cite{simpson1990harmonic}. Given the filtered line bundles
$\MF_{*,k}$,  define filtered vector bundles $\tME_{*,k}:=\MF_{*,k}\oplus
\sigma^{*}\MF_{*,k}$, which can be explicitly written as 
\begin{equation}
	\begin{split}
			\tME_{*,0}:&=\MF_{*}(\tL_0,\chi_{\frac12(D-\Lam)})\oplus
            \MF_{*}(\sigma^{*}\tL_0,\chi_{\frac12(D-\Lam)})\ ;\\
			\tME_{*,k}:&=\MF_{*}(\tL_0,\chi_{D_k-\frac12\Lam})\oplus
            \MF_{*}(\sigma^{*}\tL_0,\chi_{\sigma^{*}D_k-\frac12\Lam})\ , \
            k\neq 0\ .
	\end{split}
\end{equation}
These are polystable filtered vector bundles over $\tS\setminus \tZ$. As
for each $k=0,\cdots,n_L$, $\sigma^{*}\tME_{*,k}=\tME_{*,k}$, the filtered
bundles $\tME_{*,k}$ induce a filtered vector bundles $\ME_{*,k}$ over
$\Sigma\setminus Z$. The flat connection $\na_k^{\Lim}$ will be the unique
harmonic unitary connection corresponding to the $\ME_{*,k}$. Moreover, for
$0\leq k_1\neq k_2\leq n_L$, based on the definition of $D_{k_1}$ and
$D_{k_2}$, we can always find $\tx\in \tZ_{\mathrm{even}}$, a preimage of
an even zero $x$ of $q$,  such
that $\tME_{*,k_1}$ and $\tME_{*,k_2}$ have different filtered structures
near $\tx$. Since over even zeros, $\tS\to \Sigma$ is not a branched
covering, we conclude that near $x$, $\ME_{*,k_1}$ and $\ME_{*,k_2}$ are
different filtered bundles. By \cite[Main theorem]{simpson1990harmonic}, the
harmonic connections $\na_{k_1}$ and $\na_{k_2}$ are not gauge equivalent.

We therefore conclude the following:
\begin{proposition}
	\label{prop_different_limiting_configurations}
	For $0\leq k_1\neq k_2\leq n_L$, $(\na_{k_1}^{\LC},\phi_{k_1}^{\LC})$ and $(\na_{k_2}^{\LC},\phi_{k_2}^{\LC})$ are not gauge equivalent in $\MMH^{\LC}$. 
\end{proposition}

This leads us to define the \emph{analytic Mochizuki map} $\UpMoc$ as
\begin{equation} \label{eq_analytic_moc_irreducible}
    \UpMoc:\MM_q\lra \MMHLC\ :\ [(\ME,\vp)]\mapsto
    [(\na_0^{\LC},\phi_0^{\LC})],
\end{equation}
which we recall is the limiting configuration defined by $\TheMoc(L)$.

\subsubsection{The continuity of the limiting configurations} We now introduce the main result of 
Mochizuki \cite{Mochizukiasymptotic}. Fix $(\ME,\vp)=\chi_{\BNR}(L)\in\MM_q$. For any real parameter $t$,
$(\ME,t\vp)$ is a stable Higgs bundle. By the Kobayashi-Hitchin correspondence, there exists a unique metric $H_t$
solving the Hitchin equation. Denote by  $\na_t$ the unitary connection defined by $H_t$ and write 
$\MD_t=\na_t+t\phi_t$ for the full $\SLC$ flat connection. We then have:
\begin{theorem}[{\cite{Mochizukiasymptotic}}]
	\label{thm_moc_convergence_irreducible_case} The family $(\ME,t\vp)$ has a unique limiting configuration 
as limit for $t\to \infty$.  Moreover, for any compact set $K\subset \Sigma\setminus Z$, let $d:=\min_{x\in  K}|q|(x)$. 
Then there exist $t$-independent constants $C_k$ and $C'_k$ such that
	$$
	|(\na_t,\phi_t)-\UpMoc(\ME,\vp)|_{\MC^k}\leq C_ke^{-C'_ktd}\ .
	$$
\end{theorem}

As the map $\UpMoc$ is the composition of $\TheMoc \circ \chi_{\BNR}^{-1}$
with the nonabelian Hodge correspondence, the behavior of $\UpMoc$ is the
same as $\TheMoc$. 
Recall the decomposition $\MM_q = \bigcup \MM_{q,D}$  from the end of
Section \ref{sec_parabolic_modules}.
By Theorem \ref{thm_algebraic_Mochizuki_map_continuous}, Proposition \ref{prop_number_of_limits_of_filtered_bundle} and Proposition \ref{prop_different_limiting_configurations}, we obtain:
\begin{theorem}\label{thm_analytic_moc_irreducible_fiber}
Let $q$ be an irreducible quadratic differential. The map $\UpMoc: \MM_q \to \MMHLC$ satisfies the following properties:
	\begin{itemize}
		\item[(i)] if all the zeros of $q$ are odd, then $\UpMoc$ is continuous;
		\item[(ii)] if at least one zero of $q$ is even, 
then for each $(\ME, \vp) \in \MM_{q,D}$, there exists an integer $n_D$
            that only depends on $D$, and $n_D$ sequences $\{(\ME_i^k,
            \vp_i^k)\}$ for $k=1, \dots, n_D$, such that $\lim_{i \to
            \infty}(\ME_i^k, \vp_i^k) = (\ME_{\infty}, \vp_{\infty})$, and
            $$\lim_{i \to \infty} \UpMoc(\ME_i^{k_1}, \vp_i^{k_1}) \neq
            \lim_{i \to \infty} \UpMoc(\ME_i^{k_2}, \vp_i^{k_2}) \neq
            \UpMoc(\ME_{\infty}, \vp_{\infty})\ , \ \text{for $k_1 \neq k_2$}\ 
.$$
		\end{itemize}
\end{theorem}

\section{Reducible singular fiber and the Mochizuki map}
\label{sec_reducible_singular_fiber}
We now investigate the properties of the Hitchin fiber associated with a reducible quadratic differential, as discussed in
\cite{gothen2013singular}. Additionally, we will provide an overview of Mochizuki's technique for constructing limiting 
configurations of Hitchin fibers for reducible quadratic differentials, as detailed in \cite{Mochizukiasymptotic}. 
We also analyze the continuity of the Mochizuki map. 
\subsection{Local description of a Higgs bundle}
\label{subsec_localdescription}
Write $q=-\om\oti\om$ with $\om\in H^0(K)$,
$\Lam=\Div(\omega)$, $Z=\supp(\Lam)$, and  $\MM_q=\MH^{-1}(q)$. Compared to
the irreducible case, $\MM_q$ contains strictly semistable Higgs bundles, so we let
$\MM_q^{\rm st}$ denote the stable locus. 
We point out that there is a sign ambiguity  in the choice of $\om$,
which actually  plays an important role in the following.
\subsubsection{Local description}
Given a Higgs bundle $(\ME,\vp)$ with $\det(\vp)=q$,  define line bundles 
\begin{equation} \label{eq_kernel_reducible}
	L_{\pm}:=\ker(\vp\pm\om)\ .
\end{equation}
Then the inclusion maps $L_{\pm}\to \ME$ are injective. Similarly, we may
define an abelianization of $(\ME,\vp)$ by 
$(\ME_0=L_+\oplus L_-,\vp_0=\diag(\om,-\om))$. We then have a natural inclusion 
$\iota:\ME_0\to \ME$, which is 
is an isomorphism on $\Sigma\setminus Z$, and $\vp\circ\iota=\iota\circ\vp_0$. 

It follows from \cite[Prop.\ 7.10]{gothen2013singular} that $L_{\pm}$ are
the only $\vp$-invariant subbundles of $\ME$. 
If we write $d_{\pm}:=\deg(L_{\pm})$, then $(\ME,\vp)$ is stable (resp.\ semistable) 
if and only if $d_{\pm}<0\ (\text{resp.}\ \leq 0)$. As $\det(\ME)=\MO$, the 
map $\det(\iota):L_+\otimes L_-\to \MO$ defines a divisor $D=\Div(\det(\iota))$ such 
that $L_+\otimes L_-=\MO(-D).$ Therefore, we obtain 
\begin{equation*}
d_++d_-+\deg D=0\ ,
\end{equation*}
and $0\leq D\leq \Lam$. The Higgs bundle $(\ME,\vp)$ is semistable if and
only if $-\deg D\leq d_+\leq 0$ and stable if the equalities are strict.
For the rest of this section, we always write $D=\sum_{p\in Z}\ell_p p$. 

\subsubsection{Semistable settings}
As the fiber $\MM_q$ might contain strictly semistable Higgs bundles, we now explicitly enumerate all of the possible 
$S$-equivalence classes. When $D=0$, then $L_-=L_+^{-1}$ and $\deg(L_+)=0$. The corresponding Higgs
bundle is polystable and can be explicitly written as
$$
\Bigl(L\oplus L^{-1}, \begin{pmatrix}
	\om & 0\\
	0 & -\om
\end{pmatrix}\Bigr)\ ,
$$ where $L\in \Jac(\Sigma)$.
When $D\neq 0$, suppose $\deg(L_+)=-\deg(D)$. Then $L_-=L_+^{-1}(-D)$ and
$\deg(L_-)=0$. Under $S$-equivalence, the polystable Higgs bundle  is
$$
\Bigl(L_+(D)\oplus L_+^{-1}(-D),\begin{pmatrix}
	\om & 0\\
	0 & -\om
\end{pmatrix}\Bigr)\ ,
$$
where $L_+\in \Pic^{-\deg(D)}(\Sigma)$.

\subsection{Reducible spectral curves}
In this subsection, we introduce the algebraic data in \cite{gothen2013singular} which describes the singular fiber with
a reducible spectral curve. This plays a similar role to the parabolic modules. See \cite[Sec.\ 7.1]{gothen2013singular} for more details.

For any divisor $D$, and line bundle $L$,  define the space
$$
H^0(D,L)=\bigoplus_{p\in \supp D}\MO(L)_p/\sim\ ,
$$
where $s_1\sim s_2$ if and only if $\ord_p([s_1]-[s_2])\geq D_p$. 

Let $L\in \Pic(\Sigma)$,  define the following subspaces of $H^0(\Lam,L^2K)$:
\begin{equation*}
	\begin{split}
		&\MV( D,L):=\{s\in H^0(\Lam,L^2K)\mid\ord_p(s)=\Lam_p-D_p,\;\mathrm{if}\;D_p>0;\;s(p)=0,\;\mathrm{if}\;D_p=0\},\\
		&\MW( D,L)=\{s\in H^0(\Lam,L^2K)\mid s|_{\supp(\Lam-D)}=0\}\ .
	\end{split}
\end{equation*}
One checks that $\MW(D,L)=\cup_{D'\leq D}\MV(D',L)$. Moreover, the space
$\MV( D,L)$ is a linear subspace of $H^0(\Lam,L^2K)$ with a hyperplane removed. In addition, $\CS$ acts on $\MV( D,L)$ by 
multiplication, and $\dim(\MV( D,L)/\CS)=\deg( D)-1$.

We define the fibrations
\begin{equation*}
	\begin{split}
		p_m:\MSV( D,m)\lra \Pic^m(\Sigma),\;p_m:\MSW( D,m)\lra \Pic^m(\Sigma)
	\end{split}
\end{equation*}
such that for $L\in \Pic^m(\Sigma)$, the fibers are $\MV(D,L)$ and $\MW(D,L)$. 

\subsubsection{Algebraic data from the extension} 
The Higgs bundle $(\ME,\vp)$ can be
understood in terms of an extension of exact sequence. As $\det(\ME)=\MO$, we have the exact sequence 
$$
0\lra L_+\lra \ME\lra L_+^{-1}\lra 0\ .
$$
For each $p\in Z$, with $U\subset \Sigma$ a neighborhood of $p$,
$(\ME,\vp)$ can be written as the following splitting of $\MC^{\infty}$ bundles
\begin{equation*}
	\begin{split}
		\ME=L_+\oplus_{\MC^{\infty}} L_+^{-1}\ ,\ \bar{\pa}_{\ME}=\begin{pmatrix}
			\bpa_{L_+} & b\\
			0 & \bpa_{L_+^{-1}}
		\end{pmatrix}\ ,\ \vp=\begin{pmatrix}
		\om & c\\
		0 & -\om
	\end{pmatrix}\ .
	\end{split}
\end{equation*}

We would like to consider the restriction of $\vp+\om\cdot\id$ to $\Lam$.
As $\om|_{\Lam}=0$, $\vp+\om\cdot\id|_{L_+}=0$ and the image of
$\vp+\om\id\subset (L_+\oplus 0)\otimes K\subset \ME\otimes K$. Therefore,
the restriction of $\vp+\om\cdot \id$ to $\Lam$ defines a holomorphic map $s :L_+^{-1}|_{\Lam}\to L_+K|_{\Lam},$ or equivalently a section $s\in H^0(\Lam,L_+^2K).$ Moreover, by \cite[Lemma 7.12]{gothen2013singular}, $\Div(s)=\Lam-D$. Therefore, given any $(\ME,\vp)\in \MM_q$, we obtain an $L\in\Pic^{m}(\Sigma)$ and an element in $\MV(D,L)$. Moreover, the stability condition implies that $0\leq D\leq \Lam$, we have $-\deg  D\leq \deg L\leq 0.$

\subsubsection{Inverse construction}
The inverse of the construction above also holds; for further details, see \cite[Sec.\ 7]{gothen2013singular} and \cite[Sec.\ 5]{horn2022semi}.
Given $L\in \Pic^{m}(\Sigma)$ and $q\in \MV(D,L)$,  we define a Higgs bundle via extensions as follows. From $q,L$, we have a short exact sequence 
of complexes of sheaves:
\begin{equation*}
	\begin{tikzcd}
		& C_1^{\st}& C_2^{\st}& C_3^{\st}&\\
		0\arrow[r, ] & L^2 \arrow[r, "="]  \arrow[d, "\id"] & L^2 \arrow[r, "\pr"] \arrow[d, "c"] & 0 \arrow[r, ] \arrow[d, "0"] & 0 \\ 
		0\arrow[r, ]& L^2 \arrow[r, "c"]&L^2K \arrow[r, "\mathrm{res}(\Lam)"]&L^2K|_{\Lam'} \arrow[r, ] & 0
	\end{tikzcd},
\end{equation*}
where, for a section $s\in \Gamma(L^2)$,  $c(s):=\sqrt{-1}\omega s$, and $\mathrm{res}(\Lam)$ is the restriction map to the divisor $\Lam$. 
The long exact sequence in hypercohomology implies that $\mathrm{res}(\Lam)$ induces an isomorphism 
$$
\mathrm{res}(\Lam):\mbfH^1(C_2^{\st})\cong
\mbfH^1(C_3^{\st})=H^0(\Lam,L^2K)\ .
$$
Moreover, we have $\mbfH^1(C_2^{\st})\cong H^1(\Sigma,L^2)$, which parameterizes extensions 
\begin{equation*}
	0\lra L\lra \ME\lra L^{-1}\lra 0\ .
\end{equation*}
From $s\in H^0(\Lam,L^2K)$ and $\ME$ above, we can find a section
$c\in\Gamma(L^2K)$, and construct a Higgs bundle 
\begin{equation} \label{eq_Higgsbundle_construction_extension}
	E=L\oplus_{\MC^{\infty}} L^{-1},\;\bar{\pa}_E=\begin{pmatrix}
		\bar{\pa}_L & b\\
		0 & \bar{\pa}_{L^{-1}}
	\end{pmatrix},\;\vp=\begin{pmatrix}
	\om & c\\
	0 & -\om
\end{pmatrix},
\end{equation}
where $\bar{\pa}c=2b\om$ for $b\in \Omega^{0,1}(L^2)$ and $c$ is an
extension of $q$. 
For $0\leq D\leq \Lam$ and $-\deg D\leq m\leq 0$, the construction above defines a map 
\begin{equation*}
	\wp:\MSV(D,m)\lra \MM_q\ ,\ s\in\MV(D,L)\mapsto(\ME,\vp)\ ,
\end{equation*}
where $(\ME,\vp)$ is the Higgs bundle constructed in \eqref{eq_Higgsbundle_construction_extension}. 
When $ D=0$,  $\MV(\Lam,L)=\{0\}$ and the image of $\wp:\MSV(\Lam,0)\to \MM_q$ are the polystable Higgs bundles $\ME=L\oplus L^{-1},\;\vp=\diag(\om,-\om)$ such that $L^2\cong \MO_{\Sigma}$.

\begin{theorem}[{\cite[Thm.\ 7.7]{gothen2013singular}}]
	\label{thm_reduciblefiberBNR}
	For $ 0\leq D\leq \Lam$ and $-\deg( D)\leq m_1\leq 0$ and the map $\wp:\MSV(D,m_1)\to \MM_q$, we have
	\begin{itemize}
		\item [(i)] for $m_2=-\deg( D)-m_1$, we have $\wp(\MSV(D,m_1))=\wp(\MSV(D,m_2))$,
		\item [(ii)] for the $\CS$ action on $\MSV(D,m_1)$ by multiplication, for $\xi\in \MSV(D,m_1)$, $\wp(\CS \xi)=\wp(\xi)$,
		\item [(iii)] when $m_1\neq -\frac12\deg( D)$, $\wp:\MSV(D,m_1)/\CS \to \MM_q$ is an isomorphism onto its image,
		\item [(iv)] when $m_1=-\frac{1}{2}\deg( D)$,  $\wp:\MSV(D,m_1)/\CS \to \MM_q$ is a double branched covering, which branched along line bundles $L\in \Pic^{m_1}(\Sigma)$ such that $L^2\cong \MO(- D)$,
		\item [(v)] when $ D=0$, then $\wp:\MSV(\Lam,0)\to \MM_q$ is a double branched covering, branched along $L\in \Pic^0(\Sigma)$ such that $L^2\cong\MO.$
	\end{itemize}
\end{theorem}

\begin{example}
	\label{ex_genus_two_stable.}
	When $g=2$, for $q=-\om\otimes \om$, we can
    write $\Lam=p_1+p_2$ or $\Lam=2p$. In either case, the
    $\MM_q^{\sta}=\wp(\MSV(D,m))$ for $-\deg(D)<m<0$ and $0\leq D\leq
    \Lam$. Therefore, $m=-1,\;D=\Lam$ and $\wp(\MSV(\Lam,-1))=\MM_q^{\sta}$. Moreover, generically, the map $\wp:(\MSV(\Lam,-1))/\CS\to \MM_q^{\sta}$ is two-to-one.
\end{example}
\subsection{The stratification of the singular fiber}
We  now present two stratifications of $\MM_q$. Recall that from any Higgs bundle $(\ME,\vp)$ we obtain two line bundles $L_{\pm}$ and a 
divisor $D$. There are two different stratifications: one given by the divisor $D$ and the other by the degree of $L_+$.

\subsubsection{Divisor stratification} We first discuss the stratification
defined by the divisor. Indeed, using $D$, decompose into strata: $\MM_q=\bcup_{0\leq  D\leq \Lam}\MM_{ D}$. As the definition of $L_{\pm}$ depends on the choice of the 
square root, there is no natural map from $\MM_{ D}$ to $\Pic(\Sigma)$. Consider the following space: $\MBV_{ D}=\bcup_{-\deg( D)\leq m\leq
0}\MSV( D,m)$. This forms a fibration 
$$
\tau:\MBV_{ D}\lra \bcup_{-\deg( D)\leq m\leq 0}\Pic^m(\Sigma)\ .
$$ 
Moreover, for $L\in \Pic^{m}(\Sigma)$, we have $\tau^{-1}(L)=\MSV( D,L)$ and $\dim(\tau^{-1}(L)/\CS)=\deg( D)-1$. 
By Theorem \ref{thm_reduciblefiberBNR}, $\wp|_{\MBV_{ D}}:\MBV_{ D}\to \MM_{ D}$ is surjective. 
Moreover, since  $$\wp|_{\MSV( D,m)}=\wp|_{\MSV( D,-\deg( D)-m)}$$ generically, 
$\wp|_{\MBV_{ D}}$ is a two-to-one map. 

In summary, we obtain the following map which characterizes the singular fiber.
$$
\wp: \MBV=\bcup_{0\leq  D\leq \Lam}\MBV_{ D}\rightarrow \MM_q=\bcup_{0\leq  D\leq \Lam }\MM_{ D}.
$$
The top stratum is given by $D=\Lam$. 

\subsubsection{Degree stratification}
We next introduce the stratification defined by degrees; this encodes how different divisor stratifications are pasted together. 
For $-(2g-2)\leq m\leq 0$ and $L\in\Pic^m(\Sigma)$, 
define $\MBW(L):=\bcup_{\deg D\geq -m}\MV( D,L)$. This set is connected,
based on the definition and \cite[Lemma 7.14]{gothen2013singular}.
Moreover, if  we define 
$$\MBW_m:=\bcup_{-m\leq \deg  D,\;0\leq D\leq \Lam}\MSV( D,m)\ ,\
\MBW:=\bcup_{-(2g-2)\leq m\leq 0}\MBW_m\ ,$$ then we have $\wp(\MBW)=\wp(\MBV).$
We should also note that though $\MBW_m\cap \MBW_n=\emptyset$ for any
$m\neq n$, $\MBW$ is connected. As $L_+,L_-$ are symmetric, by Theorem
\ref{thm_reduciblefiberBNR}, we have $\wp(\MSV( D,m))=\wp(\MSV( D,-\deg(
D)-m))$, which implies that for any integer $-(2g-2+m)\leq n\leq 0$,
$\wp\MBW_m\cap \wp\MBW_n\neq \emptyset$. 

We now give an example of the degree stratification when $g=2$. 
\begin{example}
	\label{example_genustwodegreestratification}
Suppose $\om$ has only one zero with order 2. Then $\Lam=2p$, and all
    possible divisors are $D_2=2p,D_1=p,D_0=0$. The degree stratification
    is  
\begin{equation*}
	\begin{split}
		&\MBW_{-2}=\MSV(D_2,-2)\ ,\ \MBW_{-1}=\MSV(D_2,-1)\cup\MSV(D_1,-1)\\
		&\MBW_0=\MSV(D_0,0)\cup\MSV(D_1,0)\cap \MSV(D_2,0)\ .
	\end{split}
\end{equation*}
The image of $\wp(\MSV(D_2,-1))$ is stable, $\wp(\MSV(D_0,0))$ is poly-stable and $\wp(\MSW\setminus (\MSV(D_2,-1)\cup\MSV(D_0,0)))$ is semistable. 

Moreover, we have $\wp(\MSV(D_2,-2))=\wp(\MSV(D_2,0))$, $\wp(\MSV(D_1,-1))=\wp(\MSV(D_1,0))$ and $\wp|_{\MSV(D_2,-1)}$ is a branched covering. Moreover, we have $\wp(\MSV(D_2,-1))\cap \wp(\MSV(D_1,0))\neq 0$ and $\wp(\MSV(D_2,-1))\cap \wp(\MSV(D_0,0))=0$. 
\end{example}

\subsection{Algebraic Mochizuki map}
Based on the study of the local rescaling properties of Higgs bundles,
Mochizuki introduced a weight for each $p \in Z$ in \cite[Sec.\ 3]{Mochizukiasymptotic}. To be more specific, let $c$ be a real number. For each $p \in Z$, the weight we consider is given by
$$\chi_p(c) = \min\{\ell_p, (m_p + 1)c + \ell_p/2\}\ .$$

By utilizing the global geometry of a Higgs bundle, we can uniquely determine the constant $c$. We aim to choose the sign of $\omega$ such that $d_+ \leq d_-$.

\begin{lemma}[{\cite[Lemma 4.3]{Mochizukiasymptotic}}]
	\label{lemma_weightparabolic}
	If $(\ME,\vp)$ is stable, then there exists a unique constant $c>0$ such that
	\begin{equation*}
		d_++\sum_{p\in Z}\chi_p(c)=0\ ,\ d_-+\sum_{p\in
        Z}(\ell_p-\chi_p(c))=0\ .
	\end{equation*}
\end{lemma}

\begin{proof}
	Since $(\ME,\vp)$ is stable, we have $-\sum \ell_p<d_{\pm}<0$. We define the function
	\begin{equation} \label{eq_mon_function_weight}
		f(c) = d_++\sum_p\chi_p(c)\ ,
	\end{equation}
	which is strictly increasing. Moreover, for $c$ sufficiently large,
    $\chi_p(c)=\ell_p$, and therefore $f(c)=d_++\sum_p\ell_p=-d_->0$.
    Additionally, $f(0)=d_++\sum_p (\ell_p/2)$. Since $d_+\leq d_-$ and
    $d_++d_-+\sum_p\ell_p=0$, we obtain $f(0)\leq 0$. The monotonicity of $f$ implies the existence of $c_0$ such that $f(c_0)=0$.
\end{proof}

From the construction, if $d_+\leq d_-$, two weighted bundles
$(L_+,\chi_p(c_0))$ and $(L_{-},\ell_p-\chi_p(c_0))$ are obtained with weights
$\chi_p(c_0)$ and $\ell_p-\chi_p(c_0)$ at each $p\in Z$, respectively. On the
other hand, if $d_+\geq d_-$, by symmetry, weighted bundles
$(L_+,\ell_p-\chi_p(c_0))$ and $(L_-,\chi_p(c_0))$ are obtained. When
$(\ME,\vp)$ is strictly semistable,  S-equivalent to
$(L,\om)\oplus (L^{-1},-\om)$, then we would like to consider the weighted bundles $(L,0)\oplus (L^{-1},0)$ with weight zero.

Next, we define the algebraic Mochizuki map. Let $\MSF_{\pm}(\Sigma)$ be
the space of rank 1 degree zero filtered bundles on $\Sigma$, and let 
$\MSF_2(\Sigma):=\MSF_+(\Sigma)\oplus \MSF_-(\Sigma)$ be the direct sum. 
Fix a choice of $\om$. Then from any Higgs bundle $(\ME,\vp)$, we obtain the subbundles $L_{\pm}$ with degree $d_{\pm}$ and define the algebraic Mochizuki map 
\begin{equation*}
	\begin{split}
			\TheMoc&:\MM_q\lra \MSF_2(\Sigma),\\
			\TheMoc&(\ME,\vp):=\left\{\begin{matrix}
				\MF_{\st}(L_+,\chi_p(c_0))\oplus \MF_{\st}(L_-,\ell_p-\chi_p(c_0)),\;\mathrm{if}\;d_+\leq d_-\\
				\MF_{\st}(L_+,\ell_p-\chi_p(c_0))\oplus \MF_{\st}(L_-,\chi_p(c_0)),\;\mathrm{if}\;d_-\leq d_+
			\end{matrix}\right.,\;(\ME,\vp)\;\mathrm{stable},\\
			\TheMoc&(\ME,\vp):=\MF_{\st}(L,0)\oplus \MF_{\st}(L^{-1},0),\;(\ME,\vp)\;\mathrm{semistable}.
	\end{split}
\end{equation*}

We list some properties of this map.
\begin{proposition}
	\label{prop_reducible_continuous_on_strata}
	For $\TheMoc$, we have:
	\begin{itemize}
		\item [(i)]for each $\MSV( D,m)$ with $0\leq  D\leq \Lam$, $-\deg(
            D)\leq m\leq 0$, $\TheMoc|_{\wp(\MSV( D,m))}$ is continuous,
		\item [(ii)]for i=1,2 and $s_i\in \MBV_D$ with $(\ME_i,\vp_i):=\wp(s_i)$, suppose $\tau(s_1)=\tau(s_2)$, then $\TheMoc(\ME_1,\vp_1)=\TheMoc(\ME_2,\vp_2)$. In particular, $\TheMoc$ is not injective. 
	\end{itemize}
\end{proposition}
\begin{proof}
	The proof follows directly from the definition.
\end{proof}

\subsection{Exotic phenomena}
A Higgs bundle $(\ME,\vp) \in \MM_q$ is called ``exotic'' if the constant $c$ in Lemma \ref{lemma_weightparabolic} satisfies $c \neq 0$. This new behavior only appears in the Hitchin fiber with reducible spectral curve. 
In this subsection, we aim to understand the exotic phenomenon of a Higgs bundle.
We first provide another expression for the weight in Lemma \ref{lemma_weightparabolic}.

\begin{proposition}
	\label{prop_choiceofconstant_c}
	Suppose $-\deg  D <d_+\leq -\frac12\deg D$ and let $Z_0:=\supp(D)$,
    then the constant $c$ in Lemma \ref{lemma_weightparabolic} is given by
    $$c=\frac{d_++\frac12\deg D}{\deg(\Lam|_{Z_0})+|Z_0|}\ ,$$
	where $\Lam|_{Z_0}$ means the restriction of the divisor to $Z_0$, and $|Z_0|$ means the number of points in $Z_0$ without multiplicity.
\end{proposition}
\begin{proof}
	By the definition of $\chi_p(c)$, if $p\in Z_1$, $\chi_p(c)=0$ for any
    $c\geq 0$. Therefore, the choice of $c$ is determined by the equation
    $d_++\sum_{p\in Z_0}\min\{\ell_p,(m_p+1)c+\frac{\ell_p}{2}\}=0$.
    Define the function $F(c):=d_++\sum_{p\in Z_0}((m_p+1)c+\ell_p/2)$,
	and let $f(c)$ be the function defined in $\eqref{eq_mon_function_weight}$. Then $F(c)\geq f(c)$, $F(0)=0$ and for $c$ sufficiently large, $F(c)=f(c)$. 
	We compute 
    $$\sum_{p\in
    Z_0}(m_p+1)c=(\deg(\Lam|_{Z_0})+|Z_0|))c\ ,\ \sum_{p\in
    Z_0}(\ell_p/2)=\frac{1}{2}\deg D\ .$$ 
    Then for $$c_0=\frac{d_++\frac12\deg D}{\deg(\Lam|_{Z_0})+|Z_0|}$$
    we have $F(c_0)=0$. Therefore, $f(c_0)=0$, which determines the choice
    of the constant $c$ in Lemma $\ref{lemma_weightparabolic}.$
\end{proof}

\begin{corollary}
	A Higgs bundle $(\ME,\vp)$ is exotic if and only if its corresponding degrees $d_{\pm}$ satisfy $d_+\neq d_-$. Additionally, there are only a finite number of possible choices for $c$ over $\MM_q$.
\end{corollary}
\begin{proof}
	The result follows directly from the formulas in Proposition \ref{prop_choiceofconstant_c}.
\end{proof}

We will compute some examples of possible weights in some special cases. First consider the generic case. 
\begin{example}
	Suppose $\om$ has only simple zeros, therefore $Z=\{p_1,\ldots,p_{2g-2}\}$ and $\Lam=p_1+\cdots+p_{2g-2}$. Given a stable Higgs bundle $(\ME,\vp)$ with corresponding $L_+,L_-, D,d_{\pm}$, recall that from the stability condition, these degrees satisfy the followings:
	\begin{equation*}
		d_++d_-+\deg D=0\ ,\ d_+\leq d_-\ ,\ d_+<0,\;d_-<0\ ,\ -\deg
        D<d_+\leq -\frac{1}{2}\deg D\ .
	\end{equation*}
	If we write $Z_0:=\supp D$, then $\chi_p=0$ for $p\in Z\setminus Z_0$
    and $\chi_p=2c+1/2$ for $p\in Z_0$. Moreover, we have $d_++\deg(
    D)(2c+1/2)=0$. Therefore, 
	$c=-\frac{d_+}{2\deg( D)}+\frac14$, and two weighted bundles we obtained
    are
	\begin{equation*}
		(L_+,(-\frac{d_+}{\deg D}|_{Z_0},0|_{Z\setminus Z_0}))\ ,\ (L_-, (1+\frac{d_+}{\deg D}|_{Z_0},1|_{Z\setminus Z_0}))
        \ .
	\end{equation*}
\end{example}
Next,  we consider the most nongeneric case.
\begin{example}
	Suppose $\omega$ only contains one zero. Write $\Div(\om)=(2g-2)p$.
    Then the possible divisors are $ D=\ell p$, for $0\leq \ell\leq 2g-2$.
    Let $L_+$ be a line bundle with $d_+:=\deg(L_+)$ and  $-\ell<d_+\leq
    \ell/2$. Then the choice of $c$ is determined by the equation
	\begin{equation*}
		\begin{split}
			d_++\min\{\ell,(2g-1)c+\ell/2\}=0\ .
		\end{split}
	\end{equation*}
As $d_+\neq -\ell$, we must have $c=-\frac{\ell}{2(2g-1)}$ and the corresponding
    weighted bundles are
\begin{equation*}
	\begin{split}
		(L_+,-d_+),\;(L_+^{-1}\otimes \MO(-\ell p),\ell-d_+)\ .
	\end{split}
\end{equation*}
\end{example}

Finally, we  explicitly compute the limiting configurations when $g=2$ for
the strictly semistable locus of the stratification in Example \ref{example_genustwodegreestratification}. 
\begin{example}
	\label{example_limitingconfiguration_genus_2}
	When $g=2$, we consider $\Lam=2p$ and $D_i=i p$ for $i=0,1,2$. Let
    $p_m:\MSV(D,m)\to \Pic^m(\Sigma)$ be the projection. Over
    $\wp(\MSV(D_i,0))$, for $L\in \Pic^0(\Sigma)$, the S-equivalence class
    for $\wp(p_0^{-1}(L))$ is $(L\oplus L^{-1},\begin{pmatrix}
		\om & 0\\
		0 & -\om
	\end{pmatrix})$ and the corresponding weighted bundles are $(L,0)\oplus (L^{-1},0)$.
	Over $\wp(\MSV(D_2,-1))$, for $L\in \Pic^{-1}(\Sigma)$, let
    $(\ME,\vp)=\wp(L,s\in \MV(D_2,L))$, then
    $$\TheMoc(\ME,\vp)=\MF_{\st}(L,1)\oplus
    \MF_{\st}(L^{-1}(-D_2),1)=\MF_{\st}(L(p),0)\oplus
    \MF_{\st}(L^{-1}(p-D_2),0)\ .$$
\end{example}

\subsection{Discontinuous behavior}
In this subsection, we study the discontinuous behavior of $\TheMoc$. Consider 
a sequence of algebraic data $(L_i,q_i)\in \MBW_m$, where $L_i\in\Pic^m$ and $q_i\in \MSV(D,L_i)$. 
We assume that $\lim_{i\to \infty}L_i=L_{\infty}$ in $\Pic^m$ and 
$\lim_{i\to \infty}q_i=q_{\infty}\in \MSV(D_{\infty},L)$, for
$D_{\infty}\neq D$. As the space $\bcup_{\deg D'\geq -m}\MSV(D',m)$ is
connected, we can always find such a sequence.

Let $L_+^i:=L_i$ and $L_-^i:=L_i^{-1}\otimes\MO(-D)$. By Lemma \ref{lemma_weightparabolic}, the weight function, 
which we denote by $\chi_{\pm}$, is independent of $i$. In addition, we have 
$$
\lim_{i\to\infty}\TheMoc\circ
\wp(L_i,q_i)=\MF_{\st}(L_{\infty},\chi_+)\oplus \MF_{\st}(L_{\infty}^{-1}(-
D),\chi_-)\ .
$$
For $(L_{\infty},q_{\infty}\in \MSV(D_{\infty},L))$, let
$\chi^{\infty}_{\pm}$ be the corresponding weights. These depend on $D_{\infty}$ and $m$. Then $$
\TheMoc\circ
\wp(L_{\infty},q_{\infty})=\MF_{\st}(L_{\infty},\chi_+^{\infty})\oplus
\MF_{\st}(L_{\infty}^{-1}\otimes \MO(- D_{\infty}),\chi_-^{\infty})\ .
$$
Therefore, we obtain
\begin{equation} \label{eq_limit_reducible_fiber}
	\begin{split}
		&\lim_{i\to\infty}\TheMoc\circ \wp(L_i,q_i)\\
		=&\TheMoc\circ \wp(L_{\infty},q_{\infty})\otimes
        (\MF_{\st}(\MO,\chi_+-\chi_+^{\infty})\oplus \MF_{\st}(\MO(
        D_{\infty}- D),\chi_--\chi_-^{\infty}))\ .
	\end{split}	
\end{equation}

\begin{proposition}
	\label{prop_genus3_algebraic_Mochizuki}
	When $g\geq 3$, there exists a sequence $(\ME_i,\vp_i)\in \MM_q$ of stable Higgs bundles with stable limit $(\ME_{\infty},\vp_{\infty})=\lim_{i\to \infty}(\ME_i,\vp_i)$ such that 
    $$\lim_{i\to \infty}\TheMoc(\ME_i,\vp_i)\neq
    \TheMoc(\ME_{\infty},\vp_{\infty})\ .$$
\end{proposition}
\begin{proof}
	Choose $ D=\Lam$ and $d_+=-(g-1)$ with $L_i=L\in
    \Pic^{d_+}(\Sigma)$, and study the degenerate behavior for a family
    $q_i\in \MV(\Lam,L)$ which converges to $q_{\infty}\in \MSV(
    D_{\infty},L)$. Here, $D_{\infty}$ satisfies $D_{\infty}\leq  D$ and $\deg(D_{\infty})=\deg(D)-1$.
    As $q_i$ lies in the top stratum, we can always find such a family. 
    Take $(\ME_i,\vp_i)=\wp(L_i,q_i)$ and $(\ME_{\infty},\vp_{\infty})=\wp(L,q_{\infty})$. 
    When $g\geq 3$, we have $-\deg(D_{\infty})<d_+\leq -\frac12\deg(D_{\infty})$,
    which implies $(\ME_{\infty},\vp_{\infty})$ is a stable Higgs bundle. 
	
	Write $D=\sum_p\ell_p$. As $(\ME_i,\vp_i)$ is nonexotic, 
    the weights will be $\chi_+(p)=\chi_-(p)=\ell_p/2$. However, as
    $\deg(D_{\infty})\neq 2d_+$, $(\ME_{\infty},\vp_{\infty})$ is exotic.
    By Proposition \ref{prop_choiceofconstant_c}, if we write
    $\chi_{\pm}^{\infty}(p)$ for the weight functions with corresponding constant $c$, then $c>0$.
    Therefore, for $p\neq p_0$, we have
    $\chi^{\infty}_+(p)=(m_p+1)c+m_p/2>m_p/2=\chi_+(p)$.
	By \eqref{eq_limit_reducible_fiber}, $\lim_{i\to
    \infty}\TheMoc(\ME_i,\vp_i)\neq \TheMoc(\ME_{\infty},\vp_{\infty})$.
\end{proof}

When $g=2$, the stratification is simpler, and we have the following.
\begin{proposition}
\label{prop_genus2_algebraic_Mochizuki}
When $g=2$, the following holds:
\begin{itemize}
\item [(i)] Suppose $\Lam=p_1+p_2$ for $p_1\neq p_2$, then $\TheMoc|_{\MM_q^{\sta}}$ is continuous. 
Moreover, there exists a sequence of stable Higgs bundles $(\ME_i,\vp_i)\in \MM_q$ where the
limit $(\ME_{\infty},\vp_{\infty})=\lim_{i\to \infty}(\ME_i,\vp_i)$ is semistable, and  $\gamma(0)$ is also
semistable and  furthermore
$$
\lim_{i\to \infty}\TheMoc(\ME_i,\vp_i)\neq \TheMoc(\ME_{\infty},\vp_{\infty})\ .
$$
\item [(ii)] Suppose $\Lam=2p$. Then $\TheMoc|_{\MM_q^{\sta}}$ is continuous.
\end{itemize}
\end{proposition}
\begin{proof}
For (i), suppose $\Lam=p_1+p_2$, then by Example \ref{ex_genus_two_stable.}, we have $\MM_q^{\sta}=\wp(\MSV(\Lam,-1))$.
By Proposition \ref{prop_reducible_continuous_on_strata}, $\TheMoc_q|_{\MM_q^{\sta}}$ is continuous. 
However, for semistable elements other strata must be taken into consideration. Take $L\in \Pic^{-1}(\Sigma)$ and 
$q_i\in \MV(\Lam,L)$ such that $q_i$ convergence to $q_{\infty}\in \MV(p_1,L)$. We define $(\ME_i,\vp_i)=\wp(L,q_i)$ and
$(\ME_{\infty},\vp_{\infty})=\wp(L,q_{\infty})$. For each $i$, 
$$
\TheMoc(\ME_i,\vp_i)=\MF_{\st}(L,(\tfrac{1}{2},\tfrac{1}{2}))\oplus \MF_{\st}(L^{-1}(-\Lam),(\tfrac{1}{2},\tfrac{1}{2})).
$$
Moreover, we have 
$$
\TheMoc(\ME_{\infty},\vp_{\infty})=\MF_{\st}(L(D),(0,0))\oplus \MF_{\st}(L^{-1}(-D),(0,0))\neq \lim_{i\to \infty}\TheMoc(\ME_i,\vp_i).
$$
	
For (ii), by Example \ref{example_genustwodegreestratification},
$\wp(\MSV(D_2,-1))=\MM_q^{\sta}$ and by Proposition \ref{prop_reducible_continuous_on_strata}, $\TheMoc_q|_{\MM_q^{\sta}}$
is continuous. We now consider the behavior of the filtered bundle when crossing the divisors. 
\end{proof}

\subsection{The analytic Mochizuki map and  limiting configurations}
In this subsection, we construct the analytic Mochizuki map for the Hitchin fiber with a reducible spectral curve. We 
also introduce the convergence theorem of Mochizuki as stated in \cite{Mochizukiasymptotic} and 
examine the discontinuous behavior of the analytic Mochizuki map. 

For $(\ME,\vp)\in\MM_q$, we can express the abelianization as $(\ME_0,\vp_0)=(L_+\oplus L_-,\begin{pmatrix}
\om & 0\\ 0 & -\om
\end{pmatrix})$, 
thus $\TheMoc(\ME,\vp)=\MF_{\st}(L_+,\chi_+)\oplus \ML_-(L_-,\chi_-)\in \MF_2(\Sigma)$. Via the nonabelian Hodge correspondence for filtered
bundles, we obtain two Hermitian metrics $h^{\Lim}_{\pm}$ with corresponding Chern connections $A_{h^{\Lim}_{\pm}}$. These metrics 
satisfy the following proposition.

\begin{proposition}[{\cite[Lemma 4.4]{mochizuki2003asymptoticdifferent}}]
	\label{prop_limitingconfconstruct}
	The metrics $h_{\pm}^{\Lim}$ over $L_{\pm}$ satisfy
	\begin{itemize}
		\item [i)] $F_{A_{h^{\Lim}_{\pm}}}=0$ and $h_+^{\Lim}h_-^{\Lim}=1$,
		\item [ii)] for every $p\in\Sigma$, there exists an open neighborhood $(U,z)$ with $P=\{z=0\}$ such that $|z|^{-2\chi_p(c_0)}h_{+}^{\Lim}$ and $|z|^{2\chi_p(c_0)+2l_P}h_-^{\Lim}$ extends smoothly to $L_{\pm}|_U$.
	\end{itemize}
\end{proposition}

Now, $H^{\Lim}:=h_+^{\Lim} \oplus h_-^{\Lim}$ is a metric on $\ME_0$ which induces a metric on $(\ME,\vp)|_{\Sigma\setminus Z}$ 
because $(\ME,\vp)|_{\Sigma\setminus Z} \cong (\ME_0,\vp_0)|_{\Sigma\setminus Z}$. Let $(A^{\Lim},\phi^{\Lim})$ be the Chern connection 
defined by $(\ME,\vp,H^{\Lim})$ over $\Sigma\setminus Z$. Then $(A^{\Lim},\phi^{\Lim})$ is a limiting configuration that satisfies 
the decoupled Hitchin equations \eqref{eq_decoupled_Hitchin_equation}. The analytic Mochizuki map $\UpMoc$ is defined as:
\begin{equation} \label{eq_analytic_moc_reducible}
	\UpMoc: \MM_q \lra \MMHLC, \quad \UpMoc(\ME,\vp) = (A^{\Lim},\phi^{\Lim}).
\end{equation}
Note that $H^{\Lim}$ is not unique: for any constant $c$, the metric $c h_+^{\Lim} \oplus c^{-1} h_-^{\Lim}$ defines the same 
Chern connection as $H^{\Lim}$. In any case, the map $\UpMoc$ is well-defined.

Suppose $(\ME,\vp)$ is an S-equivalence class of a semistable Higgs bundle. Let $H_t$ be the harmonic metric for $(\ME,t\vp)$. For 
each constant $C>0$, define $\mu_C$ to be the automorphism of $L_+ \oplus L_-$ given by $\mu_C = C\id_{L_+} \oplus C^{-1} \id_{L_-}$. 
As $\ME \cong L_+ \oplus L_-$ on $\Sigma\setminus Z$, $\mu_C^{\st} H_t$ can be regarded as a metric on $\ME|_{\Sigma\setminus Z}$. 
Take any point $x \in \Sigma\setminus Z$ and a frame $e_x$ of $L_+|x$, and define:
\begin{equation*}
C(x,t):=\biggl(\frac{h_{L_+}^{\LC}(e_x,e_x)}{H_t(e_x,e_x)}\biggr)^{1/2}\ .
\end{equation*}

Writing $\na_t + t\phi_t$ as the corresponding flat connection of $(\ME,t\vp)$ under the nonabelian Hodge correspondence, then 
\begin{theorem}[{\cite{Mochizukiasymptotic}}]
\label{thm_moc_convergence_reduciblecase}
On any compact subset $K$ of $\Sigma\setminus Z$,
$\mu^{\st}_{C(x,t)}H_t$ converges smoothly to $H^{\LC}$. Additionally, there exist $t$-independent constants $C_k$ and $C'_k$ such that
$$
|(\na_t,\phi_t)-\UpMoc(\ME,\vp)|_{\MC^k}\leq C_ke^{-C'_kd}.
$$
\end{theorem}

Propositions \ref{prop_genus3_algebraic_Mochizuki} and \ref{prop_genus2_algebraic_Mochizuki} now give
\begin{theorem}
\label{thm_analytic_moc_reducible_fiber}
When $g\geq 3$, $\UpMoc|_{\MM_q^{\sta}}$ is discontinuous, and when $g=2$, $\UpMoc|_{\MM_q^{\sta}}$ is continuous. 
\end{theorem}

\section{The Compactified Kobayashi-Hitchin map}
\label{sec_compactified_Kobayashi_Hitchin_map}
In this section, we define a compactified version of the Kobayashi-Hitchin map and prove the main theorem of our paper. 
The Kobayashi-Hitchin map $\Xi$ is a homeomorphism between the Dolbeault moduli space $\MMD$ and the Hitchin 
moduli space $\MMH$. We wish to extend this to a map $\BXi$ from the compactified Dolbeault moduli space $\BMMD$ 
to the compactification $\BMMH\subset \MMH\cup\MMHLC$ of the Hitchin moduli space, and to study the properties
of this extended map. 
\subsection{The compactified Kobayashi-Hitchin map}
	\label{subsec_constructionlimitngconfiguration}
We first summarize the results obtained above. By the construction in Section \ref{sec_the_algebraic_and_analytic_compactifications}, 
there is an identification $\pa\BMMD\cong (\MMD\setminus \MH^{-1}(0))/\CS$. Moreover, through \eqref{eq_analytic_moc_irreducible} and
\eqref{eq_analytic_moc_reducible}, we have constructed the analytic Mochizuki map $\UpMoc: \MMD\setminus \MH^{-1}(0)\to \MMHLC$. 
Writing $(A^{\LC},\phi^{\LC}=\vp+\vp^{\da_{\LC}})=\UpMoc(\ME,\vp)$, then for $w\in \CS$, we have
$$
\UpMoc(\ME,w\vp)=(A^{\LC},\phi^{\LC}=w\vp+\bar{w}\vp^{\da_{\LC}})\ .
$$ 
Hence $\UpMoc$ descends to a map $\pa \BXi$ between $\CS$ orbits.
\begin{equation*}
\begin{split}
\pa\BXi:\pa\BMMD=(\MMD\setminus \MH^{-1}(0))/\CS\lra \MMH^{\LC}/\CS\ ,
\end{split}
\end{equation*}
Together with the initial Kobayashi-Hitchin map $\Xi:\MMD\to \MMH$, we obtain
\begin{equation*}
\begin{split}
\BXi: \BMMD=\MMD\cup\pa\BMMD\lra \MMH\cup \MMHLC/\CS. 
\end{split}
\end{equation*}

Theorems \ref{thm_moc_convergence_irreducible_case} and \ref{thm_moc_convergence_reduciblecase} show that for a Higgs bundle $(\ME,\vp)\in
\MMD\setminus \MH^{-1}(0)$ and real $t$, $\lim_{t\to \infty}\Xi(\ME,t\vp)=\pa\BXi[(\ME,\vp)/\CS]$. Thus the image of $\BXi$ lies in 
$\BMMH$, the closure of $\MMH$ in $\MMH\cup\MMD\setminus \MH^{-1}(0)$.
There are natural extensions $\BMH_{\Dol}:\BMMD\to \PMB$ and $\BMH_{\Hit}:\BMMH\to \PMB$ such that 
$\BMH_{\Hit}\circ \BXi=\BMH_{\Dol}$.
 
In summary, there are commutative diagrams
\begin{equation*}
	\begin{tikzcd} \MMD \arrow[r, "\Xi"] \arrow[d, hook] & \MMH \arrow[d,
    hook] \\ \BMMD \arrow[r, "\BXi"] & \BMMH \end{tikzcd}\ ,\ \begin{tikzcd}
		\BMMD \arrow[rd, "\BMH_{\Dol}",swap ] \arrow[r, "\BXi"] & \BMMH \arrow[d, "\BMH_{\Hit}"] \\
		&  \PMB
	\end{tikzcd}.
\end{equation*}

We now turn to the analysis of some properties of the compactified Kobayashi-Hitchin map. Define 
$$
\PMB^{\reg}=\{[(q,w)]\in \PMB\mid  q\neq 0\mathrm{\;has\;simple\;zeros}\}\ .
$$
This represents the compactified space of quadratic differentials with simple zeros. Let $\PMB^{\sing}=\PMB\setminus \PMB^{\reg}$. 
Additionally, define the open sets $\BMMD^{\reg}=\BMH_{\Dol}^{-1}(\PMB^{\reg})$ and $\BMMH^{\reg}=\BMH_{\Hit}^{-1}(\PMB^{\reg})$ 
as the collections of elements with regular spectral curves. Now set $\BMMD^{\sing}=\BMH_{\Dol}^{-1}(\PMB^{\sing})$ and
$\BMMH^{\sing}=\BMH_{\Hit}^{-1}(\PMB^{\sing})$ to be the sets of singular
fibers. We can write $\BXi=\BXi^{\reg}\cup\BXi^{\sing}$, 
where 
$$
\BXi^{\reg}:\BMMD^{\reg}\lra \BMMH^{\reg},\quad \BXi^{\sing}:\BMMD^{\sing}\lra \BMMH^{\sing}\ .
$$

\begin{proposition}
The map $\BXi^{\reg}:\BMMD^{\reg}\to \BMMH^{\reg}$ is bijective, whereas $\BXi^{\sing}:\BMMD^{\sing}\to \BMMH^{\sing}$ is neither surjective nor injective.
\end{proposition}
\begin{proof}
The bijectivity of $\BXi^{\reg}$ is established by Theorem
    \ref{thm_simple_zero_bijective}. The non-surjectivity and non-injectivity 
of $\BXi^{\sing}$ follow from Theorem \ref{thm_analytic_moc_irreducible_fiber} and Theorem \ref{thm_analytic_moc_reducible_fiber}.
\end{proof}

\subsection{Continuity properties of the compactified Kobayashi-Hitchin map}
In this subsection, we prove that the continuity of the compactified Kobayashi-Hitchin map is fully determined by the
continuity of the analytic Mochizuki map.

Let $(\ME_{i},t_i\vp_{i})$ be a sequence of Higgs bundles with real numbers
$t_i\to+\infty$, $\det(\vp_i)=q_i$, $Z_i=q_i^{-1}(0)$, and $\|q_i\|_{L^2}=1$. We denote $\xi_i=[(\ME_i,t_i\vp_i)]\in \MMD$. 
By the compactness of $\BMMD$, after passing to a subsequence, we may
assume there is $\xi_{\infty}\in \pa\BMMD$ such that $\lim_{i\to\infty}\xi_i=\xi_{\infty}$. Since $\pa\BMMD\cong (\MMD\setminus \MH^{-1}(0))/\CS$, we can select a representative $(\ME_{\infty},\vp_{\infty})$ of $\xi_{\infty}$. By Lemma \ref{prop_convergence_Dol_space}, we have that $(\ME_i,\vp_i)$ converges to $(\ME_{\infty},\vp_{\infty})$ in $\MMD$, and $q_i$ converges to $q_{\infty}$. We write $Z_{\infty}=q_{\infty}^{-1}(0).$

By Proposition \ref{prop_general_convergence_solutions}, $\lim_{i\to
\infty}\BXi(\ME_i,t_i\vp_i)$ exists. The following result 
establishes the continuity of this map with respect to the analytic Mochizuki map $\UpMoc$.
\begin{proposition} \label{prop_topology_KHmap}
Under the previous convention, $\lim_{i\to \infty}\BXi(\xi_i)=\BXi(\xi_{\infty})$ if and only if $\lim_{i\to \infty}\UpMoc(\ME_i,\vp_i)=\UpMoc(\ME_{\infty},\vp_{\infty})$. In other words, $\BXi$ is continuous at $\xi_\infty$ if and only if $\UpMoc$ is continuous at $\xi_\infty$.
\end{proposition}
\begin{proof}
Set 
    \begin{align*}
        \BXi(\ME_i,t_i\vp_i)&=\Xi(\ME_i,t_i\vp_i)=A_i+t_i\phi_i\ ,\\
        \UpMoc(\ME_i,\vp_i)&=(A_i^{\LC},\phi_i^{\LC})\ ,\\
        \UpMoc(\ME_{\infty},\vp_{\infty})&=(A_{\infty}^{\LC},\phi_{\infty}^{\LC})\ .
    \end{align*}
    By Proposition \ref{prop_general_convergence_solutions}, 
there exists a limiting configuration $(A_{\infty},\phi_{\infty}):=\lim_{i\to \infty}(A_i,\phi_i)$ over $\Sigma\setminus Z_{\infty}$. 
Let $K$ be any compact set in $\Sigma\setminus Z_{\infty}$. Note that by the convergence assumption, there exists $i_0$ such that $Z_i\cap
K=\emptyset$ for all $i\geq i_0$. Moreover, from Theorems \ref{thm_moc_convergence_irreducible_case} and \ref{thm_moc_convergence_reduciblecase}, 
for $d_i=\min_K|q_i|$, there exist $t$-independent constants $C,C'>0$ such that (up to gauge transformations)
\begin{equation*}
|(A_i,\phi_i)-(A_i^{\LC},\phi_i^{\LC})|_{\MC^k(K)}\leq Ce^{-C't_id_i}.
\end{equation*}
The convergence is uniform and exponential for fixed $K$. Therefore, over $K$, the size of
$|(A_{\infty},\phi_{\infty})-(A_{\infty}^{\LC},\phi_{\infty}^{\LC})|_{\MC^k(K)}$ is the same as the size of
$|(A_{i}^{\LC},\phi_{i}^{\LC})-(A_{\infty}^{\LC},\phi_{\infty}^{\LC})|_{\MC^k(K)}$.  This  proves the Proposition.
\end{proof}

\subsubsection{Continuity along rays}
We now investigate the behavior of the compactified Kobayashi-Hitchin map restricted to a singular fiber. Specifically, 
fix $0\neq q\in H^0(K^2)$, and denote by $[q]$ the $\mathbb{C}^\ast$-orbit of $q\times{1}$ in the compactified Hitchin base 
$\PMB$. Define $\BMMDq:=\BMH^{-1}_{\Dol}([q])$, $\BMMHq:=\BMH^{-1}_{\Hit}([q])$. Then the restriction of $\BXi$ on $\BMMDq$ 
defines a map $\BXi_q:\BMMDq\to \BMMHq$. 

\begin{theorem}
Let $q$ be an irreducible quadratic differential. 
\begin{itemize}
\item [(i)]	If $q$ contains only zeroes of odd order, then $\BXi_q$ is continuous.
\item [(ii)] If $q$ contains a zero of even order, let $\MM_q=\cup_D\MM_{q,D}$ be the stratification defined earlier. 
Then for each $D\neq 0$, there exists an integer $n_D>0$ such that for any Higgs bundle $(\MF,\psi)\in \MM_{q,D}$, 
there exist $n_D$ sequences of Higgs bundles $(\ME_i^k,\vp_i^k)$ with $k=1,\ldots, n_D$ such that
\begin{itemize}
\item [(a)] $\lim_{i\to \infty}(\ME_i^k,\vp_i^k)=(\MF,\psi)$ for $k=1,\ldots,n_D$,
\item [(b)] if $\lim_{i\to \infty}t_i=\infty$, and we write 
$$
\eta^k:=\lim_{i\to \infty}\BXi_q(\ME_i^k,\vp_i^k),\;\xi:=\lim_{i\to \infty}\BXi_q(\MF,t_i\psi),
$$ 
then $\xi,\eta^1,\ldots,\eta^{n_D}$ are $n_D+1$ different limiting configurations.
\end{itemize}
\end{itemize}
\end{theorem}
\begin{proof}
This follows from Theorem \ref{thm_analytic_moc_irreducible_fiber} and Proposition \ref{prop_topology_KHmap}.
\end{proof}

\begin{theorem}
Suppose $q$ is reducible, and let $\BMMDq^{\sta}$ be the stable locus of $\BMMDq$. Then the restriction map
$\TheMoc_q|_{\BMMDq^{\sta}}:\BMMDq^{\sta}\to \BMMHq$ is discontinuous when $g\geq 3$ and continuous when $g=2$.
\end{theorem}
\begin{proof}
This follows from Propositions \ref{prop_genus3_algebraic_Mochizuki} and \ref{prop_topology_KHmap}.
\end{proof}

\subsubsection{Varying fiber} 
With the conventions above, suppose $(\ME_i,\vp_i)$ converges to $(\ME_{\infty},\vp_{\infty})$ with 
$q_{\infty}$ having  only simple zeros, and $\xi_i=(\ME_i,t_i\vp_i)$ converges to $\xi_\infty$ on $\BMMD$. Since the condition
of having only simple zeros is  open, the $q_i$ also have simple zeros for $i$ sufficiently large. 
\begin{proposition}[{cf.\ \cite[Thm.\ 2.12]{ott2020higgs}}] 
\label{prop_homeomorphic_part_one}
Suppose $q_{\infty}$ has only simple zeros. Then, $\displaystyle\lim_{i\to\infty}\BXi(\xi_i)=\BXi(\xi_\infty)$. In particular, the map 
$\BXi^{\reg}:\BMMD^{\reg}\to \BMMH^{\reg}$ is continuous.
\end{proposition}
\begin{proof}
Let $S_i$ denote the spectral curve of $(\ME_i,\vp_i)$ with branching locus $Z_i$. Also, let $L_i:=\chi_{BNR}^{-1}(\ME_i,\vp_i)$ be the 
eigenline bundles. By the construction in Section \ref{sec_irreducible_singular_fiber}, we have $\UpMoc(\xi_i)=\MF_{\st}(L_i,\chi_i)$, where
$\chi_i=-\frac12\chi_{Z_i}$. Our assumption implies that $\MF_{\st}(L_i,\chi_i)$ converges to $\MF_{\st}(L_{\infty},\chi_{\infty})$ in the sense 
of Definition \ref{def_convergence_parabolic_bundles}. Thus, by Theorem \ref{thm_convergence_family_harmonic_bundles}, we obtain 
the convergence of the limiting configurations: $\lim_{i\to \infty}\UpMoc(\xi_i)=\UpMoc(\xi_{\infty})$. The claim follows from 
Proposition \ref{prop_topology_KHmap}.
\end{proof}

\begin{theorem}
The map $\BXi^{\reg}:\BMMD^{\reg}\to \BMMH^{\reg}$ is a homeomorphism.
\end{theorem}
\begin{proof}
By Theorem \ref{thm_simple_zero_bijective}, $\BXi^{\reg}$ is a bijection.
    Moreover, by Proposition \ref{prop_homeomorphic_part_one},
    $\BXi^{\reg}$ is continuous. Finally, that $(\BXi^{\reg})^{-1}$ is continuous follows directly from the construction in \cite{mazzeo2019asymptotic}.
\end{proof}

\appendix
\section{Classification  of rank 1 torsion modules for $A_n$ singularities}
\label{appendixA}
In this appendix, we review the classification result for rank 1 torsion free modules at $A_n$ singularities, as given in \cite{Greuel1985}. 
We compute the integer invariants defined in Subsection \ref{subsection_torsion_free_sheaf_integers}.

Let $S$ be the spectral curve of an $\SLC$ Higgs bundle, and $x$ a singular point with local defining equation  given by 
$r^2-s^{n+1}=0$; this is an $A_n$ singularity. Let $p:\tS\to S$ be the normalization, where $p^{-1}(x)=\{\tx_+,\tx_-\}$ if $n$ is odd 
and $p^{-1}(x)={\tx}$ if $n$ is even. We use $R$ to denote the completion
of the local ring $\MO_x$,  $K$ its 
field of fractions, and $\tR$ its normalization.

\subsection{$A_{2n}$ singularity}
The local equation is $r^2-s^{2n+1}=0$. The normalization induces a map between  coordinate rings, and we can write
\begin{equation*}
	\psi:\mbC[r,s]/(r^2-s^{2n+1})\lra \mbC[t],\quad \psi(f(r,s))=f(t^{2n},t^2),
\end{equation*}
where $\tR=\mbC[[t]]$ and $R=\mbC[[t^2,t^{2n+1}]]\subset \tR$.
According to \cite[Anh.\ (1.1)]{Greuel1985}, any rank 1 torsion free $R$-module can be written as
\begin{equation*}
	M_k = R + R \cdot t^k \subset \tR, \quad k=1,3,\ldots,2n+1.
\end{equation*}
Here, $M_k$ is a fractional ideal that satisfies $R\subset M_k\subset \tR$, with $M_1=\tR$ and $M_{2n+1}=R$. We may express any $f \in M_k$ as $f = \sum_{i=0}^{\frac{k-1}{2}} f_{2i}t^{2i}+\sum_{i\geq k}f_it^i$, where $f_i\in \mbC$.

We are interested in the integers $\ell_x := \dim_{\mbC}(M_k/R)$, $a_{\tx}
:= \dim_{\mbC}(\tR/C(M_k))$ and $b_x=\dim_{\mbC}(\Tor(M_k\otimes_R \tR))$.
Thus, as a $\mbC$-vector space, $M_k/R$ is generated by $t^k, t^{k+2},
\ldots, t^{2n-1}$, implying that $\ell_x=\frac{2n+1-k}{2}$.

The conductor of $M_k$ is given by $C(M_k) = \{u\in K\mid u\cdot \tR\subset
M_k\}$. By the expression of $M_k$ and a straightforward computation, we
have $C(M_k) = (t^{k-1})$, where $(t^{k-1})$ is the ideal in $\tR$
generated by $t^{k-1}$. Thus, $1,t,\ldots,t^{k-2}$ will form a basis for
$\tR/C(M_k)$, and we have $a_{\tx}=k-1$. Therefore, we have
$a_{\tx}=2n-2\ell_x$.

For $i=0, 1,\ldots,\frac{2n-1-k}{2}$, we define $s_i=t^{k+2i}\otimes_R
1-1\otimes t^{k+2i}\in M_k\otimes_R\tR$. As $k$ is odd, $t^{2n+1-k-2i}\in
R$ and $t^{2n+1-k-2i}s_i=t^{2n+1}\otimes_R 1-1\otimes_R t^{2n+1}=0$, where
the last equality is becasue $t^{2n+1}\in R$. Moreover, $\{s_1,\ldots,
s_{\frac{2n-1-k}{2}}\}$ form a basis of $\Tor(M_k\otimes_R\tR)$, thus
$b_x=\frac{2n+1-k}{2}=\ell_x$.

\subsection{$A_{2n-1}$ singularity}
The local equation is $r^2-s^{2n}=0$. The normalization induces a map
between the coordinate rings:
\begin{equation*}
	\psi:\mbC[r,s]/(r^2-s^{2n})\lra \mbC[t]\oplus \mbC[t],\quad
    \psi(f(r,s))=(f(t^n,t),f(-t^n,t))\ ,
\end{equation*}
where $\tR=\mbC[[t]]\oplus \mbC[[t]]$ and $R=\mbC[[(t,t),(t^n,-t^n)]]\cong \mbC[[(t,t),(t^n,0)]]$.
By \cite[Anh.\ (2.1)]{Greuel1985}, any rank 1 torsion free $R$-module can be written as:
\begin{equation*}
	M_k=R+R\cdot (t^k,0)\subset \tR,\quad k=0,1,\ldots,n.
\end{equation*}
Then, $M_k$ is also a fractional ideal with $R\subset M_k\subset \tR$. Moreover, $M_n=R$, and $M_0=\tR$. 

As $p^{-1}(x)=\{\tx_+,\tx_-\}$, $\tR$ contains two maximal ideals, 
$\mfm_+=((t,1))$, $\mfm_-=((1,t))$.
For $f\in M_k$, we can express $f$ as:
\begin{equation*}
	f=\sum_{i=0}^{k-1}f_{ii}(t^i,t^i)+\sum_{l\geq 0}f_{l0}(t^{k+l},0)+f_{0l}(0,t^{k+l}),
\end{equation*}
where $f_{ij}\in \mbC$. Therefore, $\ell_x=\dim_{\mbC}(M_k/R)=n-k$.
Moreover, using the expression, we can compute the conductor
$C(M_k)=((t^k,1))\cdot((1,t^k))$, which implies $a_{\tx_{\pm}}=k$.
Similarly, for $i=k,\ldots,n-1$, we define
$s_i=(t^i,0)\otimes_R(1,1)-(1,1)\otimes_R (t^i,0)$, then $(t,t)^{n-i}\cdot
s_i=0$ and $\{s_k,\ldots, s_{n-1}\}$ will be a basis for $\Tor(M_k\otimes_R
\tR)$ and $b_x=\ell_x$. 

In summary, we have the following:
\begin{proposition}
	\label{prop_appendix_computation}
	For the integers defined above, we have:
	\begin{itemize}
		\item [(i)] for the $A_{2n}$ singularity, we have
            $a_{\tx}=2n-2\ell_x$ and $b_x=\ell_x$,
		\item [(ii)] for the $A_{2n-1}$ singularity, we have
            $a_{\tx_{\pm}}=n-\ell_x$ and $b_x=\ell_x$.
	\end{itemize}
\end{proposition}

	\bibliographystyle{plain}
	\bibliography{references}
\end{document}